\documentclass[amssymb,amsfonts,refcheck,11pt,verbatim,righttag]{amsart}
\usepackage{amsfonts}
\usepackage{amssymb,amscd}

 \numberwithin{equation}{section}
\theoremstyle{plain}

\newtheorem{thm}{Theorem}[section]
\newtheorem{lem}[thm]{Lemma}
\newtheorem{pro}[thm]{Proposition}
\newtheorem{cor}[thm]{Corollary}

\newtheorem{de}[thm]{Definition}
\newtheorem{rem}[thm]{Remark}

\def\R {{\Bbb R}}
\def\N {{\Bbb N}}
\def\Z {{\Bbb Z}}
\def\M {{\mathcal M}}
\def\P {{\mathcal P}}
\def\Q {{\mathcal Q}}
\def\J {{\mathcal J}}
\def\F {{\mathcal F}}
\def\I {{\mathcal I}}
\def\A {{\mathcal A}}
\def\B {{\mathcal B}}
\def\C {{\mathcal C}}
\def\D {{\mathcal D}}
\def\W {{\mathcal W}}

\def\E{{\bf E}}

\def\sm{[\!]}

\def\disp{\displaystyle}

\begin{document}
\baselineskip 15pt
%To appear in Comm. Pure Appl. Math.
%\bigskip

\title{Dimension  theory of iterated function systems}

\date{}

\author{ De-Jun FENG}
\address{
Department of Mathematics,
The Chinese University of Hong Kong,
Shatin,  Hong Kong}
\email{djfeng@math.cuhk.edu.hk} 
\author{Huyi Hu}
\address{Department of Mathematics, Michigan State University, East Lansing, MI 48824, USA
} \email{hu@math.msu.edu}

\maketitle

\begin{abstract}

Let $\{S_i\}_{i=1}^\ell$ be an iterated function system (IFS)  on $\R^d$ with attractor $K$.
Let $(\Sigma,\sigma)$ denote the one-sided full shift over the alphabet $\{1,\ldots, \ell\}$.  We define the projection entropy function $h_\pi$ on the space of invariant measures on $\Sigma$ associated with the coding map $\pi:\; \Sigma\to K$, and develop some basic ergodic properties about it.  This concept turns out to be crucial in the study of dimensional properties of  invariant measures on $K$. We show that  for any conformal IFS (resp., the direct product of finitely many conformal IFS), without any separation  condition, the projection of an ergodic measure under $\pi$ is always exactly dimensional and, its Hausdorff dimension can be represented as the ratio of its projection entropy to its Lyapunov exponent (resp., the linear combination of projection entropies associated with several coding maps). Furthermore, for any conformal IFS and certain affine IFS,  we prove a variational principle between the Hausdorff  dimension of the attractors and that of projections  of ergodic measures.

\end{abstract}

% With AMS-LaTeX, \maketitle follows the abstract
\maketitle

%%      ---------------------------------------------------------------------
%%      ------------------- TABLE OF CONTENTS (OPTIONAL) --------------------
%%      ---------------------------------------------------------------------

%% ***** IF YOUR PAPER IS OVER 40 PAGES AND YOU WISH TO HAVE A TABLE
%% ***** OF CONTENTS, PLEASE UNCOMMENT THE FOLLOWING LINE

 \tableofcontents

%%      ---------------------------------------------------------------------
%%      ---------------------------- BODY OF PAPER --------------------------
%%      ---------------------------------------------------------------------

%%      Please input or insert the body of your paper here.

\section{Introduction}

Let $\{S_i: X\to X\}_{i=1}^\ell$ be a family of contractive maps
on a nonempty closed set $X\subset \R^d$.
Following Barnsley \cite{Bar-book}, we say that $\Phi=\{S_i\}_{i=1}^\ell$
is an {\it iterated function system} (IFS) on $X$.
Hutchinson \cite{Hut81} showed that there is a unique nonempty compact
set $K\subset X$,  called the {\it attractor} of $\{S_i\}_{i=1}^\ell$,
such that $K=\bigcup_{i=1}^\ell S_i(K)$.
A probability measure $\mu$ on $\R^d$ is said to be {\it exactly dimensional}
if there is a constant $C$ such that the {\it local dimension}
$$d(\mu, x)=\lim_{r\to 0}\frac{\log \mu (B(x,r))}{\log r}$$
exists and equals $C$ for $\mu$-a.e. $x\in \R^d$, where $B(x,r)$ denotes the closed ball of radius $r$ centered at $x$. It was shown by Young \cite{You82} that in such case,  the
Hausdorff dimension of $\mu$ is equal to $C$.
(See also \cite{Fan94,Mat95,Pes-book}.)

The motivation of the paper is to study the Hausdorff dimension of
an invariant measure $\mu$ (see Section~\ref{S1} for precise meaning) for conformal and affine IFS with
overlaps. To deal with overlaps, we regard such a system as the
image of a natural projection $\pi$ from the one-sided full shift space over $\ell$ symbols. Hence we obtain a dynamical system. We
introduce a notion {\it projection entropy}, which plays the similar
role as the classical entropy for IFS satisfying the open set
condition, and it becomes the classical entropy if the projection is
finite to one. The concept of projection entropy turns out to be
crucial in the study of dimensional properties of invariant measures
on attractors of either conformal IFS with overlaps or affine IFS.

We develop some basic properties about projection entropy
(Theorem~\ref{thm-1.0}, \ref{thm-1.0'}).
We prove that for conformal IFS with overlaps, every ergodic measure
$\mu$ is exactly dimensional and $\disp d(\mu, x)$ is equal to the
projection entropy divided by the Lyapunov exponent
(Theorem~\ref{thm-1.2}). Furthermore, if $\Phi$ is a direct product of
conformal IFS (see Definition~\ref{de-1.6} for precise meaning),
then for every ergodic measure on $K$ the local dimension can be
expressed by a Ledrappier-Young type formula in terms of projection
entropies and Lyapunov exponents (Theorem~\ref{thm-1.3}). We also
prove variational results about Hausdorff dimension for conformal
IFS and certain affine IFS (Theorem~\ref{thm-1.5} and
~\ref{thm-1.6}), which says that the Hausdorff dimension of the
attractor $K$ is equal to the supremum of Hausdorff dimension of
$\mu$ taking over all ergodic measures. The results we obtain cover
some interesting cases such as $S_i(x)=\mbox{diag}(\rho_1,\ldots,
\rho_d)x+a_i$, where $i=1,\ldots,\ell$ and $\rho_i^{-1}$ are {\it
Pisot  or Salem numbers} and $a_i\in \Z^d$.

The problem whether a given measure is exactly dimensional,
and whether the Hausdorff dimension of an attractor
can be assumed or approximated by that of an invariant measure
have been well studied in the literature for
$C^{1+\alpha}$ conformal IFS which satisfy the open set condition
(cf. \cite{Bed91, GeHa89, Pat97}). It is well known that in such
case, any ergodic measure $\mu$ is exactly dimensional with the
Hausdorff dimension given by the classic entropy divided by the Lyapunov exponent.
 Furthermore there is
a unique invariant measure $\mu$ with $\dim_H(\mu)=\dim_H(K)$, the
Hausdorff dimension of $K$. However the problems become much
complicated and  intractable without the assumption of the open set
condition. Partial results have only been obtained for conformal
IFS that satisfy the {\it finite type condition} (see \cite{NgWa01}
for the definition). In that case, a Bernoulli measure is exactly
dimensional and its Hausdorff dimension may be expressed as the
upper Lyapunov exponent of certain random matrices (see e.g.
\cite{Fen03, Fen05, Lal98, LePo94, LaNg99}), and furthermore the
Hausdorff dimension of $K$ can be computed (see e.g.
\cite{Lal97,RaWe98,NgWa01}).

There are  some results  for certain special non-overlapping affine
IFS. McMullen \cite{McM84} and Bedford \cite{Bed84} independently
computed the Hausdorff dimension and the box dimension of the
attractor of the following planar affine IFS
\begin{equation*}\label{1.2}
    S_i(x) = \begin{bmatrix}n^{-1} &0\\0&k^{-1}\end{bmatrix} x
            +\begin{bmatrix} {a_i}/{n}\\ {b_i}/{k},  \end{bmatrix},\qquad i=1,\ldots, \ell,
\end{equation*}
where all $a_i,b_i$ are integers, $0 \leq a_i <n$ and $0 \leq b_i
<k$. Furthermore they showed that there is a Bernoulli measure of
full Hausdorff dimension. This result was extended by Kenyon and
Peres \cite{KePe96} to higher dimensional self-affine Sierpinski
sponges, for which ergodic measures are proved to be exactly
dimensional with Hausdorff dimension given by  a Ledrappier-Young
type formula. Another  extension of McMullen and Bedford's result to
a boarder class of planar affine IFS $\{S_i\}_{i=1}^\ell$ was given
by Gatzouras and Lalley \cite{GaLa92}, in which $S_i$ map the unit
square $(0,1)^2$ into disjoint rectangles with sides parallel to the
axes (where the longer sides are parallel to the $x$-axis,
furthermore once projected onto the $x$-axis these rectangles are
either identical, or disjoint). Further extensions were given
recently by Bara\'{n}ski \cite{Bar07}, Feng and Wang \cite{FeWa05}, Luzia \cite{Luz06} and
Olivier \cite{Oli08}.  For other related results,
see e.g. \cite{PrUr89, Led92,KePe96a, GaPe97, Hu96,HuLa95, Fen05,
 Shm06, BaMe07,KaSh08}.

Along another direction, in \cite{Fal88} Falconer gave a variational
formula for the Hausdorff and box
dimensions for ``almost all'' self-affine sets under some
assumptions. This formula remains true under some weaker conditions
\cite{Sol98,JPS07}. K\"{a}enm\"{a}ki \cite{Kae04} proved that for
``almost all'' self-affine sets there exists an ergodic measure $m$
so that $m\circ \pi^{-1}$ is of  full Hausdorff dimension.

Our arguments use ergodic theory and Rohlin's theory about conditional
measures.  The proofs of Theorem \ref{thm-1.1} and Theorem \ref{thm-1.3}
are based on some ideas from the work of Ledrappier and Young \cite{LeYo85}
and techniques in analyzing the densities of conditional
measures associated with overlapping IFS.

So far we have restricted ourselves on the study of finite contractive IFS. However we point out that part of our results remain valid for certain non-contractive infinite IFS (see Section~\ref{S-10} for details).

The paper is organized  as follows.
The main results are given in Section~\ref{S1}.
In Section~\ref{S2}, we  prove some density results
about conditional measures.
In Section~\ref{S3}, we investigate the properties of
projection entropy and prove Theorem \ref{thm-1.0} and \ref{thm-1.0'}.
In Section~\ref{S4}, we give some local geometric
properties of a $C^1$ IFS. In Section~\ref{S5},
we prove a  generalized  version of  Theorem \ref{thm-1.1},
which is based on a key proposition (Proposition \ref{pro-5.1})
about the densities of conditional measures. In Section~\ref{S6},
we prove Theorem~\ref{thm-1.3}
and \ref{thm-1.4}. In Section~\ref{S7}, we prove Theorem \ref{thm-1.5}
and in Section~\ref{S8}, we prove Theorem \ref{thm-1.6}. In Section ~\ref{S-10} we give a remark regarding certain non-contractive infinite IFS.

\section{Statement of main results}\label{S1}

Let $\{S_i\}_{i=1}^\ell$ be an IFS on a closed set $X\subset \R^d$. Denote by $K$ its attractor.
Let $\Sigma=\{1,\ldots,\ell\}^\N$  associated with the left shift $\sigma$ (cf. \cite{Bow75}). Let $\M_\sigma(\Sigma)$ denote the space of $\sigma$-invariant measures on $\Sigma$, endowed with the weak-star topology.  Let $\pi: \Sigma\to K$ be the canonical projection defined by
\begin{equation}
\label{e-1.1} \{\pi(x)\}=\bigcap_{n=1}^\infty S_{x_1}\circ S_{x_2}\circ\cdots
\circ S_{x_n}(K),\qquad\mbox{where } x=(x_i)_{i=1}^\infty.
\end{equation}
A measure $\mu$ on $K$ is called {\it invariant} (resp., {\it ergodic})
for the IFS if there is an invariant (resp. ergodic) measure $\nu$ on $\Sigma$
such that $\mu=\nu\circ \pi^{-1}$.

Let $(\Omega,\F,\nu)$ be a probability space.
For a sub-$\sigma$-algebra $\A$ of $\F$ and $f\in L^1(\Omega,\F,\nu)$,
we denote by ${\bf E}_\nu(f|\A)$ the the {\it conditional expectation of
$f$ given $\A$}.
For countable $\F$-measurable partition $\xi$ of $\Omega$,
we denote by ${\bf  I}_\nu(\xi|\A)$ the {\it conditional
information of $\xi$ given $\A$}, which is given by the formula
\begin{equation}
\label{e-1.2}
 {\bf I}_\nu(\xi|\A)=-\sum_{A\in \xi}\chi_A\log \E_\nu(\chi_A |\A),
\end{equation}
where $\chi_A$ denotes the characteristic function on $A$.
The {\it conditional entropy of $\xi$ given $\A$}, written
$H_\nu(\xi|\A)$, is defined by the formula
\begin{equation*}
 H_\nu(\xi|\A)=\int {\bf I}_\nu(\xi|\A)\; d\nu.
\end{equation*}
(See e.g. \cite{Par-book} for more details.)
The above information and entropy are unconditional when
$\A={{\mathcal N}}$, the trivial $\sigma$-algebra consisting of sets
of measure zero and one, and in this case we write
\begin{equation*}
 {\bf I}_\nu(\xi|{{\mathcal
N}})=:{\bf I}_\nu(\xi)\quad\mbox{and}\quad H_\nu(\xi|{{\mathcal
N}})=:H_\nu(\xi).
\end{equation*}

Now we consider the space $(\Sigma, \B(\Sigma),m)$, where $\B(\Sigma)$
is the Borel $\sigma$-algebra on $\Sigma$ and $m\in \M_\sigma(\Sigma)$.
Let $\P$ denote the Borel partition
\begin{equation}
\label{e-1P}
\P=\{[j]: 1\leq j\leq \ell\}
\end{equation}
 of $\Sigma$, where
$[j]=\{(x_i)_{i=1}^\infty\in \Sigma:\; x_1=j\}$.
Let  ${\mathcal I}$ denote the $\sigma$-algebra
\begin{equation*}
\I=\{B\in \B(\Sigma):\; \sigma^{-1}B=B\}.
\end{equation*}
For convenience, we  use
$\gamma$ to denote the Borel $\sigma$-algebra $\B(\R^d)$ on $\R^d$.

\begin{de} \label{de-1.1}  {\rm For any $m\in \M_\sigma(\Sigma)$, we call
\begin{equation*}
h_\pi(\sigma,m):=H_m(\P|\sigma^{-1}\pi^{-1}\gamma)-H_m(\P|\pi^{-1}\gamma)
\end{equation*}
the {\it projection entropy of $m$ under $\pi$
w.r.t.  $\{S_i\}_{i=1}^\ell$}, and we call
$$
h_\pi(\sigma,m,x):={\bf E}_m\left(f\big|{\mathcal I}\right)(x)
$$
the {\it local projection entropy of $m$ at $x$ under $\pi$
w.r.t.  $\{S_i\}_{i=1}^\ell$}, where $f$ denotes the function ${\bf I}_m(\P|\sigma^{-1}\pi^{-1}\gamma)-{\bf I}_m(\P|\pi^{-1}\gamma)$.}
\end{de}

It is clear that $h_\pi(\sigma,m)=\int h_\pi(\sigma,m,x)\;dm(x)$.
Our first result is the following theorem.

\begin{thm}\label{thm-1.0} Let $\{S_i\}_{i=1}^\ell$ be an IFS. Then
\begin{itemize}
 \item[(i)] For any $m\in \M_\sigma(\Sigma)$, we have $0\leq h_\pi(\sigma,m)\leq h(\sigma,m)$, where $h(\sigma,m)$ denotes the classical measure-theoretic entropy of $m$ associated with $\sigma$.
\item[(ii)] The map $m\mapsto h_\pi(\sigma,m)$ is affine on  $\M_\sigma(\Sigma)$.  Furthermore if
$m=\int \nu \;d{\Bbb P}(\nu)$ is the ergodic decomposition of $m$, we have
$$
h_\pi(\sigma,m)=\int h_\pi(\sigma,\nu) \;d {\Bbb P} (\nu).
$$
\item[(iii)] For any $m\in \M_\sigma(\Sigma)$, we have
$$
\lim_{n\to \infty}\frac{1}{n} {\bf I}_m(\P_0^{n-1}|\pi^{-1}\gamma)(x)=h(\sigma,m,x)-h_\pi(\sigma,m,x)
$$
for $m$-a.e. $x\in \Sigma$, where $h(\sigma,m,x)$ denotes the local entropy of $m$ at $x$, that is,
$h(\sigma,m,x)={\bf I}_m(\P|\sigma^{-1}\B(\Sigma))(x)$.
\end{itemize}
\end{thm}

 Part (iii) of the theorem is an analogue of  the classical relativized  Shannon-McMillan-Breiman theorem (see, e.g. \cite[Lemma 4.1]{Bog92}).
However, we should notice that the sub $\sigma$-algebra $\pi^{-1}\gamma$ in our consideration is not $\sigma$-invariant in general
(see Remark \ref{rem-3.14}).

Part (iii) also implies that if the map $\pi\colon \Sigma \to K$ is finite-to-one, then $$h_\pi(\sigma,m)=h(\sigma,m)$$ for any $m\in \M_\sigma(\Sigma)$. In Section~\ref{S3},  we will present a sufficient and necessary condition for the equality (see Corollary \ref{cor-3.17}). However for general
overlapping IFS, the projection entropy can be strictly less than the classical entropy.

In our next theorem, we  give a geometric characterization of the projection entropy for certain affine IFS, which
will be used later in the proof of our variational results about the Hausdorff and box dimensions of self-affine sets.

\begin{thm}
\label{thm-1.0'}
Assume that $\Phi=\{S_i\}_{i=1}^\ell$ is an IFS on $\R^d$ of the form
$$S_i(x)=Ax+c_i \qquad (i=1,\ldots, \ell),$$
where $A$ is a $d\times d$ non-singular contractive real matrix and $c_i\in \R^d$. Let $K$ denote the attractor of $\Phi$. Let $\Q$ denote the partition $\{[0,1)^d+\alpha:\; \alpha\in \Z^d\}$ of $\R^d$. For $n=0, 1,\ldots$, and $x\in \R^d$,  we set $
\Q_n=\{A^nQ:\; Q\in \Q\}$. Then
\begin{itemize}
\item[(i)] For any  $m\in \M_\sigma(\Sigma)$, we have
\begin{equation*}
h_\pi(\sigma,m)=\lim_{n\to\infty}\frac{H_m(\pi^{-1}\Q_n)}{n}.
\end{equation*}
\item[(ii)] Moreover,
\begin{equation*}
\lim_{n\to \infty}\frac{\log \#\{Q\in \Q:\; A^nQ\cap K\neq \emptyset\}}{n}
=\sup\{h_\pi(\sigma,m):\; m\in \M_\sigma(\Sigma)\}.
\end{equation*}
\end{itemize}
\end{thm}

\medskip

To give the applications of  projection entropy in dimension theory of IFS,  we need some more notation and definitions.
\begin{de}
\label{e-j7.7}
{\rm $\{S_i: X\to X\}_{i=1}^\ell $
is called a {\it $C^1$ IFS} on a compact set
$X\subset \R^d$ if each $S_i$ extends to a contracting
$C^1$-diffeomorphism $S_i: U\to S_i(U)\subset U$ on an open set
$U\supset X$.
}
\end{de}

For any $d\times d$ real  matrix $M$, we use $\|M\|$ to denote the usual norm of $M$, and  $\sm M\sm$ the smallest singular value of $M$, i.e.,
\begin{equation}
\label{e-M1}
\begin{split}
\|M\| &=\max\{|Mv|: \;
v\in \R^d, |v|=1\} \quad \mbox{and}\\
 \sm M\sm &=\min\{|Mv|: \;
v\in \R^d, |v|=1\}.
\end{split}
\end{equation}

\begin{de}
\label{de-1.2}
{\rm
Let $\{S_i\}_{i=1}^\ell$ be a $C^1$ IFS. For $x=(x_j)_{j=1}^\infty\in \Sigma$, the {\it upper and lower Lyapunov exponents of $\{S_i\}_{i=1}^\ell$ at $x$} are defined respectively by
\begin{align*}
&\overline{\lambda}(x)=-\liminf_{n\to \infty} \frac{1}{n}\log \sm S^\prime_{x_1\ldots x_n}(\pi\sigma^n x)\sm,\quad \\
&\underline{\lambda}(x)=-\limsup_{n\to \infty} \frac{1}{n}\log \| S^\prime_{x_1\ldots x_n}(\pi\sigma^n x)\|,
\end{align*}
where $S^\prime_{x_1\ldots x_n}(\pi\sigma^n x)$ denotes the differential of $S_{x_1\ldots x_n}:=
S_{x_1}\circ S_{x_2}\circ \ldots \circ S_{x_n}$ at $\pi\sigma^n x$. When $\overline{\lambda}(x)=\underline{\lambda}(x)$, the common value, denoted as  $\lambda(x)$,  is called the {\it Lyapunov exponent of $\{S_i\}_{i=1}^\ell$ at $x$}.
}
\end{de}

It is easy to check that both $\overline{\lambda}$ and $\underline{\lambda}$ are positive-valued  $\sigma$-invariant functions on $\Sigma$ (i.e. $\overline{\lambda}=\overline{\lambda}\circ \sigma$ and $\underline{\lambda}=\underline{\lambda}\circ \sigma$). Recall that for a probability measure $\mu$ on $\R^d$,  the {\it local upper and lower dimensions} are defined respectively by
$$\overline{d}(\mu, x)=\limsup_{r\to 0}\frac{\log \mu (B(x,r))}{\log r},\quad \underline{d}(\mu, x)=\liminf_{r\to 0}\frac{\log \mu (B(x,r))}{\log r},$$
 where $B(x,r)$ denotes the closed ball of radius $r$ centered at $x$. If $\overline{d}(\mu, x)=\underline{d}(\mu, x)$, the common value is denoted as $d(\mu,x)$ and is called the {\it local dimension} of $m$ at $x$.

  The following theorem  gives an estimate of local dimensions of invariant measures on the attractor of  an arbitrary $C^1$ IFS, without any separation condition.
\begin{thm}
\label{thm-1.1}Let $\{S_i\}_{i=1}^\ell$ be a $C^1$ IFS with attractor $K$.  Then for $\mu=m\circ
\pi^{-1}$, where $m\in \M_\sigma(\Sigma)$, we have the following estimates:
\begin{equation*}
    \overline{d}(\mu,\pi x)\leq \frac{h_{\pi}(\sigma, m, x)}
  {\underline{\lambda}(x)} \quad\mbox{  and }\quad
  \underline{d}(\mu,\pi x)\geq \frac{h_\pi(\sigma,m,x)}
  {\overline{\lambda}(x)} \quad\mbox {for $m$-a.e. $x\in \Sigma$,}\end{equation*}
where $h_\pi(\sigma,m,x)$ denotes the local projection entropy of $m$ at $x$ under $\pi$ (see Definition \ref{de-1.1}). In particular, if $m$ is ergodic, we have
\begin{equation*}
    \frac{h_\pi(\sigma,m)}{\int \overline{\lambda} \;dm}\leq  \underline{d}(\mu,z)\leq \overline{d}(\mu,z)\leq \frac{h_\pi(\sigma,m)}{\int\underline{\lambda} \;dm} \quad\mbox {for $\mu$-a.e. $z\in K$}.
\end{equation*}
\end{thm}
\medskip

\begin{de}
\label{de-1.3'} {\rm Let $\{S_i\}_{i=1}^\ell$ be a $C^1$ IFS and  $m\in \M_\sigma(\Sigma)$. We say that  $\{S_i\}_{i=1}^\ell$ is  {\it $m$-conformal}  if
$\lambda(x)$ exists (i.e.,
$\overline{\lambda}(x)=\underline{\lambda}(x)$)  for $m$-a.e. $x\in
\Sigma$.}
\end{de}

As a direct application of Theorem \ref{thm-1.1}, we have
\begin{thm}
\label{thm-1.2}Assume that  $\{S_i\}_{i=1}^\ell$ is  $m$-conformal
for some $m\in \M_\sigma(\Sigma)$. Let $\mu=m\circ\pi^{-1}$. Then we
have
\begin{equation}
\label{e-c1}
    d(\mu,\pi x)= \frac{h_\pi(\sigma,m,x)}{\lambda(x)}\quad\mbox {for $m$-a.e. $x\in \Sigma$}.
  \end{equation}
In particular, if $m$ is ergodic, we have
\begin{equation}
\label{e-c2}
   d(\mu,z)=\frac{h_\pi(\sigma,m)}{\int {\lambda} \;dm}\quad\mbox {for $\mu$-a.e. $z\in K$}.
\end{equation}
\end{thm}

%There are some sufficient conditions to guarantee the $m$-conformality.
Recall that $S: U\to S(U)$ is a conformal map
if $S^\prime (x): \R^d\to \R^d$ satisfies $\|S^\prime(x)\|\neq 0$
and $ |S^\prime(x) y|=\|S^\prime (x)\| |y|$ for all $x\in U$ and
$y\in \R^d$.

\begin{de}
\label{de-1.3} {\rm A $C^1$ IFS $\{S_i\}_{i=1}^\ell$ is said to be
{\it weakly conformal} if
 $$\frac{1}{n}(\log \sm S^\prime_{x_1\ldots x_n}(\pi\sigma^n x)\sm- \log \| S^\prime_{x_1\ldots x_n}(\pi\sigma^n x)\|)
 $$
 converges to $0$ uniformly on $\Sigma$ as $n$ tends to $\infty$. We say that
$\{S_i\}_{i=1}^\ell$ is {\it conformal} if each $S_i$
extends to a conformal map $S_i: U\to S_i(U)\subset U$ on an open
set $U\supset K$, where $K$ is the attractor of
$\{S_i\}_{i=1}^\ell$. }
\end{de}

By  definition,  a conformal IFS is always weakly conformal. Furthermore, a weakly conformal IFS is
$m$-conformal for each $m\in \M_\sigma(\Sigma)$ (see Proposition \ref{pro-4.6}(ii)). There are
some natural examples of  weakly conformal IFS which are not
conformal. For instance, let $S_i(x)= A_i x+a_i$ ($i=1,\ldots,\ell$)
such that, for each $i$, $A_i$ is a contracting linear map with
eigenvalues equal to each other in modulus, and $A_i A_j=A_jA_i$ for
different $i,j$. Then such an IFS is always weakly conformal but not
necessary to be conformal. The first conclusion follows from the
asymptotic behavior
$$
\lim_{n\to \infty}
\sm A_i^n \sm ^{1/n}=\lim_{n\to \infty}
\| A_i^n \| ^{1/n}=\rho(A_i)\qquad  (i=1,\ldots, \ell),
$$
where $\rho(A_i)$ denotes the spectral radius of $A_i$ (cf. \cite{Yam67}).

Theorem \ref{thm-1.2} verifies the existence of local dimensions for
invariant measures on the attractor of an arbitrary weakly
conformal IFS attractors, without any separation assumption.   We point out that the exact dimensionality for
overlapping  self-similar measures was first claimed  by
Ledrappier,   nevertheless  no proof has been
written out (cf. \cite[p. 1619]{PeSo00}). We remark that this property was also conjectured later
by Fan, Lau and Rao in \cite{FLR02}.

We can extend the above result to  a  class of  non-conformal IFS.

\begin{de}\label{de-1.6}{\rm
Assume for $j=1,\ldots, k$, $\Phi_j:=\{S_{i,j}\}_{i=1}^\ell$ is a $C^1$ IFS defined on a compact set $X_j\subset \R^{q_j}$. Let $\Phi:=\{S_i\}_{i=1}^\ell$ be the IFS on  $X_1\times\cdots \times X_k\subset \R^{q_1}\times \cdots \times \R^{q_k}$ given by
$$
S_i(z_1,\ldots, z_k)=\left(S_{i,1}(z_1),\ldots, S_{i,k}(z_k)\right) \ \  (i=1,\ldots, \ell,\; j=1,\ldots,k,\; z_j\in X_j).
$$
We say that  $\Phi$ is the {\it direct product} of $\Phi_1,\ldots, \Phi_k$,  and write
$\Phi=\Phi_1\times\cdots\times \Phi_k$.}
\end{de}

\begin{thm}
\label{thm-1.3} Let $\Phi=\{S_i\}_{i=1}^\ell$ be  the direct product
of $k$ $C^1$ IFS $\Phi_1,\ldots, \Phi_k$. Let $\mu=m\circ
\pi^{-1}$, where $m\in \M_\sigma(\Sigma)$. Assume that $\Phi_1,\ldots,\Phi_k$ are $m$-conformal. Then \begin{itemize}
 \item[(i)] $d(\mu,z)$ exists for $\mu$-a.e. $z$. \item[(ii)]
   Assume furthermore that $m$ is ergodic. Then $\mu$ is exactly dimensional. Let $\tau$ be a permutation on $\{1,\ldots,k\}$ such that
$$\lambda_{\tau(1)}\leq \lambda_{\tau(2)}\leq \cdots\leq \lambda_{\tau(k)},$$ where
$\lambda_j=\int \lambda_j(x) \;dm(x)$, and $\lambda_j(x)$ denotes the Lyapunov exponent of $\Phi_j$ at $x\in \Sigma$.
Then   we have
\begin{equation}
\label{e-affine}
 d(\mu,
z)=\frac{h_{\pi_1}(\sigma,m)}{\lambda_{\tau(1)}}+\sum_{j=2}^k\frac{h_{\pi_j}(\sigma,m)-
h_{\pi_{j-1}}(\sigma,m)}{\lambda_{\tau(j)}} \quad \mbox{ for $\mu$-a.e. z},
\end{equation}
where  $\pi_j$ denotes the canonical projection w.r.t. the IFS
$\Phi_{\tau(1)}\times\cdots \times \Phi_{\tau(j)}$, and
$h_{\pi_j}(\sigma,m)$
 denotes the projection entropy of $m$ under
$\pi_j$.
\end{itemize}
\end{thm}

We mention that fractals satisfy the conditions of the theorem
include many interesting examples such as those studied in
\cite{McM84, Bed84, GaLa92, KePe96}, etc.

As an application of Theorem \ref{thm-1.3}, we have

\begin{thm}
\label{thm-1.4} Let $\{S_i\}_{i=1}^\ell$ be an IFS on $\R^d$ of the form
$$S_i(x)=A_ix+a_i, \qquad i=1,\ldots, \ell,$$ such that each $A_i$ is a
nonsingular contracting linear map on $\R^d$, and $A_i A_j=A_jA_i$
for any $1\leq i,j\leq \ell$. Then for any ergodic measure $m$ on
$\Sigma$, $\mu=m\circ \pi^{-1}$ is exactly dimensional.
\end{thm}

Indeed, under the assumption of Theorem \ref{thm-1.4}, we can show that
 there is a nonsingular linear transformation  $T$ on $\R^d$  such that the IFS
 $\{T\circ S_i\circ T^{-1}\}_{i=1}^\ell$ is the direct product of some weakly conformal IFS.
Hence we can apply Theorem \ref{thm-1.3} in this situation.
%However,  we do not know whether the exact dimensionality remains true without the commutativity  assumption on $A_i$'s.

We remark that  formula (\ref{e-affine}) provides an analogue of that for the Hausdorff dimension of $C^{1+\alpha}$ hyperbolic measures along the unstable (resp. stable) manifold established by
Ledrappier and Young \cite{LeYo85}.

The problem of the existence of local dimensions has also a long
history in smooth dynamical systems. In \cite{You82}, Young proved
that an ergodic hyperbolic measure invariant under a $C^{1+\alpha}$
surface diffeomorphism is always  exact dimensional. For a measures
$\mu$ in high-dimensional $C^{1+\alpha}$ systems,
Ledrappier and Young \cite{LeYo85} proved the existence of $\delta^u$ and
$\delta^s$, the local dimensions along stable and unstable local
manifolds, respectively, and the upper local dimension of $\mu$ is
bounded by the sum of $\delta^u$, $\delta^s$, and the multiplicity of
$0$ as an exponent. Eckmann and Ruelle \cite{EcRu85} indicated that
it is unknown whether the local dimension of $\mu$ is the sum of
$\delta^u$ and $\delta^s$ if $\mu$ is a hyperbolic measure. Then the
problem was referred as Eckmann-Ruelle conjecture, and affirmatively
answered by Barreira, Pesin and Schmeling in \cite{BPS99} seventeen
years later. Some partial dimensional results were obtained for
measures invariant under hyperbolic endomorphism \cite{Sch98a, Sch98b}. Recently, Qian and Xie \cite{QiXi08}
proved the exact dimensionality of ergodic measures invariant under $C^2$ expanding endomorphism on smooth Riemannian manifolds.

In the remaining part of this section, we present  some  variational results about the Hausdorff dimension and  the box dimension of
attractors of IFS and that of invariant measures. First we
consider conformal IFS.

\begin{thm}
\label{thm-1.5}
Let $K$ be the attractor of a  weakly conformal IFS $\{S_i\}_{i=1}^\ell$.
Then we have
\begin{eqnarray}  \mbox{}\qquad \dim_H K&=&\dim_BK\label{e-1.11}\\
&=&\sup
\left\{\dim_H \mu:\; \mu=m\circ \pi^{-1}, \;m\in
\M_\sigma(\Sigma), \; m \mbox{ is ergodic}\right\}\label{e-1.12}\\
&=&\max
\left\{\dim_H \mu:\; \mu=m\circ \pi^{-1}, \;m\in
\M_\sigma(\Sigma)\right\}\nonumber\\
&=&\sup\left\{\frac{h_\pi(\sigma,m)}{\int \lambda \;dm}: \;m\in
\M_\sigma(\Sigma)\right\},\label{e-1.14} \end{eqnarray}
where $\dim_BK$ denotes the box dimension of $K$.
\end{thm}

 Equality (\ref{e-1.11}) was first proved by Falconer
\cite{Fal89} for  $C^{1+\alpha}$ conformal IFS.
It is not known whether the supremum in (\ref{e-1.12}) and (\ref{e-1.14}) can be attained in  the general setting of Theorem \ref{thm-1.5}. However, this is true if the IFS $\{S_i\}_{i=1}^\ell$  satisfies an additional separation condition defined as follows.
\begin{de}
\label{de-1.5}
{\rm An IFS $\{S_i\}_{i=1}^\ell$ on a compact set $X\subset \R^d$ is said to satisfy the
{\it asymptotically weak separation condition} (AWSC), if
\begin{equation*}
\lim_{n\to \infty}\frac{1}{n}\log t_n=0,
\end{equation*}
where $t_n$ is given by
\begin{equation}
\label{e-1.16}
t_n=\sup_{x\in \R^d} \#\{S_u:\; u\in \{1,\ldots,\ell\}^n,\; x\in S_u(K)\},
\end{equation}
here  $K$ is the attractor of $\{S_i\}_{i=1}^\ell$.}
\end{de}

The above definition was first introduced in \cite{Fen07} under a
slightly different setting. For example, if $1/\rho$ is a Pisot or
Salem number, then the IFS
$\{\rho x+a_i\}_{i=1}^\ell$ on $\R$, with $a_i\in \Z$, satisfies the
AWSC (see  Proposition 5.3 and Remark 5.5 in \cite{Fen07}).
Recall that  a real number $\beta>1$ is said to be a {\it Salem number} if it is an
algebraic integer whose algebraic conjugates all have modulus not greater
than $1$, with at least one of which on the unit circle. Whilst $\beta>1$
is called a {\it Pisot number} if it is an
algebraic integer whose algebraic conjugates all have modulus less than $1$. For instance, the largest root ($\approx 1.72208$) of $x^4-x^3-x^2-x+1$ is a Salem number, and the golden ratio $(\sqrt{5}+1)/2$ is a Pisot number.
One is referred to \cite{Sal63} for more examples and  properties about Pisot and Salem numbers.
Under
the AWSC assumption, we can show that  the projection entropy map
$m\mapsto h_\pi(\sigma,m)$ is upper semi-continuous on
$\M_\sigma(\Sigma)$ (see Proposition \ref{pro-3.9}) and, as a
consequence, the supremum (\ref{e-1.12}) and (\ref{e-1.14}) can be
attained at  ergodic measures (see Remark \ref{rem-4.4}).

Next we consider some affine IFS.
\begin{thm}
\label{thm-1.6}
Let $\Phi=\{S_i\}_{i=1}^\ell$ be an affine  IFS on $\R^d$ given by
$$S_i(x_1,\ldots,x_d)=(\rho_1 x_1,\cdots, \rho_dx_d)+(a_{i,1},\ldots, a_{i,d}),$$
where $\rho_1>\rho_2>\cdots >\rho_d>0$ and $a_{i,j}\in \R$.
Let $K$ denote the attractor of $\Phi$, and write $\lambda_j=\log(1/\rho_j)$ for $j=1,\ldots, d$
 and $\lambda_{d+1}=\infty$. View $\Phi$ as the direct product of $\Phi_1,\ldots, \Phi_d$,
 where $\Phi_j=\{S_{i,j}(x_j)=\rho_jx_j+ a_{i,j}\}_{i=1}^\ell$.
 Assume that $\Phi_1\times\cdots\times \Phi_j$ satisfies the AWSC for  $j=1,\ldots,d$.  Then we have
\begin{eqnarray*}
\dim_HK&=&\max\left\{\dim_H \mu: \;\mu=m\circ \pi^{-1}, \; m \mbox{ is ergodic}\right\}\\
&=&\max\left\{\sum_{j=1}^{d}\left(\frac{1}{\lambda_j}-\frac{1}{\lambda_{j+1}}\right)h_{\pi_j}(\sigma,m)
:\; m \mbox{ is ergodic}\right\},
\end{eqnarray*}
 where  $\pi_j$ is the canonical projection w.r.t. the IFS $\Phi_1\times\cdots \times \Phi_j$. Furthermore
\begin{eqnarray*}
\dim_BK=\sum_{j=1}^{d}\left(\frac{1}{\lambda_j}-\frac{1}{\lambda_{j+1}}\right)H_j,
\end{eqnarray*}
where $H_j:=\max\{h_{\pi_j}(\sigma,m):\; m \mbox{ is ergodic}\}$.

\end{thm}

It is direct to check that if $\Phi_j$ satisfies the AWSC for each $1\leq j\leq d$, then so does $\Phi_1\times\cdots\times \Phi_j$. Hence for instance,  the condition of Theorem \ref{thm-1.6} fulfills when  $1/\rho_j$ are Pisot numbers or Salem numbers and $(a_{i,1},\ldots, a_{i,d})\in \Z^d$.
Different from  the earlier works  on the Hausdorff  dimension of deterministic self-affine sets and self-affine measures (see e.g. \cite{McM84, Bed84, KePe96, GaLa92, HuLa95, Bar07, Oli08}), our model in Theorem \ref{thm-1.6} admits certain overlaps. The two variational results in Theorem \ref{thm-1.6} provide some new  insights in the study of overlapping self-affine IFS.  An interesting question is whether  the results of Theorem \ref{thm-1.6} remain true without the AWSC assumption. It is related to the open problem whether a non-conformal repeller carries an ergodic measure of  full dimension (see \cite{GaPe96} for a survey). We remark that in the general case, we do have the following inequality(see Lemma \ref{lem-8.2}): $$\underline{\dim}_B K \geq \sum_{j=1}^{d}\left(\frac{1}{\lambda_j}-\frac{1}{\lambda_{j+1}}\right)\sup\{h_{\pi_j}(\sigma,m):\; m \mbox{ is ergodic}\}.$$
Furthermore Theorem \ref{thm-1.6} can be extended somewhat (see Remark \ref{rem-9.3} and Theorem \ref{thm-9.4}).

\section{Density results about conditional measures}\label{S2}

We prove some density results about conditional measures in this section.
To begin with, we give a brief introduction to Rohlin's theory of Lebesgue
spaces, measurable partitions and conditional measures.
The reader is referred to \cite{Roh49,Par-book1} for more details.

 A probability space $(X, \B, m)$ is called a {\it Lebesgue space} if it is isomorphic to a probability space which is the union of  $[0,s]$ ($0\leq s\leq 1$)  with Lebesgue measure and a countable number of atoms.
Now let $(X, \B,m)$ be a Lebesgue space. A {\it measurable partition} $\eta$ of $X$ is a partition of $X$ such that, up to a set of measure zero, the quotient
space $X/\eta$ is separated by a countable number of measurable sets $\{B_i\}$. The quotient space $X/\eta$ with its inherit probability space structure, written as $(X_\eta, \B_\eta, m_\eta)$,  is again a Lebesgue space.
Also, any measurable partition $\eta$ determine a sub-$\sigma$-algebra of $\B$,
 denoted by $\widehat{\eta}$, whose elements are unions of elements of $\eta$.
Conversely, any sub-$\sigma$-algebra of $\B'$ of $\B$ is also countably generated, say by $\{B_i'\}$, and therefore all the sets of the form $\cap A_i$,
where $A_i=B_i^\prime$ or its complement, form a measurable partition.
In particular, $\B$ itself is corresponding to a partition into single points.
An important property of Lebesgue space and measurable partitions is the following.

\begin{thm}[Rohlin \cite{Roh49}]
\label{thm-2.1}
 Let $\eta$ be a measurable partition of a Lebesgue space $(X, \B, m)$. Then, for every $x$ in a set of full $m$-measure, there is a probability measure $m^\eta_x$ defined on $\eta(x)$, the element of $\eta$ containing $x$. These measures are uniquely characterized (up to sets of $m$-measure $0$) by the following properties: if $A\subset X$ is a measurable set, then
$x\mapsto m^\eta_x(A)$ is $\widehat{\eta}$-measurable and $m(A)=\int
m^\eta_x(A)d m(x)$. These properties imply that for any $f\in
L^1(X,\B, m)$, $m_x^\eta(f)=\E_m(f|\widehat{\eta})(x)$ for $m$-a.e.\! $x$,
and $m(f)=\int \E_m(f|\widehat{\eta})dm$.
\end{thm}

The family of measures $\{m^\eta_x\}$ in the above theorem is called the {\it canonical system of  conditional measures associated with $\eta$}.

Throughout the remaining part of this section,  we  assume that $(X,\B,m)$ is
a Lebesgue space.  Let $\eta$ be a measurable partition of $X$, and
let $\{m^\eta_x\}$ denote the corresponding canonical system of
conditional measures. Suppose that $\pi: X\to  \R^d$ is a
$\B$-measurable map. Denote $\gamma:=\B(\R^d)$, the
Borel-$\sigma$-algebra on $\R^d$. For $y\in \R^d$, we use $B(y,r)$
to denote the closed ball in $\R^d$ of radius $r$ centered at $y$. Also,  we denote for $x\in X$,
\begin{equation}
\label{e-ball}
B^\pi(x,r)=\pi^{-1} B(\pi x,r).
\end{equation}

\begin{lem}
\label{lem-2.2} Let $A\in \B$.  \begin{itemize}
\item[(i)]
 The map
$x\mapsto m_x^\eta(B^\pi(x,r)\cap A)$ is $\hat{\eta} \vee
\pi^{-1}\gamma$-measurable for each $r>0$, where $\hat{\eta} \vee
\pi^{-1}\gamma$ denotes the smallest  sub-$\sigma$-algebra of $\B$
containing $\hat{\eta}$ and $\pi^{-1}\gamma$.
\item[(ii)]
The following  functions
$$
\liminf_{r\to 0}\frac{ m_x^\eta(B^\pi(x,r)\cap A)}{
m_x^\eta(B^\pi(x,r))},\quad \limsup_{r\to 0}\frac{
m_x^\eta(B^\pi(x,r)\cap A)}{ m_x^\eta(B^\pi(x,r))}
$$
and
$$
\inf_{r>0}\frac{ m_x^\eta(B^\pi(x,r)\cap A)}{
m_x^\eta(B^\pi(x,r))}
$$
are $\hat{\eta} \vee \pi^{-1}\gamma$-measurable, where we interpret
${0}/ {0}=0$.
\end{itemize}
\end{lem}

\begin{proof} We first prove (i). Let $A\in \B$ and $r>0$. For $n\in
\N$, let $\D_n$ denote the collection
$$
\D_n=\{[0,2^{-n})^d+\alpha:\; \alpha\in 2^{-n}\Z^d\}.
$$
 For $y\in \R^d$, denote
$$W_n(y)=\bigcup_{Q\in \D_n:\; Q\cap B(y,r)\neq \emptyset} Q.
$$
 Write $\W_n:=\{W_n(y):\; y\in \R^d\}$. It is clear that $\W_n$ is countable for each $n\in \N$. Furthermore, we have
$W_n(y)\downarrow B(y,r)$ for each $y\in \R^d$ as $n\to \infty$,
that is, $W_{n+1}(y)\subset W_n(y)$ and $\bigcap_{n=1}^\infty
W_n(y)=B(y,r)$. As a consequence, we have $\pi^{-1}W_n(\pi
x)\downarrow B^\pi(x,r)$ and hence
$$
m_x^\eta(B^\pi(x,r)\cap A)=\lim_{n\to \infty}
m^\eta_x(\pi^{-1}W_n(\pi x)\cap A)\qquad (x\in X).$$ Therefore to show that
$x\mapsto m_x^\eta(B^\pi(x,r)\cap A)$ is $\hat{\eta} \vee
\pi^{-1}\gamma$-measurable, it suffices to show that $x\mapsto
m_x^\eta(\pi^{-1}W_n(\pi x)\cap A)$ is $\hat{\eta} \vee
\pi^{-1}\gamma$-measurable for each $n\in \N$.

Fix $n\in \N$. For $F\in \W_n$, let $\Gamma_n(F)=\{x\in X:\; W_n(\pi
x) =F\}$. Then $\Gamma_n(F)\in \pi^{-1}\gamma$. By Theorem
\ref{thm-2.1},  $m_x^\eta(\pi^{-1}F\cap A)$ is an
$\hat{\eta}$-measurable function of $x$ for each $F\in \W_n$. However
$$
m_x^\eta(\pi^{-1}W_n(\pi x)\cap A)=\sum_{F\in
\W_n}\chi_{\Gamma_n(F)} (x) m_x^\eta(\pi^{-1}F\cap A).$$
Hence $m_x^\eta(\pi^{-1}W_n(\pi x) \cap A)$ is $\hat{\eta}
\vee \pi^{-1}\gamma$-measurable, so is $m_x^\eta(B^\pi(x,r)\cap A)$.

To see (ii), note that for $x\in \Sigma$ and $r>0$ satisfying
$m_x^\eta(B^\pi(x,r))>0$, we have
$$
\frac{ m_x^\eta(B^\pi(x,r)\cap A)}{ m_x^\eta(B^\pi(x,r))}=\lim_{q\downarrow r:\; q\in {\Bbb Q}^+}
\frac{ m_x^\eta(B^\pi(x,q)\cap A)}{ m_x^\eta(B^\pi(x,q))}.
$$
Hence for the three limits in (ii), we can restrict $r$ to be
positive rationals. It together with (i) yields the desired
measurability.
\end{proof}

\begin{lem}
\label{lem-2.4} Let $A\in \B$. Then for $m$-a.e.\! $x\in X$,
\begin{equation}
\label{e-2.1} \lim_{r\to 0}\frac{ m_x^\eta(B^\pi(x,r)\cap
A)}{ m_x^\eta(B^\pi(x,r))}=\E_m(\chi_A|\hat{\eta}\vee
\pi^{-1}\gamma)(x).
\end{equation}
\end{lem}
\begin{proof}
Let $\overline{f}(x)$ and $\underline{f}(x)$ be  the  values
obtained by taking the upper and lower limits in the left hand side
of (\ref{e-2.1}). By Lemma \ref{lem-2.2}, both $\overline{f}$ and
$\underline{f}$ are $\hat{\eta} \vee \pi^{-1}\gamma$-measurable. In
the following we only show that
$\overline{f}(x)=\E_m(\chi_A|\hat{\eta}\vee \pi^{-1}\gamma)(x)$ for
$m$-a.e.\! $x$. The proof for
$\underline{f}(x)=\E_m(\chi_A|\hat{\eta}\vee \pi^{-1}\gamma)(x)$ is
similar.

 We first prove that
\begin{equation}
\label{e-2.2} \int_{B\cap \pi^{-1}D} \overline{f} \;dm =\int_{B\cap
\pi^{-1}D} \E_m(\chi_A|\hat{\eta}\vee \pi^{-1}\gamma)\;dm \qquad
(B\in \hat{\eta},\; D\in \gamma).
\end{equation}
By Theorem \ref{thm-2.1}, for any given $C\in \eta$, $m_x^\eta$
($x\in C$) represent the same measure supported on $C$, which we
rewrite as $m_C$. Fix $C\in \eta$. We define measures $\mu_C$ and
$\nu_C$ on $\R^d$ by $\mu_C(E)=m_C(\pi^{-1}E\cap A)$ and
$\nu_C(E)=m_C(\pi^{-1}E)$ for all $E\in \gamma$. It is clear that
$\mu_C\ll\nu_C$.  Define
$$g_C(z)=\limsup_{r\to 0} \frac{\mu_C(B(z,r))}{\nu_C(B(z,r))}\qquad (z\in \R^d).
$$
Then $\overline{f}(x)=g_{\eta(x)}(\pi x)$ for all $x\in \Sigma$.
According to the differentiation theory of measures on $\R^d$ (see,
e.g., \cite[Theorem 2.12]{Mat95}), $g_C=\frac{d\mu_C}{d\nu_C}$
$\nu_C$-a.e. Hence for each $D\in \gamma$, we have $\int_D g_C(z)\
d\nu_C(z)=\mu_C(D)$, i.e., $\int_{\pi^{-1}D } g_C(\pi y)
\;dm_C(y)=\mu_C(D)=m_C(\pi^{-1}D\cap A)$. That is,
\begin{equation}
\label{e-2.3} \int _{\pi^{-1}D} \overline{f} \;
dm_x^\eta=m_x^\eta(\pi^{-1}D\cap A)\qquad (x\in X).
\end{equation}
To see (\ref{e-2.2}), let $B\in \hat{\eta}$. Then
\begin{eqnarray*}
\int_{B\cap \pi^{-1}D} \overline{f} \;dm &=& \int \chi_B\chi_{\pi^{-1}D} \overline{f} \; dm
=\int \E_m\left(\chi_B\chi_{\pi^{-1}D} \overline{f}|\hat{\eta}\right)\; dm\\
&=&\int \chi_B \E_m\left( \chi_{\pi^{-1}D} \overline{f}|\hat{\eta}\right)\; dm\\
&=&\int_B\left(\int_{\pi^{-1}D} \overline{f}\; dm_x^\eta\right) dm(x)\qquad (\mbox{by Theorem \ref{thm-2.1}})\\
&=&\int_Bm_x^\eta(\pi^{-1}D\cap A) dm(x)\qquad (\mbox{by (\ref{e-2.3})})\\
&=&\int \chi_B(x) \E_m\left(\chi_{\pi^{-1}D\cap A}|\hat{\eta}\right)(x) \;dm(x)\qquad (\mbox{by Theorem \ref{thm-2.1}}).
\end{eqnarray*}
Thus we have
\begin{eqnarray*}
\int_{B\cap \pi^{-1}D} \overline{f} \;dm&=&\int \E_m\left(\chi_B\chi_{\pi^{-1}D\cap A}|\hat{\eta}\right)(x) \;dm(x)\\
&=&\int \chi_B\chi_{\pi^{-1}D\cap A}dm=m(B\cap \pi^{-1}D\cap A)\\
&=&\int \E_m(\chi_{B\cap \pi^{-1}D}\chi_A|\hat{\eta}\vee \pi^{-1}\gamma) \;dm\\
&=&\int \chi_{B\cap \pi^{-1}D} \E_m(\chi_A|\hat{\eta}\vee \pi^{-1}\gamma) \;dm\\
&=&\int_{B\cap \pi^{-1}D} \E_m(\chi_A|\hat{\eta}\vee \pi^{-1}\gamma) \;dm.
\end{eqnarray*}
This establishes (\ref{e-2.2}).

Let $R=\overline{f}-{\bf E}_m(\chi_A|\hat{\eta}\vee \pi^{-1}\gamma)$.
Then $R$ is $\hat{\eta}\vee \pi^{-1}\gamma$-measurable and
$$\int_{B\cap \pi^{-1}(D)}R \;dm=0 \qquad (B\in \hat{\eta},\; D\in \pi^{-1}\gamma).$$
Denote $\F=\{B\cap \pi^{-1}(D):\; B\in \hat{\eta},\; D\in \pi^{-1}\gamma\}$ and
let $$\F'=\left\{\bigcup_{i=1}^k F_i: \; k\in \N,\;  F_1,\ldots, F_k \in \F \mbox{ are disjoint}\right\}.$$
It is clear that $\int_FR\; dm=0$ for all $F\in \F'$.
Moreover it is  a routine to check that $\F'$ is an algebra which contains  $\hat{\eta}$ and $\pi^{-1}\gamma$, and hence  $\F'$ generates the  $\sigma$-algebra $\hat{\eta}\vee \pi^{-1}\gamma$.

We claim that
$R=0$ $m$-a.e.  Assume this is not true. Then  there exists $\epsilon>0$ such that the set $\{R>\epsilon\}$, or
$\{R<-\epsilon\}$, has positive $m$-measure. Without loss of generality, we assume that $m\{R>\epsilon\}>0$.
Since $\F'$ is an algebra which generates $\hat{\eta}\vee \pi^{-1}\gamma$, there exists a sequence $F_i\in \F'$ such that $m(F_i\triangle \{R>\epsilon\})$ tends to $0$ as $i\to \infty$ (cf. \cite[Theorem 0.7]{Wal-book}).  We conclude that $\int_{F_i}R\; dm$ tends to $\int_{\{R>\epsilon\}} R\; dm>0$ as $i\to \infty$, which contradicts the fact $\int_{F_i}R\; dm=0$.
\end{proof}

\begin{rem}
\label{rem-2.5}
{\rm
\begin{itemize}
\item[(i)] Letting $\eta={\mathcal N}$ be the trivial partition of $X$ in the above lemma, we obtain
$\displaystyle\lim_{r\to 0}\frac{ m(B^\pi(x,r)\cap
A)}{ m(B^\pi(x,r))}=\E_m(\chi_A|\pi^{-1}\gamma)(x)$  $m$-a.e.
\item[(ii)] In general, ${\bf E}_{m^\eta_x}(\chi_A|\pi^{-1}\gamma)(x)=\E_m(\chi_A|\hat{\eta}\vee
\pi^{-1}\gamma)(x)$ $m$-a.e.,   both of them equal  $\displaystyle\lim_{r\to 0}\frac{ m_x^\eta(B^\pi(x,r)\cap
A)}{ m_x^\eta(B^\pi(x,r))}$ $m$-a.e.   by (i).
\end{itemize}
}
\end{rem}

\begin{pro}\label{pro-2.5}
Let $\xi$ be a countable measurable partition of $X$.
Then for $m$-a.e.\! $x\in X$,
\begin{equation}\label{e-2.4}
\lim_{r\to 0}\log \frac
{m^\eta_x\left(B^\pi(x,r)\cap \xi(x)\right)}
{m^\eta_x\left(B^\pi(x,r)\right)}
=-{\bf I}_m
\left(\xi|\hat{\eta}\vee\pi^{-1}\gamma\right)(x),
\end{equation}
where ${\bf I}_m(\cdot|\cdot)$ denotes the conditional information (see (\ref{e-1.2}) for the definition).
Furthermore, set
\begin{equation}\label{e-2.5}
g(x)=-\inf_{r>0}\log\frac
{m^\eta_x\left(B^\pi(x,r)\cap \xi(x)\right)}
{m^\eta_x\left(B^\pi(x,r)\right)}
\end{equation}
and assume $H_m(\xi)<\infty$. Then $g\geq 0$ and $g\in L^1(X,{\B},m)$.
\end{pro}

\begin{proof}
(\ref{e-2.4}) follows directly from Lemma \ref{lem-2.4} and the following equality
$$
\lim_{r\to 0}\log \frac
{m^\eta_x\left(B^\pi(x,r)\cap \xi(x)\right)}
{m^\eta_x\left(B^\pi(x,r)\right)}=\sum_{A\in \xi}\chi_A(x)\lim_{r\to 0}\log \frac
{m^\eta_x\left(B^\pi(x,r)\cap A\right)}
{m^\eta_x\left(B^\pi(x,r)\right)}.
$$
Now we turn to the proof of (\ref{e-2.5}).
It is clear that $g$ is non-negative.  By Lemma \ref{lem-2.2}, $g$ is measurable. In the following we show that
 $g\in L^1(X,\B,m)$.

Let  $C\in \eta$ and $A\in \xi$ be given. As in the proof of Lemma \ref{lem-2.4}, we define measures $\mu_C$ and $\nu_C$ on $\R^d$ by $\mu_C(E)=m_C(\pi^{-1}E\cap A)$ and $\nu_C(E)=m_C(\pi^{-1}E)$ for all $E\in \gamma$.
By  Theorem 7.4 in \cite{Rud-book}, we have
$$
\mu_C\left\{z\in \R^d:\;
\inf_{r>0}\frac{\mu_C(B(z,r))}{\nu_C(B(z,r))}<\lambda\right\}\leq 3^d \lambda\qquad (\lambda>0).
$$
Hence for any $\lambda>0$,
$$
m_C\left(\
\left\{
x\in X:\;
\inf_{r>0} \frac{m_C\left(B^\pi(x,r)\cap A\right)}
{m_C\left(B^\pi(x,r)\right)}<\lambda
\right\}
\cap A \right)
\leq 3^d \lambda.
$$
Integrating $C$ with respect to $m_\eta$, we obtain
$$
m\left(\
\left\{
x\in X:\;
\inf_{r>0} \frac{m^\eta_x\left(B^\pi(x,r)\cap A\right)}
{m^\eta_x\left(B^\pi(x,r)\right)}<\lambda \right\}
\cap A \right)
\leq 3^d \lambda.
$$
Denote $\displaystyle g^A(x)=\inf_{r>0}\frac
 {m^\eta_x\left(B^\pi(x,r)\cap A\right)}
 {m^\eta_x\left(B^\pi(x,r) \right)}$.
Then the above inequality can be rewritten as
$$m(A\cap \{g^A<\lambda\})\leq 3^d \lambda.$$
Note that by (\ref{e-2.5}),  $g(x)=-\sum_{A\in \xi}\chi_A(x)\log g^A(x)$. Since $g$ is non-negative, we have
\begin{eqnarray*}
\int g\; dm &=&\int_0^\infty m\{g>t\}\;dt
 = \int_0^\infty\sum_{A\in \xi} m(A\cap \{g^A<e^{-t}\})\; dt \\
&\leq& \sum_{A\in \xi} \int_0^\infty  \min \{m(A),
3^d e^{-t}\}\;dt\\
&\leq& \sum_{A\in \xi} \left(-m(A)\log m(A)+ m(A)+ m(A)\log 3^d \right)\\
&=&H_m(\xi)+1+\log 3^d.
\end{eqnarray*}
This finishes the proof of the proposition.
\end{proof}

\begin{rem}
\label{rem-2.6}
{\rm
  Consider the case $X=\Sigma$ and $\xi=\P$, where $\P$ is defined as in (\ref{e-1P}). Suppose that $\{S_i\}_{i=1}^\ell$ is a family of mappings such that $S_i\colon \pi(\Sigma)\to S_i(\pi(\Sigma))\subset \R^d$ is  homeomorphic for each $i$. Then in (\ref{e-2.4}) and (\ref{e-2.5}), we can change the terms $B^\pi(x,r)$ to $\pi^{-1}R_{r,x}(\pi x)$, where
$R_{r,x}(z):=S_{x_1}^{-1} B(S_{x_1} (z),r)$. To see it, fix $i$ and define $\pi'=S_i\circ \pi$. Then we have
\begin{eqnarray*}
\lim_{r\to 0} \frac
{m^\eta_x\left(\pi^{-1}R_{r,x}(\pi x)\cap [i]\right)}
{m^\eta_x\left(\pi^{-1}R_{r,x}(\pi x)\right)}&=&
\lim_{r\to 0} \frac
{m^\eta_x\left(B^{\pi'}(x,r) \cap [i]\right)}
{m^\eta_x\left(B^{\pi'}(x,r)\right)}\\
&=&\E_m(\chi_{[i]}|\hat{\eta}\vee (\pi')^{-1}\gamma)(x).
\end{eqnarray*}
However, $(\pi')^{-1}\gamma=\pi^{-1}\gamma$ due to the assumption on $S_i$. Hence the last term in the above formula  equals
$\E_m(\chi_{[i]}|\hat{\eta}\vee \pi^{-1}\gamma)(x)$. Thus we can replace the terms $B^\pi(x,r)$ by $\pi^{-1}R_{r,x}(\pi x)$ in (\ref{e-2.4}). For the change in (\ref{e-2.5}), we may use a similar argument.
}
\end{rem}

\begin{lem}
\label{lem-2.7}
Let $\pi\colon X\to\R^d$ and $\phi\colon X\to \R^k$ be two $\B$-measurable maps. Let $\eta$ be the partition of $X$ given by
$\eta=\{\pi^{-1}(z)\colon z\in \R^d\}$. Let $A\in \B$ and $t>0$. Then for  $m$-a.e.\! $x\in X$, we have
\begin{equation}
\label{e-2.6}
m^\eta_x(B^\phi(x,t)\cap A)\geq \limsup_{r\to 0}\frac{ m\left(B^\phi(x,t)\cap A \cap B^\pi(x,r)\right)}{m\left(B^\pi(x,r)\right)}
\end{equation}
and
\begin{equation}
\label{e-2.7}
m^\eta_x(U^\phi(x,t)\cap A)\leq \liminf_{r\to 0}\frac{ m\left(U^\phi(x,t)\cap A \cap B^\pi(x,r)\right)}{m\left(B^\pi(x,r)\right)},
\end{equation}
where $B^\phi(x,t):=\phi^{-1}B(\phi x, t)$,  $U^\phi(x,t):=\phi^{-1}U(\phi x, t)$, here $U(z,t)$ denotes the open ball in $\R^d$ centered at $z$ of radius $t$.
\end{lem}
\begin{proof}
Fix $A\in \B$ and $t>0$.
Similar to the proof of Lemma \ref{lem-2.2}, for $n\in
\N$, let $\D_n$ denote the collection
$$
\D_n=\{[0,2^{-n})^k+\alpha:\; \alpha\in 2^{-n}\Z^k\}.
$$
 For $y\in \R^k$, denote
$$W_n(y)=\bigcup_{Q\in \D_n:\; Q\cap B(y,t)\neq \emptyset} Q,\qquad \widehat{W}_n(y)=\bigcup_{Q\in \D_n:\; Q\subset U(y,t)} Q.
$$
 Write $\W_n:=\{W_n(y):\; y\in \R^k\}$ and $\widehat{\W}_n:=\{\widehat{W}_n(y):\; y\in \R^k\}$. It is clear that both $\W_n$ and
  $\widehat{\W}_n$ are  countable for each $n\in \N$. Furthermore, we have
$W_n(y)\downarrow B(y,t)$ and $\widehat{W}_n(y)\uparrow U(y,t)$ for each $y\in \R^k$ as $n\to \infty$. As a consequence, we have $\phi^{-1}W_n(\phi
x)\downarrow B^\phi(x,t)$ and $\phi^{-1}\widehat{W}_n(\phi
x)\uparrow U^\phi(x,t)$ for $x\in X$. Therefore
$$
m_x^\eta(B^\phi(x,t)\cap A)=\lim_{n\to \infty}
m^\eta_x(\phi^{-1}W_n(\phi x)\cap A)$$
and
$$
m_x^\eta(U^\phi(x,t)\cap A)=\lim_{n\to \infty}
m^\eta_x(\phi^{-1}\widehat{W}_n(\phi x)\cap A)$$
for each $x\in X$.

In the following we only prove (\ref{e-2.6}). The proof of (\ref{e-2.7}) is essentially identical.
 For $n\in \N$ and $F\in \W_n$, let
$\Gamma_n(F)=\{x\in X:\; W_n(\phi x)=F\}$.  Then for $m$-a.e.\! $x$ and all $n\in \N$, we have
\begin{eqnarray*}
m_x^\eta(\pi^{-1}W_n(\phi x)\cap A)&=&\sum_{F\in \W_n}\chi_{\Gamma_n(F)} (x) m_x^\eta(\phi^{-1}F \cap A)\\
&=&\sum_{F\in \W_n}\chi_{\Gamma_n(F)} (x) \E_m(\chi_{\phi^{-1}F\cap A}|\hat{\eta})(x)\\
&=&\sum_{F\in \W_n}\chi_{\Gamma_n(F)} (x) \E_m(\chi_{\phi^{-1}F\cap A}|\pi^{-1}\gamma)(x)\\
%\quad (\mbox{ since $\hat{\eta}\circeq \phi^{-1}\gamma$}) \\
&=&\sum_{F\in \W_n}\chi_{\Gamma_n(F)} (x)\lim_{r\to 0}
\frac{ m\left(\phi^{-1}F\cap A \cap B^\pi(x,r)\right)}{m\left(B^\pi(x,r)\right)} \\
&&\qquad
\qquad \qquad (\mbox{ by Lemma \ref{lem-2.4}})\\
&=&\lim_{r\to 0}\frac{ m\left(\phi^{-1}W_n(\phi x)\cap A \cap B^\pi(x,r)\right)}{m\left(B^\pi(x,r)\right)}\\
&\geq &\limsup_{r\to 0}\frac{ m\left(B^\phi(x,t) \cap A \cap B^\pi(x,r)\right)}{m\left(B^\pi(x,r)\right)}.
\end{eqnarray*}
Letting $n\to \infty$, we obtain (\ref{e-2.6}).
\end{proof}

\begin{rem}
\label{rem-2.8}
{\rm
Under the condition of Lemma \ref{lem-2.7}, assume that $$g\colon \pi(X)\to g(\pi(X))\subset \R^d$$ is a homeomorphism. Then
 we may replace the terms $B^\pi(x,r)$ in (\ref{e-2.6}) and (\ref{e-2.7}) by $B^{g\pi}(x,r)$. To see it, let
$\pi'=g\circ\pi$. It is easy to see the partition $\eta$ is just the same as $\{(\pi')^{-1}(z)\colon z\in \R^d\}$.
}
\end{rem}

\begin{pro}
\label{pro-2.9}
Let $T\colon X\to X$ be a measure-preserving transformation on $(X,\B,m)$, and let $\eta$ be a measurable partition of $X$. Suppose that $\pi\colon X\to \R^d$ is a bounded $\B$-measurable function. Then for any $r>0$,
$$
\lim_{n\to \infty}\frac{1}{n}\log m^\eta_{T^nx}\left(B^\pi(T^nx,r)\right)=0 \quad \mbox{ for $m$-a.e.\! $x\in
X$}.
$$
\end{pro}
\begin{proof}
Fix $r>0$ and $t>0$. Since $\pi(X)$ is a bounded subset of $\R^d$,
we can cover it by $\ell$ balls $B(\pi x_i,r/2)$ of radius $r/2$,
where $x_i\in X$ and $i=1,\dots, \ell$.
Define $$
A_n=\{ x\in X\colon \ m_x^\eta(B^\pi(x,r))\leq
e^{-nt}\}, \qquad n\in \N.$$
If a ball $B^\pi(x_i,r/2)$ intersects $A_n$,
then for any $y\in A_n\cap B^\pi(x_i,r/2)$, we have
$B^\pi(x_i,r/2)\subset B^\pi(y,r)$
because $B(\pi x_i,r/2)\subset B(\pi y,r)$ by the triangle inequality.
So the definition of $A_n$ gives
$m_y^\eta(A_n\cap B^\pi( x_i,r/2))
\leq m_y^\eta(B^\pi(y ,r))\leq e^{-nt}$.
Hence
$$
m(A_n\cap B^\pi( x_i,r/2))=\int m_y^\eta(A_n\cap B^\pi( x_i,r/2))\; dm(y)
\leq e^{-nt}
$$
and $m(A_n)\leq \ell e^{-nt}$.

This estimate gives directly that
$g(x):=\log m_x^\eta(B^\pi(x,r))\in L^1(X, \B, m)$.
Note that $g(T^nx)=\sum_{i=1}^n g(T^ix)-\sum_{i=1}^{n-1} g(T^ix)$.
By the Birkhoff ergodic theorem we can get
$\lim_{n\to \infty} \frac 1n g(T^nx)=0$ for $m$-a.e.\! $x\in X$,
which is the desired result.
 \end{proof}

\begin{lem}
\label{lem-2.10}
Let  $\A$ be a sub-$\sigma$-algebra of $\B$. Let $A\in \B$ with $m(A)>0$. Then
$$
\E_m(\chi_A|\A)(x)>0
$$
for $m$-a.e.\! $x\in A$.
\end{lem}
\begin{proof}
Let $W:=\{\E_m(\chi_A|\A)\leq 0\}$. Then $W\in \A$.
  Hence
  $$
  0\geq \int_W \E_m(\chi_A|\A)\; dm=\int_W \chi_A \; dm(x)=m(A\cap W),
  $$
  which implies $m(A\cap W)=0$.
 This finishes the proof.
\end{proof}

\section{Projection measure-theoretic entropies associated with IFS}\label{S3}

Throughout this section,  let $\{S_i\}_{i=1}^\ell$ be an IFS on a closed set $X\subset\R^d$, and $(\Sigma,\sigma)$ the one-sided full shift over $\{1,\ldots,\ell\}$. Let $\M_\sigma(\Sigma)$ denote the collection of all $\sigma$-invariant
Borel probability measures on $\Sigma$.  Let $\pi\colon\Sigma\to \R^d$ be defined as in (\ref{e-1.1}), and $h_\pi(\sigma,\cdot)$ as in Definition \ref{de-1.1}.
\subsection{Some basic properties}
In this subsection, we present some basic properties of
projection measure-theoretic entropy.  Our first result is the following.

\begin{pro}
\label{pro-3.1} \begin{itemize} \item[(i)] $0\leq h_\pi(\sigma,m)\leq
h(\sigma,m)$ for every $m\in \M_\sigma(\Sigma)$, where $h(\sigma,m)$
denotes the classical measure-theoretic entropy of $m$.
\item[(ii)] The projection entropy function is affine on $\M_\sigma(\Sigma)$, i.e., for any $m_1,m_2\in \M_\sigma(\Sigma)$ and any
$0\leq p\leq 1$, we have
\begin{equation}
\label{e-3.1}
h_\pi(\sigma,pm_1+(1-p)m_2)=ph_\pi(\sigma,m_1)+(1-p)h_\pi(\sigma,m_2).
\end{equation}
\end{itemize}
\end{pro}

The proof of the above proposition will be given later. Now let us recall some notation. If $\xi$ is a partition of $\Sigma$, then $\widehat{\xi}$ denotes the $\sigma$-algebra generated by $\xi$. If $\xi_1,\ldots, \xi_n$ are countable partitions of $\Sigma$, then $\bigvee_{i=1}^n\xi_i$ denotes the partition consisting of sets $A_1\cap\cdots\cap A_n$ with $A_i\in \xi_i$. Similarly for $\sigma$-algebras $\A_1,\A_2,\ldots,$ $\bigvee_{n}\A_n$ denotes the $\sigma$-algebra generated by $\bigcup_n\A_n$.

Let $\P$ be the partition of $\Sigma$ defined as in (\ref{e-1P}). Write $\P_{0}^n=\bigvee_{i=0}^n \sigma^{-i}\P$ for $n\geq 0$.  Let $\gamma$  denote the Borel $\sigma$-algebra $\B(\R^d)$ on $\R^d$. Similar to Definition \ref{de-1.1}, we give the following definition.

\begin{de}
\label{de-3.1}
{\rm
Let $k\in \N$ and $\nu\in \M_{\sigma^k}(\Sigma)$. Define
$$
h_\pi(\sigma^k,\nu):=H_\nu\left(\P_0^{k-1}\big|\sigma^{-k}\pi^{-1}\gamma\right)-
H_\nu\left(\P_{0}^{k-1}\big|\pi^{-1}\gamma\right).
$$
}
\end{de}
 The term $h_\pi(\sigma^k,\nu)$ can be viewed as the projection measure-theoretic entropy of $\nu$ w.r.t. the IFS
 $\{S_{i_1}\circ \cdots\circ S_{i_k}:
 \;1\leq i_j\leq \ell \mbox { for $1\leq j\leq k$}\}$. The following
 proposition
 exploits the connection between $h_\pi(\sigma^k,\nu)$ and
 $h_\pi(\sigma,m)$, where $m=\frac{1}{k}\sum_{i=0}^{k-1}\nu \circ
 \sigma^{-i}$.

 \begin{pro}
 \label{pro-3.2}
Let $k\in \N$ and $\nu\in \M_{\sigma^k}(\Sigma)$. Set
$m=\frac{1}{k}\sum_{i=0}^{k-1}\nu \circ \sigma^{-i}$. Then $m$ is
$\sigma$-invariant, and $h_{\pi}(\sigma,m)=\frac{1}{k}h_\pi(\sigma^k,\nu)$.
 \end{pro}

To prove Propositions \ref{pro-3.1} and \ref{pro-3.2}, we first give some lemmas about the (conditional) information and entropy (see \S2 for the definitions).
\begin{lem}[cf. \cite{Par-book}]
\label{lem-3.3} Let $m$ be a Borel probability measure on $\Sigma$. Let $\xi,\eta$ be two countable Borel partitions of $\Sigma$ with $H_m(\xi)<\infty$, $H_m(\eta)<\infty$, and $\A$  a sub-$\sigma$-algebra of $\B(\Sigma)$. Then we have
 \begin{itemize}
\item[(i)] ${\bf I}_{m\circ \sigma^{-1}}(\xi|\A)\circ \sigma={\bf I}_m(\sigma^{-1}\xi|\sigma^{-1}\A)$.
\item[(ii)] ${\bf I}_m(\xi\vee \eta|\A)={\bf I}_m(\xi|\A)+{\bf I}_m(\eta|\widehat{\xi}\vee \A)$.
\item[(iii)] $H_m(\xi\vee \eta|\A)=H_m(\xi|\A)+H(\eta|\widehat{\xi}\vee \A)$.
\item[(iv)] If $\A_1\subset \A_2\subset\cdots $ is an increasing sequence of sub-$\sigma$-algebras with $\A_n\uparrow \A$, then ${\bf I}_m(\xi|\A_n)$ converges almost everywhere and in $L^1$ to ${\bf I}_m(\xi|\A)$. In particular, $\lim_{n\to \infty} H_m(\xi|\A_n)=H_m(\xi|\A)$.
\end{itemize}
\end{lem}

\begin{lem}
\label{lem-3.4} Denote  $g(x)=-x\log x$ for $x\geq 0$. For any
integer $k\geq 2$ and $x_1,\ldots, x_k\geq 0$, we have
$\frac{1}{k}\sum_{i=1}^k g(x_i)\leq g\left(\frac{1}{k}\sum_{i=1}^k
x_i\right)\leq \sum_{i=1}^k g(x_i/k)$ and
\begin{equation}
\label{e-3.2} \sum_{i=1}^k g(x_i)-(x_1+\ldots +x_k)\log k\leq
g(x_1+\ldots+x_k)\leq \sum_{i=1}^k g(x_i).
\end{equation}
 Moreover
for any $p_1,p_2\geq 0$ with $p_1+p_2=1$,
\begin{equation} \label{e-3.3}
\sum_{j=1}^2 p_jg(x_j)\leq g\left(\sum_{j=1}^2 p_jx_j\right)\leq
\sum_{j=1}^2 p_jg(x_j)+g(p_j) x_j.
\end{equation}
\end{lem}
\begin{proof} Standard. \end{proof}

\begin{lem}
\label{lem-3.5} Let $m$ be a Borel probability measure on $\Sigma$. Assume $\xi$ and $\eta$ are
two countable Borel partitions of $\Sigma$ such that each member in
$\xi$ intersects at most $k$ members of $\eta$. Then $H_m(\xi)\geq
H_m(\xi\vee \eta)-\log k$.
\end{lem}
\begin{proof} Although the result is standard, we give a short proof for the convenience of the reader. Denote $g(x)=-x\log x$ for $x\in [0,1]$. Then
\begin{eqnarray*}
H_m(\xi)&=&\sum_{A\in \xi} g(m(A))=\sum_{A\in \xi}g \left(\sum_{B\in
\eta,\; B\cap A\neq \emptyset} m(A\cap B)\right)\\
&\geq & \sum_{A\in \xi} \left[\left(\sum_{B\in \eta,\; B\cap A\neq
\emptyset}g (m(A\cap B))\right)-m(A)\log k\right]
\quad \mbox{(by (\ref{e-3.2}))}\\
&\geq & \left(\sum_{A\in \xi} \sum_{B\in \eta} g (m(A\cap B))\right)-\log k\\
&=& H_m(\xi\vee \eta)-\log k.
\end{eqnarray*}
This finishes the proof. \end{proof}

The following simple lemma plays an important role in our analysis.
\begin{lem}
\label{lem-3.6} $\widehat{\P}\vee \sigma^{-1}\pi^{-1}\gamma=\widehat{\P}\vee
\pi^{-1}\gamma$.
\end{lem}
\begin{proof} We only  prove  $\widehat{\P}\vee
\sigma^{-1}\pi^{-1}\gamma\subseteq\widehat{\P}\vee \pi^{-1}\gamma$. The other
direction can be proved by an essentially identical argument. Note
that each member in $\widehat{\P}\vee \sigma^{-1}\pi^{-1}\gamma$ can be
written as
$$\bigcup_{j=1}^\ell [j]\cap \sigma^{-1}\pi^{-1}A_j$$
with $A_j\in \gamma$. However, it is direct to check
that  $$[j]\cap \sigma^{-1}\pi^{-1}A_j=[j]\cap \pi^{-1}(S_j (A_j)).$$
Since $S_j$ is injective and contractive (thus continuous),
we have $S_j(A_j)\in\gamma$. Therefore $\bigcup_{j=1}^\ell [j]\cap
\sigma^{-1}\pi^{-1}A_j\in \widehat{\P}\vee \pi^{-1}\gamma$. \end{proof}

\begin{lem}
\label{lem-3.7} Let $m$ be a Borel probability measure on $\Sigma$ and  $k\in \N$. We have
\begin{eqnarray*}
&\mbox{}& H_m\left(\P_{0}^{k-1}\big|\sigma^{-k}\pi^{-1}\gamma\right)-
H_m\left(\P_{0}^{k-1}\big|\pi^{-1}\gamma\right)\\
&\mbox{}& \quad=\sum_{j=0}^{k-1} H_{m\circ \sigma^{-j}}(\P|\sigma^{-1}\pi^{-1}\gamma)-H_{m\circ \sigma^{-j}}(\P|\pi^{-1}\gamma).
\end{eqnarray*}
Moreover if $m\in \M_\sigma(\Sigma)$, then $$H_m\left(\P_{0}^{k-1}\big|\sigma^{-k}\pi^{-1}\gamma\right)-
H_m\left(\P_{0}^{k-1}\big|\pi^{-1}\gamma\right)=kh_\pi(\sigma,
m).$$
\end{lem}
\begin{proof} For $j=0,1,\ldots,k-1$,  we have
\begin{eqnarray*}
&\mbox{}&
{\bf I}_m\left(\P_{0}^{k-1}\big|\sigma^{-j}\pi^{-1}\gamma\right)-
{\bf I}_m\left(\P_{0}^{k-1}\big|\sigma^{-(j+1)}\pi^{-1}\gamma\right)\\
&=& {\bf I}_m\left(\sigma^{-j}\P\big|\sigma^{-j}\pi^{-1}\gamma\right)+
{\bf I}_m\left(\bigvee_{0\leq i\leq k-1,\;i\neq
j}\sigma^{-i}\P\big|\sigma^{-j}\widehat{\P}\vee\sigma^{-j}\pi^{-1}\gamma\right)\\
&&  \mbox{}\;\;  -{\bf I}_m\left(\P_{0}^{k-1}\big|\sigma^{-(j+1)}\pi^{-1}\gamma\right)\qquad (\mbox{by Lemma \ref{lem-3.3}(ii)})
\\
&=& {\bf I}_m\left(\sigma^{-j}\P\big|\sigma^{-j}\pi^{-1}\gamma\right)+
{\bf I}_m\left(\bigvee_{0\leq i\leq k-1,\;i\neq
j}\sigma^{-i}\P\big|\sigma^{-j}\widehat{\P}\vee\sigma^{-(j+1)}\pi^{-1}\gamma\right) \\
&& \mbox{}\;\;-{\bf I}_m\left(\P_{0}^{k-1}\big|\sigma^{-(j+1)}\pi^{-1}\gamma\right)\qquad (\mbox{by Lemma \ref{lem-3.6}})\\
&=&
{\bf I}_m\left(\sigma^{-j}\P\big|\sigma^{-j}\pi^{-1}\gamma\right)-{\bf I}_m\left(\sigma^{-j}\P\big|\sigma^{-(j+1)}\pi^{-1}\gamma\right)
\quad (\mbox{by Lemma \ref{lem-3.3}(ii)})\\
&=&{\bf I}_{m\circ \sigma^{-j}}\left(\P\big|\pi^{-1}\gamma\right)\circ
\sigma^j-{\bf I}_{m\circ \sigma^{-j}}\left(\P\big|\sigma^{-1}\pi^{-1}\gamma\right)\circ
\sigma^{j}\quad (\mbox{by Lemma \ref{lem-3.3}(i)}).
\end{eqnarray*}
Summing $j$ over $\{0,\ldots,k-1\}$ yields
\begin{equation}
\label{e-tt}
\begin{split}
\mbox{}&{\bf I}_m\left(\P_{0}^{k-1}\big|\pi^{-1}\gamma\right)-
{\bf I}_m\left(\P_{0}^{k-1}\big|\sigma^{-k}\pi^{-1}\gamma\right)\\
&=\sum_{j=0}^{k-1}\left({\bf I}_{m\circ \sigma^{-j}}\left(\P\big|\pi^{-1}\gamma\right)\circ
\sigma^j-{\bf I}_{m\circ \sigma^{-j}}\left(\P\big|\sigma^{-1}\pi^{-1}\gamma\right)\circ
\sigma^{j}\right).\\
\end{split} \end{equation}
 Taking  integration, we obtain
the desired formula. \end{proof}

For any $n\in \N$, let $\D_n$ be the partition of $\R^d$ given by
\begin{equation}
\label{e-at1}
\D_n=\{[0,2^{-n})^d+\alpha:\; \alpha\in 2^{-n} \Z^d\}.
\end{equation}

\begin{lem}
\label{lem-3.8} Let $m\in \M_\sigma(\Sigma)$.  For each $n\in \N$,
we have
$$H_m(\P|\sigma^{-1}\pi^{-1}\widehat{\D_n})-H_m(\P|\pi^{-1}\widehat{\D_n})\geq -d\log
(\sqrt{d}+1).$$
\end{lem}

\begin{proof} Since $m$ is $\sigma$-invariant, by Lemma \ref{lem-3.3}(iii), we have
\begin{equation}
\label{e-3.4}
\begin{split}
\mbox{}& H_m(\P|\sigma^{-1}\pi^{-1}\widehat{\D_n})-H_m(\P|\pi^{-1}\widehat{\D_n})\\
&=H_m(\P\vee
\sigma^{-1}\pi^{-1}{\D_n})-H_m(\sigma^{-1}\pi^{-1}{\D_n})\\
&\mbox{}\qquad -H_m(\P\vee \pi^{-1}{\D_n})+H_m(\pi^{-1}{\D_n})\\
&=H_m(\P\vee \sigma^{-1}\pi^{-1}{\D_n})-H_m(\P\vee \pi^{-1}{\D_n}).\\
\end{split}
\end{equation}
Observe that for each $1\leq j\leq \ell$ and  $Q\in \D_n$, $$[j]\cap
\sigma^{-1}\pi^{-1}(Q)=[j]\cap \pi^{-1}(S_j(Q)).$$
 Since $S_j$ is
contractive,  $\text{diam}(S_j(Q))\leq 2^{-n}\sqrt{d}$ and thus $S_j(Q)$ intersects at most $(\sqrt{d}+1)^d$  members in
$\D_n$. It deduces that $[j]\cap \sigma^{-1}\pi^{-1}(Q)$ intersects
at most $(\sqrt{d}+1)^d$  members in $\P\vee \pi^{-1}{\D_n}$. By Lemma
\ref{lem-3.5}, we have
\begin{equation}
\label{e-3.5'}
\begin{split}
 H_m(\P\vee
\sigma^{-1}\pi^{-1}{\D_n})&\geq H_m(\P\vee \sigma^{-1}\pi^{-1}{\D_n}\vee
\pi^{-1}{\D_n})\\
&\mbox{}\qquad -d\log (\sqrt{d}+1)\\
&\geq H_m(\P\vee \pi^{-1}{\D_n})-d\log (\sqrt{d}+1).\\
\end{split}
\end{equation}
Combining it with (\ref{e-3.4}) yields the desired inequality.
\end{proof}

\begin{proof}[Proof of Proposition \ref{pro-3.1}] We first prove
part (i) of the proposition, i.e.,
$$0\leq h_\pi(\sigma,m)\leq
h(\sigma,m).$$ Since $\widehat{\D_n}\uparrow \gamma$ as $n$ tends to $\infty$, by Lemma \ref{lem-3.3}(iv),
we have
$$\lim_{n\to
\infty}H_m(\P|\sigma^{-1}\pi^{-1}\D_n)-H_m(\P|\pi^{-1}\D_n)=
H_m(\P|\sigma^{-1}\pi^{-1}\gamma)-H_m(\P|\pi^{-1}\gamma).$$
It together with Lemma
\ref{lem-3.8} yields
$$H_m(\P|\sigma^{-1}\pi^{-1}\gamma)-H_m(\P|\pi^{-1}\gamma)\geq
-d\log(\sqrt{d}+1).$$ Using the same argument to the IFS $\{S_{i_1\ldots i_k}: 1\leq i_j\leq \ell, 1\leq j\leq k\}$, we have
$$H_m\left(\P_{0}^{k-1}\big|\sigma^{-k}\pi^{-1}\gamma\right)-
H_m\left(\P_{0}^{k-1}\big|\pi^{-1}\gamma\right)\geq
-d\log(\sqrt{d}+1).$$
It together with Lemma \ref{lem-3.7} yields
$h_\pi(\sigma,m)\geq -d\log (\sqrt{d}+1)/k$. Since $k$ is arbitrary, we have $h_\pi(\sigma,m)\geq 0$. To see $h_\pi(\sigma,m)\leq
h(\sigma,m)$, it suffices to observe that
\begin{eqnarray*}
kh_\pi(\sigma,m)&=&H_m\left(\P_{0}^{k-1}\big|\sigma^{-k}\pi^{-1}\gamma\right)-
H_m\left(\P_{0}^{k-1}\big|\pi^{-1}\gamma\right)\\
&\leq&
H_m\left(\P_{0}^{k-1}\big|\sigma^{-k}\pi^{-1}\gamma\right)\leq
H_m\left(\P_{0}^{k-1}\right).
\end{eqnarray*}

Now we turn to the proof of part (ii).
 Let $m_1,m_2\in \M_{\sigma}(\Sigma)$ and $m=pm_1+(1-p)m_2$ for some $p\in [0,1]$.
 Using (\ref{e-3.3}), for any finite or countable Borel
partition $\xi$ we have
\begin{equation}
\label{e-3.5} |H_m(\xi)-pH_{m_1}(\xi)-(1-p)H_{m_2}(\xi)|\leq
g(p)+g(1-p)\leq \log 2.
\end{equation}
Let $k\in \N$. By Lemma \ref{lem-3.7}, Lemma \ref{lem-3.3}(iv),  and (\ref{e-3.4}), we
have
\begin{equation}
\label{e-july3}
\begin{split}
h_\pi(\sigma,m)&= \frac{1}{k}\left(
H_m\left(\P_{0}^{k-1}\big|\sigma^{-k}\pi^{-1}\gamma\right)-
H_m\left(\P_{0}^{k-1}\big|\pi^{-1}\gamma\right)\right)\\
&= \frac{1}{k}\lim_{n\to \infty} \left(
H_m\left(\P_{0}^{k-1}\big|\sigma^{-k}\pi^{-1}\widehat{\D_n}\right)-
H_m\left(\P_{0}^{k-1}\big|\pi^{-1}\widehat{\D_n}\right)\right)\\
&= \frac{1}{k}\lim_{n\to \infty} \left(
H_m\left(\P_{0}^{k-1}\vee\sigma^{-k}\pi^{-1}\D_n\right)-
H_m\left(\P_{0}^{k-1}\vee\pi^{-1}\D_n\right)\right).\\
\end{split}
\end{equation}
The above statement is true when $m$ is replaced by $m_1$ and $m_2$.
However by (\ref{e-3.5}),
$$
H_m\left(\P_{0}^{k-1}\vee\sigma^{-k}\pi^{-1}\D_n\right)-
H_m\left(\P_{0}^{k-1}\vee\pi^{-1}\D_n\right)
$$
differs from
$$
\sum_{j=1}^2 p_j
\left[H_{m_j}\left(\P_{0}^{k-1}\vee\sigma^{-k}\pi^{-1}\D_n\right)-
H_{m_j}\left(\P_{0}^{k-1}\vee\pi^{-1}\D_n\right)\right]
$$
 at most $2\log 2$, where $p_1=p$ and $p_2=1-p$. This together with (\ref{e-july3}) yields
 (\ref{e-3.1}).
 \end{proof}

\begin{proof}[Proof of Proposition \ref{pro-3.2}]
 Let $k\geq 2$ and $\nu\in \M_{\sigma^k}(\Sigma)$. We claim that
$h_\pi(\sigma^k,\nu\circ \sigma^{-j})=h_\pi(\sigma^k,\nu)$ for
any $1\leq j\leq k-1$. To prove the claim, it suffices to prove
$h_\pi(\sigma^k,\nu\circ \sigma^{-1})=h_\pi(\sigma^k,\nu)$. Note that both $\nu$ and  $\nu\circ \sigma^{-1}$ are $\sigma^{k}$-invariant. By
Lemma \ref{lem-3.7}, we have
\begin{eqnarray*}
h_\pi(\sigma^k,\nu)&=&
{H}_\nu\left(\P_{0}^{k-1}\big|\sigma^{-k}\pi^{-1}\gamma\right)-
{H}_\nu\left(\P_{0}^{k-1}\big|\pi^{-1}\gamma\right)\\
&=&\sum_{j=0}^{k-1}\left(H_{\nu\circ \sigma^{-j}}\left(\P\big|\sigma^{-1}\pi^{-1}\gamma\right)-H_{\nu\circ \sigma^{-j}}\left(\P\big|\pi^{-1}\gamma\right)\right), \end{eqnarray*} whilst
\begin{eqnarray*}
h_\pi(\sigma^k,\nu\circ \sigma^{-1})&=&
{H}_{\nu\circ \sigma^{-1}}\left(\P_{0}^{k-1}\big|\sigma^{-k}\pi^{-1}\gamma\right)-
{H}_{\nu\circ \sigma^{-1}}\left(\P_{0}^{k-1}\big|\pi^{-1}\gamma\right)\\
&=&\sum_{j=0}^{k-1}\left(H_{\nu\circ \sigma^{-j-1}}\left(\P\big|\sigma^{-1}\pi^{-1}\gamma\right)-H_{\nu\circ \sigma^{-j-1}}\left(\P\big|\pi^{-1}\gamma\right)\right). \end{eqnarray*}
Since $\nu$ is
$\sigma^{k}$-invariant, we obtain $h_\pi(\sigma^k,\nu\circ \sigma^{-1})=h_\pi(\sigma^k,\nu)$. This finishes the proof of the claim.  To complete the proof of the proposition, let  $m=\frac{1}{k}\sum_{i=0}^{k-1}\nu\circ \sigma^{-i}$. It is clear that $m$ is $\sigma$-invariant. By Proposition \ref{pro-3.1}(ii), $h_\pi(\sigma^k,\cdot)$ is affine on $\M_{\sigma^k}(\Sigma)$. Hence
$$h_\pi(\sigma^k,m)=\frac{1}{k}\sum_{i=0}^{k-1}h_\pi(\sigma^k,\nu\circ\sigma^{-i})=h_\pi(\sigma^k,\nu).$$
Combining it with Lemma \ref{lem-3.7} yields the equality
$h_\pi(\sigma,m)=\frac{1}{k}h_\pi(\sigma^k,\nu)$. \end{proof}

\subsection{A version of Shannon-McMillan-Breiman  Theorem associated with IFS}

In this subsection, we prove the following Shannon-McMillan-Breiman type theorem associated with IFS, which is needed in the proof of Theorem \ref{thm-1.3}. It is also of independent interest.

\begin{pro}
\label{pro-3.13}
Let $\{S_i\}_{i=1}^\ell$ be an  IFS and  $m\in \M_\sigma(\Sigma)$. Then
\begin{equation}
\label{e-3.10}
\lim_{k\to \infty}\frac{1}{k}{\bf I}_m\left(\P_{0}^{k-1}\big|\pi^{-1}\gamma \right)(x)=\E_m(f|\I)(x)=h(\sigma,m,x)-
h_\pi(\sigma,m,x).
\end{equation}
almost everywhere and in $L^1$, where $$f:={\bf I}_m(\P|\sigma^{-1}\B(\Sigma))+ {\bf I}_m(\P|\pi^{-1}\gamma)-{\bf I}_m(\P|\sigma^{-1}\pi^{-1}\gamma),$$  $\I=\{B\in \B(\Sigma):\; \sigma^{-1}B=B\}$, and $h(\sigma,m,x)$,  $h_\pi(\sigma,m,x)$ denote the classical local entropy and the local projection entropy of $m$ at $x$ (see Definition \ref{de-1.1}), respectively.
Moreover if $m$ is ergodic, then the limit in (\ref{e-3.10}) equals
$h(\sigma,m)-h_\pi(\sigma,m)$ for $m$-a.e.\! $x\in \Sigma$.
\end{pro}

\begin{rem}
\label{rem-3.14}
{\rm
If $\xi$ is a countable Borel partition of $\Sigma$, and $\A\subset \B(\Sigma)$ is a sub-$\sigma$-algebra with $\sigma^{-1}\A=\A$, then the relativized
Shannon-McMillan-Breiman Theorem states that
$$
\lim_{k\to \infty}\frac{1}{k}{\bf I}_m\left(\xi_{0}^{k-1}\big|\A \right)(x)=\E_m(g|\I)(x)
\qquad \mbox{ for $m$-a.e.\! $x\in \Sigma$},
$$
where $g={\bf I}_m\left(\xi|\A\vee \xi_{1}^\infty \right)$ (see, e.g., \cite[Lemma 4.1]{Bog92}). However under the  setting of Proposition \ref{pro-3.13}, the sub-$\sigma$-algebra $\pi^{-1}\gamma$ is not invariant in general.
}
\end{rem}

In the following we present a generalized version of Proposition \ref{pro-3.13}.
\begin{pro}
\label{pro-3.13'}
Let $\xi$ be a countable Borel partition of $\Sigma$ with $H_m(\xi)<\infty$, and let $\A\subset \B(\Sigma)$ be  a sub-$\sigma$-algebra so that  $\widehat{\xi}\vee \sigma^{-1}\A=\widehat{\xi}\vee\A$.
Let $m\in \M_\sigma(\Sigma)$. Then
\begin{equation}
\label{e-3.10'}
\lim_{k\to \infty}\frac{1}{k}{\bf I}_m\left(\xi_{0}^{k-1}\big|\A\right)(x)=\E_m(f|\I)(x)
\end{equation}
almost everywhere and in $L^1$, where $$f:={\bf I}_m\left(\xi|\sigma^{-1}\A\vee \bigvee_{i=1}^\infty \sigma^{-i}\widehat{\xi}\right)+ {\bf I}_m(\xi|\A)-{\bf I}_m(\xi|\sigma^{-1}\A),$$
 and $\I=\{B\in \B(\Sigma):\; \sigma^{-1}B=B\}$.
\end{pro}

To prove Proposition \ref{pro-3.13'}, we need the following lemma.

\begin{lem}[\cite{Man87}, Corollary 1.6, p. 96]
\label{lem-3.15}
Let $m\in \M_\sigma(\Sigma)$. Let $F_k\in L^1(\Sigma,m)$ be a sequence that converges almost everywhere
and in $L^1$ to $F\in L^1(\Sigma,m)$. Then
$$
\lim_{k\to \infty}\frac{1}{k}\sum_{j=0}^{k-1}F_{k-j}(\sigma^j(x))=\E_m(F|\I)(x)
$$
almost everywhere and in $L^1$.
\end{lem}

\begin{proof}[Proof of Proposition \ref{pro-3.13'}]
For $k\geq 2$ and $x\in \Sigma$, we write
$$
g_k(x)={\bf I}_m\left(\xi_{0}^{k-1}\big|\A \right)(x)
-{\bf I}_m\left(\xi_{0}^{k-2}\big|\A \right)(\sigma x).
$$
Then
\begin{equation}
\label{e-3.11}
{\bf I}_m\left(\xi_{0}^{k-1}\big|\A \right)(x)={\bf I}_m(\xi|\A)(\sigma^{k-1}x)+\sum_{j=0}^{k-2}g_{k-j}(\sigma^jx).
\end{equation}
We claim that
\begin{equation}
\label{e-3.12}
g_k(x)={\bf I}_m\left(\xi\big|\sigma^{-1}\A \vee \bigvee_{i=1}^{k-1}\sigma^{-i}\widehat{\xi} \right)(x)
+ {\bf I}_m\left(\xi\big|\A \right)(x)-{\bf I}_m\left(\xi\big|\sigma^{-1}\A \right)(x).
%\nonumber\\
\end{equation}
By the claim and Lemma \ref{lem-3.3}(iv), $g_k$ converges almost everywhere and in $L^1$ to $f$.
It together with (\ref{e-3.11}) and Lemma \ref{lem-3.15} yields (\ref{e-3.10'}).

Now we turn to the proof of (\ref{e-3.12}). Let $k\geq 2$. We have
\begin{equation}
\label{e-3.13}
\begin{split}
{\bf I}_m\left(\xi_{0}^{k-1}\big|\sigma^{-1}\A \right)(x)
&={\bf I}_m\left(\xi\big|\sigma^{-1}\A \right)(x)+
{\bf I}_m\left(\bigvee_{i=1}^{k-1}\sigma^{-i}\xi\big|\sigma^{-1}\A\vee \widehat{\xi} \right)(x)\\
&={\bf I}_m\left(\xi\big|\sigma^{-1}\A \right)(x)+
{\bf I}_m\left(\bigvee_{i=1}^{k-1}\sigma^{-i}\xi\big|\A \vee \widehat{\xi} \right)(x),\\
\end{split}
\end{equation}
using the property $\sigma^{-1}\A \vee \widehat{\xi}=\A \vee \widehat{\xi}$.
Meanwhile, we have
\begin{equation}
\label{e-3.14}
\begin{split}
\mbox{} & {\bf I}_m\left(\xi_{0}^{k-1}\big|\sigma^{-1}A \right)(x)\\
&={\bf I}_m\left(\bigvee_{i=1}^{k-1}\sigma^{-i}\xi\big|\sigma^{-1}\A \right)(x)
+
{\bf I}_m\left(\xi\big|\sigma^{-1}\A \vee \bigvee_{i=1}^{k-1}\sigma^{-i}\widehat{\xi} \right)(x)\\
&=
{\bf I}_m\left(\xi_0^{k-2}\big|\A \right)(\sigma x)
+
{\bf I}_m\left(\xi\big|\sigma^{-1}\A \vee \bigvee_{i=1}^{k-1}\sigma^{-i}\widehat{\xi} \right)(x).\\
\end{split}
\end{equation}
Combining (\ref{e-3.13}) with (\ref{e-3.14}) yields
\begin{equation}
\label{e-3.15}
\begin{split}
&\mbox{} {\bf I}_m\left(\xi\big|\sigma^{-1}\A \right)(x)+
{\bf I}_m\left(\bigvee_{i=1}^{k-1}\sigma^{-i}\xi\big|\A \vee \widehat{\xi} \right)(x)\\
&={\bf I}_m\left(\xi_0^{k-2}\big|\A \right)(\sigma x)
+
{\bf I}_m\left(\xi\big|\sigma^{-1}\A \vee \bigvee_{i=1}^{k-1}\sigma^{-i}\widehat{\xi} \right)(x).\\
\end{split}
\end{equation}
However
\begin{equation}
\label{e-3.16}
{\bf I}_m\left(\xi_{0}^{k-1}\big|\A \right)(x)=
{\bf I}_m\left(\xi\big|\A \right)(x)+
{\bf I}_m\left(\bigvee_{i=1}^{k-1}\sigma^{-i}\xi\big|\A \vee \widehat{\xi} \right)(x).
\end{equation}
Combining (\ref{e-3.15}) with (\ref{e-3.16}) yields (\ref{e-3.12}). This finishes the proof of Proposition \ref{pro-3.13'}.
\end{proof}

\bigskip
We remark that Proposition \ref{pro-3.13} can be stated in terms of conditional measures. To see it,
let
$$\eta=\{\pi^{-1}(z): z\in \R^d\}$$
be the measurable partition of $\Sigma$ generated by the canonical projection $\pi$ associated with $\{S_i\}_{i=1}^\ell$. For  $m\in \M_\sigma(\Sigma)$, let $\{m_x^\eta\}_{x\in \Sigma}$ denote the canonical system of conditional
measures w.r.t. $\eta$. For $x\in \Sigma$ and $k\in \N$, let $\P^k_0(x)$ denote the element in the partition $\P_0^k$ containing $x$. Then Proposition \ref{pro-3.13} can be restated as the following.

\begin{pro}
\label{pro-3.16}
For $m\in \M_\sigma(\Sigma)$,  we have
\begin{equation}
\label{e-3.17} -\lim_{k\to \infty}\frac{1}{k} \log
m^\eta_x(\P^k_0(x))=\E_m(f|\I)(x) \qquad \mbox{ for $m$-a.e.\! $x\in
\Sigma$},
\end{equation}
where $f:={\bf I}_m(\P|\sigma^{-1}\B(\Sigma))+
{\bf I}_m(\P|\pi^{-1}\gamma)-{\bf I}_m(\P|\sigma^{-1}\pi^{-1}\gamma)$. Moreover
if $m$ is ergodic, then the limit in (\ref{e-3.17}) equals
$h(\sigma,m)-h_\pi(\sigma,m)$ for $m$-a.e.\! $x\in \Sigma$.
\end{pro}
\begin{proof}
It suffices to show that
for each $k\in \N$,
$$\log m^\eta_x(\P^k_0(x))=-{\bf I}_m(\P^k_0|\pi^{-1}\gamma)(x) \mbox{ almost everywhere}.$$
 To see this, by Theorem \ref{thm-2.1} we have
$$
\sum_{A\in
\P^k_0}\chi_A(x)m^\eta_x(A)=\sum_{A\in
\P^k_0}\chi_A(x)\E_m(\chi_A|\pi^{-1}\gamma)(x)\quad \mbox{ for $m$-a.e.\!
$x\in \Sigma$}.$$
Taking logarithm yields the desired result.
\end{proof}

\begin{rem}
\label{rem-3.18}
{\rm
In Proposition \ref{pro-3.16},
for $m$-a.e.\! $x\in \Sigma$, we have
$$
\lim_{k\to \infty} -\frac{1}{k}\log
m^\eta_x(\P^k_0(y))=\E_m(f|\I)(y) \quad \mbox{ for $m^\eta_x$-a.e.\!
$y\in \eta(x)$}.
$$
To see this, denote
$$
R=\left\{y\in \Sigma: \;-\lim_{k\to \infty}\frac{1}{k} \log
m^\eta_y(\P^k_0(y))=\E_m(f|\I)(y)\right\}.
$$
Then $1=m(R)=\int m^\eta_x(R\cap \eta(x))\;dm(x)$. Hence
$m^\eta_x(R\cap \eta(x))=1$  $m$-a.e. For $y\in R\cap \eta(x)$, we have $$\lim_{k\to
\infty} -\frac{1}{k}\log m^\eta_x(\P^k_0(y))= \lim_{k\to \infty}
-\frac{1}{k}\log m^\eta_y(\P^k_0(y))= \E_m(f|\I)(y).$$
}
\end{rem}

As a corollary of Proposition \ref{pro-3.16}, we have

\begin{cor}
\label{cor-3.17}
Let $m\in \M_\sigma(\Sigma)$. Then
\begin{eqnarray*}
h_\pi(\sigma,m)=h(\sigma,m) & \Longleftrightarrow & \lim_{k\to \infty}\frac{1}{k}\log m_x^\eta(\P_0^k(x))=0 \;\mbox{ $m$-a.e.}\\
 & \Longleftrightarrow &\dim_H m_x^\eta=0 \;\mbox{ $m$-a.e.}
\end{eqnarray*}
In particular, if $\dim_H\pi^{-1}(z)=0$ for each $z\in \R^d$, then $h_\pi(\sigma,m)=h(\sigma,m)$. Here $\dim_H$ denotes the Hausdorff dimension.
\end{cor}
\begin{proof}Let $f$ be defined as in Proposition \ref{pro-3.16}. Then
$$\int \E_m(f|\I) \; dm =\int f\;dm= h(\sigma,m)-h_\pi(\sigma,m).$$
By (\ref{e-3.17}), $\E_m(f|\I)(x)\geq 0$ for $m$-a.e.\! $x\in \Sigma$. Hence we have
\begin{eqnarray*}
h(\sigma,m)=h_\pi(\sigma,m) &\Longleftrightarrow& \E_m(f|\I)=0\;\mbox{ $m$-a.e.} \\
&\Longleftrightarrow &\lim_{k\to \infty}\frac{1}{k}\log m_x^\eta(\P_0^k(x))=0\;\mbox{ $m$-a.e.}
\end{eqnarray*}
Using dimension theory of measures (see, e.g., \cite{Fan94}), we have  $$\dim_H m_x^\eta={\mbox{ess}\sup}_{y\in \eta(x)} \liminf_{k\to \infty} \frac{\log m_x^\eta(\P_0^k(y))}{\log \ell^{-k}}.$$
It together with Remark \ref{rem-3.18} yields
$$\E_m(f|\I)=0\;\mbox{ $m$-a.e.}\Longleftrightarrow \dim_H m_x^\eta=0\;\mbox{ $m$-a.e.}
$$
This finishes the proof of the first part of the corollary.

To complete the proof, assume that $\dim_H\pi^{-1}(z)=0$ for each $z\in \R^d$. Then for each $x\in \Sigma$, $\dim_H\eta(x)=0$ and hence $\dim_H m_x^\eta=0$. Thus $h_\pi(\sigma,m)=h(\sigma,m)$.
\end{proof}

\subsection{Projection entropy under the ergodic decomposition}
In this subsection, we first prove the following result.
\begin{pro}\label{pro-tt} Let $\{S_i\}_{i=1}^\ell$ be an IFS and $m\in \M_\sigma(\Sigma)$. Assume that
$m=\int \nu \;d{\Bbb P}(\nu)$ is the ergodic decomposition of $m$. Then
$$
h_\pi(\sigma,m)=\int h_\pi(\sigma,\nu) \;d {\Bbb P} (\nu).
$$
\end{pro}
\begin{proof}
Let $\I$ denote the $\sigma$-algebra $\{B\in \B(\Sigma):\ \sigma^{-1}B=B\}$, and  let $m\in \M_\sigma(\Sigma)$. Then there exists an $m$-measurable partition
$\varepsilon$ of $\Sigma$ such that $\widehat{\varepsilon}=\I$ modulo sets of zero $m$-measure (see \cite[pp. 37-38]{Par-book1}). Let $\{m_x^\varepsilon\}$ denote the conditional measures of $m$  associated with the partition
$\varepsilon$. Then $m=\int m_x^\varepsilon\; dm(x)$ is just the ergodic decomposition of $m$ (see e.g., \cite[Theorem 2.3.3]{Kel-book}). Hence to prove the proposition, we need to show that
\begin{equation}\label{e-a12}
h_\pi(\sigma,m)=\int h_\pi(\sigma, m_x^\varepsilon) \;dm(x).
\end{equation}

We first show the direction ``$\leq$'' in (\ref{e-a12}). Note that $\I$ is $\sigma$-invariant and $\widehat{\P}\vee \sigma^{-1}\pi^{-1}\gamma=\widehat{\P}\vee \pi^{-1}\gamma$. Hence we have $
\widehat{\P}\vee \sigma^{-1}\pi^{-1}\gamma\vee \I=\widehat{\P}\vee \pi^{-1}\gamma\vee \I
$. Taking $\xi=\P$ and $\A=\pi^{-1}\gamma\vee \I$ in Proposition \ref{pro-3.13'} yields
\begin{equation}
\label{e-a13}
\lim_{k\to \infty}\frac{1}{k}{\bf I}_m\left(\P_{0}^{k-1}\big|\pi^{-1}\gamma\vee \I\right)(x)=\E_m(f|\I)(x)
\end{equation}
almost everywhere and in $L^1$, where $$f:={\bf I}_m\left(\P|\sigma^{-1}\B(\Sigma)\right)+ {\bf I}_m(\P|\pi^{-1}\gamma\vee \I)-{\bf I}_m(\P|\sigma^{-1}\pi^{-1}\gamma\vee \I).$$
By Remark \ref{rem-2.5}(ii), we have $${\bf I}_{m^\varepsilon_x}\left(\P_{0}^{k-1}|\pi^{-1}\gamma\right)(x)={\bf I}_m\left(\P_{0}^{k-1}\big|\pi^{-1}\gamma\vee \I\right)(x).$$ Hence according to the ergodicity of $m_x^\varepsilon$ and Proposition \ref{pro-3.13}, we have
\begin{eqnarray*}
h(\sigma, m_x^\varepsilon)-h_\pi(\sigma, m_x^\varepsilon)&=&\lim_{k\to \infty}\frac{1}{k}{\bf I}_{m^\varepsilon_x}\left(\P_{0}^{k-1}|\pi^{-1}\gamma\right)(x)\\
&=&\lim_{k\to \infty} \frac{1}{k}{\bf I}_m\left(\P_{0}^{k-1}\big|\pi^{-1}\gamma\vee \I\right)(x)
\end{eqnarray*}
almost everywhere and
\begin{equation}
\label{e-a14}
\int h(\sigma, m_x^\varepsilon)-h_\pi(\sigma, m_x^\varepsilon)\; dm(x)=\lim_{k\to \infty}\frac{1}{k}H_m\left(\P_{0}^{k-1}\big|\pi^{-1}\gamma\vee \I\right).
\end{equation}
Using Proposition \ref{pro-3.13} again we have
\begin{equation}
\label{e-a15}
h(\sigma, m)-h_\pi(\sigma, m)=\lim_{k\to \infty}\frac{1}{k}H_m\left(\P_{0}^{k-1}\big|\pi^{-1}\gamma\right).
\end{equation}
However, $H_m\left(\P_{0}^{k-1}\big|\pi^{-1}\gamma\vee \I\right)\leq H_m\left(\P_{0}^{k-1}\big|\pi^{-1}\gamma\right)$ (see e.g. \cite[Theorem 4.3 (v)]{Wal-book}). By (\ref{e-a14}), (\ref{e-a15}) and the above inequality, we have
$$\int h(\sigma, m_x^\varepsilon)-h_\pi(\sigma, m_x^\varepsilon)\; dm(x)\leq h(\sigma, m)-h_\pi(\sigma, m).$$
It is well known (see \cite[Theorem 8.4]{Wal-book}) that $\int h(\sigma, m_x^\varepsilon)\; dm(x)=h(\sigma, m)$.
Hence we obtain the inequality $h_\pi(\sigma, m)\leq \int h_\pi(\sigma, m_x^\varepsilon)\; dm(x)$.

Now we prove the direction ``$\geq$ '' in (\ref{e-a12}). For any $n\in \N$, let $\D_n$ be defined as in (\ref{e-at1}). Since $\widehat{\D_n}\uparrow \gamma$, we have
\begin{equation}
\label{e-temp1}
h_\pi(\sigma,m)=\lim_{n\to \infty}H_m(\P|\sigma^{-1}\pi^{-1}\widehat{\D_n})-H_m(\P|\pi^{-1}\widehat{\D_n}).
\end{equation}
Now fix $n\in \N$ and denote
$A(m)=H_m(\P|\sigma^{-1}\pi^{-1}\widehat{\D_n})-H_m(\P|\pi^{-1}\widehat{\D_n})$
and
\begin{eqnarray*}
B(m)&=&H_m(\sigma^{-1}\pi^{-1}{\D_n}|\P\vee \pi^{-1}\widehat{\D_n})\\
&=&H_m(\P\vee\sigma^{-1}\pi^{-1}\widehat{\D_n}\vee \pi^{-1}\widehat{\D_n})-H_m(\P\vee \pi^{-1}\widehat{\D_n}).
\end{eqnarray*}
Then by (\ref{e-3.4}) and (\ref{e-3.5'}), we have
\begin{equation}
\label{temp2}
B(m)-c\leq A(m)\leq B(m),
\end{equation}
where $c=d\log(\sqrt{d}+1)$. As a conditional entropy function, $B(m)$ is  concave on $\M_\sigma(\Sigma)$ (see, e.g., \cite[Lemma 3.3 (1)]{HYZ06}). Hence by Jensen's inequality and (\ref{temp2}), we have
$$
A(m)\geq B(m)-c\geq \int B(m^\varepsilon_x) \; dm(x)-c\geq \int A(m^\varepsilon_x) \; dm(x)-c.
$$
That is,
\begin{eqnarray*}
&\mbox{}& H_m(\P|\sigma^{-1}\pi^{-1}\widehat{\D_n})-H_m(\P|\pi^{-1}\widehat{\D_n})\\
&\mbox{}& \quad  \geq  \int H_{m^\varepsilon_x}(\P|\sigma^{-1}\pi^{-1}\widehat{\D_n})-H_{m^\varepsilon_x}(\P|\pi^{-1}\widehat{\D_n}) \; dm(x)-c.
\end{eqnarray*}
Letting $n\to \infty$, using (\ref{e-temp1}) and Lebesgue dominated convergence theorem, we have
$$
h_\pi(\sigma,m)\geq  \int h_\pi(\sigma, m^\varepsilon_x)\; dm(x)-c.
$$
Replacing $\sigma$ by $\sigma^k$ we have
\begin{equation}
\label{e-1120}
h_\pi(\sigma^k,m)\geq  \int h_\pi(\sigma^k, m^{\varepsilon_k}_x)\; dm(x)-c,
\end{equation}
where $\varepsilon_k$ denotes a measurable partition of $\Sigma$ such that
$$\widehat{\varepsilon_k}=\{B\in \B(\Sigma):\;\sigma^{-k}B=B\}$$
modulo sets of zero $m$-measure.
Note that $m=\int m_x^{\varepsilon_k}\; dm(x)$ is the ergodic decomposition of $m$ with respect to $\sigma^k$. Hence
$m=\int (1/k)\sum_{i=0}^{k-1} m_x^{\varepsilon_k}\circ \sigma^{-i}~ dm(x)$ is the ergodic decomposition of $m$ with respect to $\sigma$.
It follows that
\begin{equation}
\label{e-1120a}
\frac{1}{k}\sum_{i=0}^{k-1} m_x^{\varepsilon_k}\circ \sigma^{-i}=m_x^\varepsilon \quad \mbox{$m$-a.e.}
\end{equation}
By (\ref{e-1120}),   Proposition \ref{pro-3.2} and (\ref{e-1120a}), we have
\begin{eqnarray*}
h_\pi(\sigma^k,m)&=&\frac{1}{k} \sum_{i=0}^{k-1}h_\pi(\sigma^k,m\circ \sigma^{-i})\\
&\geq&  \frac{1}{k} \sum_{i=0}^{k-1} \int  h_\pi(\sigma^k, m^{\varepsilon_k}_x\circ \sigma^{-i})\; dm(x)-c\\
 &=&\int  h_\pi\Big(\sigma^k, \frac{1}{k} \sum_{i=0}^{k-1} m^{\varepsilon_k}_x\circ \sigma^{-i}\Big)\; dm(x)-c\\
 &=&\int  h_\pi(\sigma^k,  m^{\varepsilon}_x)\; dm(x)-c.
 \end{eqnarray*}
  Using Proposition \ref{pro-3.2} again yields
  $$
h_\pi(\sigma,m)\geq  \int h_\pi(\sigma, m^\varepsilon_x)\; dm(x)-c/k\quad \mbox{for any $k\in \N$.}
 $$
 Hence we have $h_\pi(\sigma,m)\geq  \int h_\pi(\sigma, m^\varepsilon_x)\; dm(x)$, as desired.
\end{proof}

\begin{proof}[Proof of Theorem \ref{thm-1.0}] It follows directly from Propositions \ref{pro-3.1}, \ref{pro-3.13}  and \ref{pro-tt}.
\end{proof}

\subsection{The projection entropy  for certain affine IFS and the proof of Theorem \ref{thm-1.0'}}

In this subsection, we assume that $\Phi=\{S_i\}_{i=1}^\ell$ is an IFS on $\R^d$ of the form
$$S_i(x)=Ax+c_i\qquad (i=1,\ldots, \ell),$$
where $A$ is a $d\times d$ non-singular real matrix with $\|A\|<1$ and $c_i\in \R^d$. Let $K$ denote the attractor of $\Phi$.

Let $\Q$ denote the partition $\{[0,1)^d+\alpha:\; \alpha\in \Z^d\}$ of $\R^d$. For $n=0, 1,\ldots$, and $x\in \R^d$,  we set
$$
\Q_n=\{A^nQ:\; Q\in \Q\},\quad \Q_n+x=\{A^nQ+x:\; Q\in \Q\}.
$$
We have  the following geometric characterization of $h_\pi$ for the
IFS $\Phi$ (i.e., Theorem \ref{thm-1.0'}).

\begin{pro}
\label{pro-3.19}
\begin{itemize}
\item[(i)] Let $m\in \M_\sigma(\Sigma)$. Then
\begin{equation}
\label{e-3.18}
h_\pi(\sigma,m)=\lim_{n\to\infty}\frac{H_m(\pi^{-1}\Q_n)}{n}.
\end{equation}

\item[(ii)]
\begin{equation*}
\lim_{n\to \infty}\frac{\log \#\{Q\in \Q:\; A^nQ\cap K\neq \emptyset\}}{n}
=\sup\{h_\pi(\sigma,m):\; m\in \M_\sigma(\Sigma)\}.
\end{equation*}
\end{itemize}
\end{pro}

To prove the above proposition, we need the following lemma.

\begin{lem}
\label{lem-3.20}
Assume that $\Omega$ is a subset of $\{1,\ldots,\ell\}$ such that $S_i(K)\cap S_j(K)=\emptyset$ for all $i,j\in \Omega$ with $i\neq j$.  Suppose that $\nu$ is an invariant measure on $\Sigma$  supported on $\Omega^\N$, i.e., $\nu([j])=0$ for all $j\in \{1,\ldots, \ell\}\backslash \Omega$. Then $h_\pi(\sigma,\nu)=h(\sigma,\nu)$.
\end{lem}
\begin{proof} It suffices to prove that $h_\pi(\sigma,\nu)\geq h(\sigma,\nu)$. Recall that
$$h_\pi(\sigma,\nu)=H_\nu(\P|\sigma^{-1}\pi^{-1}\gamma)-H_\nu(\P|\pi^{-1}\gamma)$$ and
$H_\nu(\P|\sigma^{-1}\pi^{-1}\gamma)\geq H_\nu(\P|\sigma^{-1}\B(\Sigma))=h(\sigma,\nu)$. Hence we only need to show
 $H_\nu(\P|\pi^{-1}\gamma)=0$. To do this, denote
$$\delta=\min\{d(S_i(K), S_j(K)):\; i,j\in \Omega, i\neq j\}.$$
Then $\delta>0$. Let $\xi$ be an arbitrary  finite Borel partition of $K$ so that $\mbox{diam}(A)<\delta/2$ for $A\in \xi$.
Set ${\mathcal W}=\{[i]:\; i\in \Omega\}$. Since $\nu$ is supported on $\Omega^\N$, we have
$$
H_\nu(\P|\pi^{-1}\widehat{\xi})=H_\nu(\P\vee \pi^{-1}\xi)-H_\nu(\pi^{-1}\xi)=H_\nu({\mathcal W}\vee \pi^{-1}\xi)-H_\nu(\pi^{-1}\xi).$$
However for each $A\in \xi$, there is at most one $i\in \Omega$ such that $S_i(K)\cap A\neq \emptyset$, i.e.,
 $[i]\cap \pi^{-1}A\neq \emptyset$.  This forces that $H_\nu({\mathcal W}\vee \pi^{-1}\xi)=H_\nu(\pi^{-1}\xi)$.
Hence $$H_\nu(\P|\pi^{-1}\widehat{\xi})=0.$$ By the arbitrariness of
$\xi$ and Lemma
\ref{lem-3.3}(iv), we have $H_\nu(\P|\pi^{-1}\gamma)=0.$
\end{proof}

\begin{proof}[Proof of Proposition \ref{pro-3.19}] We first prove (i). Let $m\in \M_\sigma(\Sigma)$. Denote $\gamma=\B(\R^d)$. According to Proposition \ref{pro-3.2}, we have
$$
H_m(\P_0^{p-1}|\sigma^{-p}\pi^{-1}\gamma)-H_m(\P_0^{p-1}|\pi^{-1}\gamma)=ph_\pi(\sigma,m)\qquad (p\in \N).
$$
Now fix $p$. Since $\widehat{\Q_n}\uparrow \gamma$, by Lemma \ref{lem-3.3}(iv), there exists $k_0$ such that for $k\geq k_0$,
\begin{eqnarray*}
&\mbox{}&|H_m(\P_0^{p-1}|\sigma^{-p}\pi^{-1}\gamma)-H_m(\P_0^{p-1}|\sigma^{-p}\pi^{-1}\widehat{\Q_{kp}})|\leq 1,
\quad \mbox{and}  \\
&& |H_m(\P_0^{p-1}|\pi^{-1}\gamma)-H_m(\P_0^{p-1}|\pi^{-1}\widehat{\Q_{(k+1)p}})|\leq 1.
\end{eqnarray*}
It follows that for $k\geq k_0$,
\begin{equation}
\label{e-3.20}
\begin{split}
ph_\pi(\sigma,m)-2 & \leq H_m(\P_0^{p-1}|\sigma^{-p}\pi^{-1}\widehat{\Q_{kp}})-H_m(\P_0^{p-1}|\pi^{-1}\widehat{\Q_{(k+1)p}})\\
\mbox{} &\leq ph_\pi(\sigma,m)+2.
\end{split}
\end{equation}
Now we estimate  the difference of conditional entropies in (\ref{e-3.20}). Note that
\begin{eqnarray*}
H_m(\P_0^{p-1}|\sigma^{-p}\pi^{-1}\widehat{\Q_{kp}})&=&H_m(\P_0^{p-1}\vee \sigma^{-p}\pi^{-1}\Q_{kp})-
H_m(\sigma^{-p}\pi^{-1}\Q_{kp})\\
&=&H_m(\P_0^{p-1}\vee \sigma^{-p}\pi^{-1}\Q_{kp})-
H_m(\pi^{-1}\Q_{kp})
\end{eqnarray*}
and
$$
H_m(\P_0^{p-1}|\pi^{-1}\widehat{\Q_{(k+1)p}})=H_m(\P_0^{p-1}\vee \pi^{-1}\Q_{(k+1)p})-
H_m(\pi^{-1}\Q_{(k+1)p}).
$$
Hence we have
\begin{equation}
\label{e-3.21}
\begin{split}
&\mbox{}H_m(\P_0^{p-1}|\sigma^{-p}\pi^{-1}\widehat{\Q_{kp}})-H_m(\P_0^{p-1}|\pi^{-1}\widehat{\Q_{(k+1)p}})\\
&= H_m(\P_0^{p-1}\vee \sigma^{-p}\pi^{-1}\Q_{kp})-H_m(\P_0^{p-1}\vee \pi^{-1}\Q_{(k+1)p})\\
& \mbox{}\quad\;\; +H_m(\pi^{-1}\Q_{(k+1)p})-H_m(\pi^{-1}\Q_{kp}).\\
\end{split}
\end{equation}
Observe that for each $[u]\in \P_0^{p-1}$ and any $Q\in \Q$,
$$[u]\cap \sigma^{-p}\pi^{-1}A^{kp}Q=[u]\cap \pi^{-1}S_u A^{kp}Q.$$
 Since the linear part of $S_u$ is $A^p$,  the set $S_u A^{kp}Q$ intersects at most $2^d$ elements of $\Q_{(k+1)p}$.
Therefore each element of
$\P_0^{p-1}\vee\sigma^{-p}\pi^{-1}{\Q_{kp}}$ intersects at
most $2^d$ elements of $\P_0^{p-1}\vee \pi^{-1}\Q_{(k+1)p}$. Similarly, the statement is also true if the two partitions are interchanged. Therefore by Lemma \ref{lem-3.5}, we have
\begin{equation*}
|H_m(\P_0^{p-1}\vee \sigma^{-p}\pi^{-1}\Q_{kp})-H_m(\P_0^{p-1}\vee \pi^{-1}\Q_{(k+1)p})|\leq d\log 2.
\end{equation*}
It together with (\ref{e-3.20}) and (\ref{e-3.21}) yields
\begin{eqnarray*}
ph_\pi(\sigma,m)-2-d\log 2&\leq& H_m(\pi^{-1}\Q_{(k+1)p})-H_m(\pi^{-1}\Q_{kp})\\
&\leq & ph_\pi(\sigma,m)+2+d\log 2
\end{eqnarray*}
for $k\geq k_0$. Hence we have
\begin{eqnarray*}
&\mbox{}&\limsup_{k\to \infty} \frac{H_m(\pi^{-1}\Q_{kp})}{kp}\leq h_\pi(\sigma,m)+\frac{2+d\log 2}{p} \quad
\mbox{ and }\\ &&\liminf_{k\to \infty} \frac{H_m(\pi^{-1}\Q_{kp})}{kp}\geq h_\pi(\sigma,m)-\frac{2+d\log 2}{p}.
\end{eqnarray*}
By a volume argument,  there is a large integer $N$ ($N$ depends on $A$, $d$, $p$; and it is independent of $k$) such that for any $i=0,\ldots, p-1$, each element of $\Q_{kp+i}$ intersects at most $N$ elements of $\Q_{kp}$, and vice versa.
Hence by Lemma \ref{lem-3.5}, $|H_m(\pi^{-1}\Q_{kp})-H_m(\pi^{-1}\Q_{kp+i})|<\log N$ for $0\leq i\leq p-1$.
It follows that
\begin{eqnarray*}
 &\mbox{}&\limsup_{k\to \infty} H_m(\pi^{-1}\Q_{kp})/(kp)=\limsup_{n\to \infty}H_m(\pi^{-1}\Q_{n})/n
\mbox{ and}\\
&\mbox{}&\liminf_{k\to \infty} H_m(\pi^{-1}\Q_{kp})/(kp)=\liminf_{n\to \infty}H_m(\pi^{-1}\Q_{n})/n.
 \end{eqnarray*}
   Thus we have
\begin{eqnarray*}
h_\pi(\sigma,m)-\frac{2+d\log 2}{p}&\leq& \liminf_{n\to \infty} \frac{H_m(\pi^{-1}\Q_{n})}{n}\leq
\limsup_{n\to \infty} \frac{H_m(\pi^{-1}\Q_{n})}{n}\\
&\leq& h_\pi(\sigma,m)+\frac{2+d\log 2}{p}.
\end{eqnarray*}
Letting $p$ tend to infinity, we obtain (\ref{e-3.18}).

To show (ii), we assume $K\subset [0,1)^d$, without loss of generality. Note that the number of (non-empty) elements
in the partition $\pi^{-1}\Q_n$ is just equal to
$$N_n:=\#\{Q\in \Q:\; A^nQ\cap K\neq \emptyset\}.$$ Hence by (\ref{e-3.2}), we have
$$
H_m(\pi^{-1}\Q_n)\leq \log N_n, \quad \forall \; m\in \M_\sigma(\Sigma).
$$
This together with (i)  proves
\begin{equation*}
\liminf_{n\to \infty}\frac{\log N_n}{n}\geq \sup\{h_\pi(\sigma,m):\; m\in \M_\sigma(\Sigma)\}.
\end{equation*}
To prove (ii), we still need to show
\begin{equation}
\label{e-3.23}
\limsup_{n\to \infty}\log N_n/n\leq \sup\{h_\pi(\sigma,m):\; m\in \M_\sigma(\Sigma)\}.
\end{equation}
We may assume that $\limsup_{n\to \infty}\log N_n/n>0$, otherwise there is nothing to prove. Let $n$ be a large integer so that $N_n>7^d$.  Choose a subset $\Gamma$ of $$\{Q: \; A^nQ\cap K\neq \emptyset, Q\in \Q\}$$ such that
$\# \Gamma>7^{-d}N_n$, and
\begin{equation}
\label{e-3.24}
2Q\cap 2\widetilde{Q}=\emptyset \qquad \mbox{ for different }Q,\widetilde{Q}\in \Gamma,
\end{equation}
where $2Q:=\bigcup_{P\in \Q: \; \overline{P}\cap \overline{Q}\neq \emptyset}P$, and $\overline{P}$ denotes the closure of $P$.
For each $Q\in \Gamma$, since $A^nQ\cap K\neq \emptyset$, we can pick a word $u=u(Q)\in \Sigma_n$ such that $S_uK\cap A^nQ\neq \emptyset$. Consider the collection $W=\{u(Q): Q\in \Gamma\}$. The separation condition (\ref{e-3.24}) for  elements in $\Gamma$ guarantees that
$$
S_{u(Q)}(K)\cap S_{u(\widetilde{Q})}(K)=\emptyset\quad\mbox{for all } Q,\widetilde{Q}\in \Gamma \mbox{ with }Q\neq \widetilde{Q}.
$$
Define a Bernoulli measure $\nu$ on $W^\N$  by
$$\nu([w_1\ldots w_k])=(\#\Gamma)^{-k}\qquad (k\in \N, \;w_1,\ldots, w_k\in W).$$
Then $\nu$ can be viewed as a $\sigma^n$-invariant measure on $\Sigma$ (by viewing $W^\N$ as a subset of $\Sigma$).
By Lemma \ref{lem-3.20}, we have $h_\pi(\sigma^n, \nu)=h(\sigma^n,\nu)=\log \#\Gamma$. Define $\mu=\frac{1}{n}\sum_{i=0}^{n-1}\nu\circ \sigma^{-i}$. Then $\mu\in \M_\sigma(\Sigma)$, and by Proposition \ref{pro-3.2},
$$
h_\pi(\sigma, \mu)=\frac{h_\pi(\sigma^n, \nu)}{n}=\frac{\log \#\Gamma}{n}\geq \frac{\log (7^{-d}N_n)}{n},
$$
from which (\ref{e-3.23}) follows.
\end{proof}

\subsection{Upper semi-continuity of $h_\pi(\sigma,\cdot)$ under the AWSC}

In this subsection, we prove the following proposition.

\begin{pro}
\label{pro-3.9} Assume that $\{S_i\}_{i=1}^\ell$ is an IFS which satisfies the AWSC (see Definition \ref{de-1.5}). Then the map $m\mapsto
h_\pi(\sigma, m)$ is upper semi-continuous on $\M_\sigma(\Sigma)$.
\end{pro}

We first prove a lemma.
\begin{lem}
\label{lem-3.10} Let $\{S_i\}_{i=1}^\ell$ be an IFS with attractor $K\subset \R^d$. Assume that
$$\#\{1\leq i\leq \ell:\; x\in S_i(K)\}\leq k$$ for some $k\in \N$ and each $x\in \R^d$.  Then
$H_\nu(\P|\pi^{-1}\gamma)\leq \log k$ for any Borel probability measure $\nu$ on $\Sigma$.
\end{lem}
\begin{proof} A compactness argument shows that there is $r_0>0$ such that $$\#\{1\leq i\leq \ell:\; B(x,r_0)\cap S_i(K)\neq \emptyset\}\leq k$$
for each $x\in \R^d$.    Let  $n\in \N$ so that $2^{-n}\sqrt{d}<r_0$. Then for each
$Q\in\D_n$, where $\D_n$ is defined as in (\ref{e-at1}), there are at most $k$ different  $i\in\{1,\ldots,
\ell\}$
 such that $S_i(K)\cap Q\neq \emptyset$. It follows that each member
 in $\pi^{-1}\D_n$ intersects at most $k$ members of $\P\vee
 \pi^{-1}\D_n$. By Lemma \ref{lem-3.5}, we have
 $$H_\nu(\P|\pi^{-1}\widehat{\D_n})=H_\nu(\P\vee\pi^{-1}\D_n)-H_\nu(\pi^{-1}\D_n)\leq \log
 k.$$
Note that $\pi^{-1}\widehat{\D_n}\uparrow \pi^{-1}\gamma$.  Applying Lemma \ref{lem-3.3}(iv), we obtain
$$H_\nu(\P|\pi^{-1}\gamma)=\lim_{n\to \infty}H_\nu(\P|\pi^{-1}\widehat{\D_n})\leq \log
 k.$$
 \end{proof}

As a corollary, we have
\begin{cor}
\label{cor-3.11}
Under the condition of Lemma \ref{lem-3.10}, we have
$$h_\pi(\sigma,m)\geq h(\sigma,m)-\log k$$
 for any $m\in
\M_\sigma(\Sigma)$.
\end{cor}
\begin{proof} By the definition of $h_\pi(\sigma,m)$ and Lemma
\ref{lem-3.10}, we have
$$
h_\pi(\sigma,m)=H_m(\P|\sigma^{-1}\pi^{-1}\gamma)-H_m(\P|\pi^{-1}\gamma)\geq
H_m(\P|\sigma^{-1}\pi^{-1}\gamma)-\log k.
$$
However, $H_m(\P|\sigma^{-1}\pi^{-1}\gamma)\geq
H_m(\P|\sigma^{-1}\B(\Sigma))=h(\sigma,m)$. This implies the desired
result. \end{proof}

To prove Proposition \ref{pro-3.9}, we  need the following lemma.

\begin{lem}
\label{lem-3.new}
Let $\{S_i\}_{i=1}^\ell$ be an IFS with attractor $K$. Suppose that $\Omega$ is a subset of $\{1,\ldots, \ell\}$ such that there is a map $g\colon \{1,\ldots,\ell\}\to \Omega$ so that $$
S_i=S_{g(i)}\qquad (i=1,\ldots,\ell).
$$
Let $(\Omega^\N, \widetilde{\sigma})$ denote the one-sided full shift over $\Omega$. Define $G: \Sigma\to \Omega^\N$ by $(x_j)_{j=1}^\infty\mapsto (g(x_j))_{j=1}^\infty$. Then
\begin{itemize}
\item[(i)] $K$ is also the attractor of $\{S_i\}_{i\in \Omega}$. Moreover if we let $\widetilde{\pi}\colon \Omega^\N\to K$ denote the canonical projection w.r.t. $\{S_i\}_{i\in \Omega}$, then we have $\pi=\widetilde{\pi}\circ G$.
\item[(ii)]
Let $m\in \M_\sigma(\Sigma)$. Then $\nu=m\circ G^{-1}\in \M_{\widetilde{\sigma}}(\Omega^\N)$. Furthermore,
$h_\pi(\sigma,m)=h_{\widetilde{\pi}}(\widetilde{\sigma},\nu)$. In particular, $h_\pi(\sigma,m)\leq \log (\#\Omega)$.
\end{itemize}
\end{lem}
\begin{proof}
(i) is obvious. To see (ii), let $m\in \M_\sigma(\Sigma)$. It is easily seen that the following diagram commutes:
$$
\begin{CD}
\Sigma @ > \sigma >>\Sigma\\
@VGVV  @VVGV\\
\Omega^\N @>\widetilde{\sigma}>>\Omega^\N.
\end{CD}
$$
That is, $\widetilde{\sigma}\circ G=G\circ \sigma$. Hence $\nu=m\circ G^{-1}\in \M_{\widetilde{\sigma}}(\Omega^\N)$.
To show that $h_\pi(\sigma,m)=h_{\widetilde{\pi}}(\widetilde{\sigma},\nu)$,
let $\Q=\{[i]: \;i\in \Omega\}$ be the canonical partition of $\Omega^\N$. Then
\begin{eqnarray*}
h_{\widetilde{\pi}}(\widetilde{\sigma},\nu)
&=&H_{m\circ
G^{-1}}(\Q|\widetilde{\sigma}^{-1}\widetilde{\pi}^{-1}\gamma)-
H_{m\circ G^{-1}}(\Q|\widetilde{\pi}^{-1}\gamma)\\
&=&H_m(G^{-1}(\Q)|G^{-1}\widetilde{\sigma}^{-1}\widetilde{\pi}^{-1}\gamma)-
H_m(G^{-1}(\Q)|G^{-1}\widetilde{\pi}^{-1}\gamma)\\
&=&H_m(G^{-1}(\Q)|\sigma^{-1}{\pi}^{-1}\gamma)-
H_\nu(G^{-1}(\Q)|{\pi}^{-1}\gamma),
\end{eqnarray*}
using the facts $G\circ \sigma=\widetilde{\sigma}\circ G$ and
$\widetilde{\pi}\circ G=\pi$. Since $\P\vee G^{-1}(\Q)=\P$, we have
\begin{eqnarray*}
&\mbox{}& h_\pi(\sigma,m)- h_{\widetilde{\pi}}(\widetilde{\sigma}, m\circ
G^{-1})\\
&\mbox{}&\quad =\left(H_m\left(\P|\sigma^{-1}{\pi}^{-1}\gamma\right)-
H_m\left(\P|{\pi}^{-1}\gamma\right)\right)\\
&\mbox{}&\qquad \quad -
\left(H_m(G^{-1}(\Q)|\sigma^{-1}{\pi}^{-1}\gamma)-
H_m(G^{-1}(\Q)|{\pi}^{-1}\gamma)\right)\\
&\mbox{}&\quad =\left(H_m\left(\P|\sigma^{-1}{\pi}^{-1}\gamma\right)-
H_m(G^{-1}(\Q)|\sigma^{-1}{\pi}^{-1}\gamma)\right)\\
&\mbox{}&\qquad \quad -
\left(H_m\left(\P|{\pi}^{-1}\gamma\right)-
H_m(G^{-1}(\Q)|{\pi}^{-1}\gamma)\right)\\
&\mbox{}&\quad =H_m\left(\P|\sigma^{-1}{\pi}^{-1}\gamma\vee G^{-1}(\widehat{\Q})\right)-
H_m\left(\P|{\pi}^{-1}\gamma\vee G^{-1}(\widehat{\Q})\right).
\end{eqnarray*}
An argument similar to the proof of Lemma \ref{lem-3.6} shows that
$$\sigma^{-1}{\pi}^{-1}\gamma\vee G^{-1}(\widehat{\Q})={\pi}^{-1}\gamma\vee
G^{-1}(\widehat{\Q}).$$ Hence we have $h_\pi(\sigma,m)= h_{\widetilde{\pi}}(\widetilde{\sigma}, m\circ
G^{-1})$.
\end{proof}

\begin{proof}[Proof of Proposition \ref{pro-3.9}] Let $(\nu_n)$ be a sequence in $\M_\sigma(\Sigma)$ converging to $m$ in the weak-star topology. We need to show that $\limsup_{n\to \infty}h_\pi(\sigma, \nu_n)\leq h_\pi(\sigma,
m)$. To see this, it suffices to show that
\begin{equation}
\label{e-3.8}
\limsup_{n\to
\infty}h_\pi(\sigma, \nu_n)\leq h_\pi(\sigma, m)+\frac{1}{k}\log
t_k
\end{equation}
 for each $k\in \N$, where $t_k$ is given as in Definition \ref{de-1.5}.

 To prove (\ref{e-3.8}), we fix $k\in \N$. Define an equivalence relation $\sim$ on $\{1,\ldots, \ell\}^k$ by $u\sim v$ if $S_u=S_v$.
Let $\underline{u}$ denotes the equivalence class containing $u$. Denote $S_{\underline{u}}=S_u$.  Set $\J=\{\underline{u}:\; u\in \{1,\ldots, \ell\}^k\}$.  Let $(\J^\N, T)$ denote the one-sided full shift space over the alphabet $\J$. Let $G: \Sigma\to \J^N$ be  defined by
$$(x_i)_{i=1}^\infty \mapsto  \left(\underline{x_{jk+1}\cdots x_{(j+1)k}}\right)_{j=0}^\infty.$$
It is clear that the following diagram commutes:
$$
\begin{CD}
\Sigma @>\sigma^k>> \Sigma\\
@VGVV  @VVGV\\
 \J^N @> T>>  \J^N
\end{CD}
$$
That is,  $T\circ G=G\circ \sigma^k$.   It implies that $\nu_n\circ G^{-1}$, $m\circ G^{-1}\in \M_{T}(\J^\N)$ and $$\lim_{n\to \infty}\nu_n\circ G^{-1}=m\circ G^{-1}.$$ Hence we have
\begin{equation}
\label{e-3.9}
h(T, m\circ G^{-1})\geq \limsup_{n\to \infty}h(T,\nu_n\circ G^{-1}),
\end{equation}
where we use  the upper semi-continuity of the classical measure-theoretic entropy map on $(\J^\N, T)$.
Define $\widetilde{\pi}: \J^\N\to K$ by
$$\widetilde{\pi}\left((\underline{u_i})_{i=1}^\infty\right)=
\lim_{n\to \infty}S_{\underline{u_1}}\circ \cdots \circ S_{\underline{u_n}}(K).$$
Then  $\widetilde{\pi}\circ G=\pi$. By the assumption of AWSC (\ref{e-1.16}) and  Corollary \ref{cor-3.11} (considering the IFS $\{S_{\underline{u}}:\;
 \underline{u}\in \J\}$), we have
\begin{eqnarray*}
h_{\widetilde{\pi}}(T, m\circ G^{-1})&\geq& h(T, m\circ G^{-1})-\log t_k\\
&\geq& \limsup_{n\to \infty}h(T,\nu_n\circ G^{-1})-\log t_k\qquad (\mbox{ by (\ref{e-3.9})})\\
&\geq& \limsup_{n\to \infty}h_{\widetilde{\pi}}(T,\nu_n\circ G^{-1})-\log t_k,
\end{eqnarray*}
where the last inequality follows from Proposition \ref{pro-3.1}(i).  Then (\ref{e-3.8}) follows from the above inequality, together with Proposition \ref{pro-3.2} and the following claim:
\begin{equation}
\label{e-3.10*}
h_{\widetilde{\pi}}(T, \nu\circ G^{-1})=h_\pi(\sigma^k, \nu)\qquad (\nu\in \M_\sigma(\Sigma)).
\end{equation}
However, (\ref{e-3.10*}) just comes from Lemma \ref{lem-3.new}, where we consider the IFS $\{S_{u}:\; u\in \{1,\ldots,\ell\}^k\}$ rather than $\{S_i\}_{i=1}^\ell$.
\end{proof}

\section{Some geometric properties of $C^1$ IFS}\label{S4}
In this section we give some geometric properties of $C^1$ IFS.
\begin{lem}
\label{lem-4.1} Let $S: U\to S(U)\subset \R^d$ be a $C^1$  diffeomorphism  on an
open set $U\subset \R^d$, and $X$ a compact subset of $U$. Let $c>1$. Then there exists $r_0>0$ such that
\begin{equation}
\label{e-a4.1}
 c^{-1} \sm S'(x) \sm \cdot |x-y|\leq |S(x)-S(y)|\leq c\|S'(x)\|\cdot |x-y|
\end{equation}
for all $x\in X$,  $y\in U$ with $|x-y|\leq r_0$, where $S'(x)$ denotes the differential of $S$ at $x$,
$\sm \cdot\sm$ and $\|\cdot\|$ are defined as in (\ref{e-M1}). As a consequence,
\begin{equation}
\label{e-a4.2}
 B(S(x),c^{-1} \sm S'(x) \sm r)\subset
S\left(B(x,r)\right)\subset B(S(x),c\|S'(x)\|r)
\end{equation}
for all $x\in X$ and $0<r\leq r_0$.
\end{lem}
\begin{proof} Let $c>1$. We only prove (\ref{e-a4.1}), for it is not hard to derive (\ref{e-a4.2}) from (\ref{e-a4.1}).
Assume on the contrary that (\ref{e-a4.1}) is not true. Then there
exist two sequences $(x_n)\subset X, (y_n)\subset U$ such that
$x_n\neq y_n$, $\lim_{n\to \infty} |x_n-y_n|=0$ and for each $n\geq
1$,
\begin{equation}
\label{e-4.2}
 \begin{split}
 \mbox{either }&|S(x_n)-S(y_n)|\geq c \|S^\prime (x_n)\|\cdot|x_n-y_n|, \\
 \mbox{or }\quad  &|S(x_n)-S(y_n)|\leq c^{-1} \sm S^\prime (x_n) \sm \cdot|x_n-y_n|.
 \end{split}
\end{equation}
Since $X$ is compact, without lost of generality, we assume that
$$\lim_{n\to \infty} x_n=x=\lim_{n\to \infty}y_n.$$
 Write
$S=(f_1,f_2,\ldots,f_d)^t$. Then each component $f_j$ of $S$ is a
$C^1$ real-valued function defined on $U$. Choose a small
$\epsilon>0$ such that
$$
\{z\in \R^d: |z-x|\leq \epsilon \mbox { for some } x\in X\}\subset U.
$$
Take $N\in \N$ such that $|x_n-y_n|<\epsilon$ for $n\geq N$. By the mean value theorem, for each $n\geq N$ and $1\leq j\leq d$, there exists $z_{n,j}$ on the segment $L_{x_n, y_n}$ connecting $x_n$ and $y_n$ such that
$$
f_j(x_n)-f_j(y_n)=\nabla f_j(z_{n,j})\cdot (x_n-y_n),
$$
where $\nabla f_j$ denote the gradient of $f_j$. Therefore
$|S(x_n)-S(y_n)|=|M_n(x_n-y_n)|$ with $M_n:=(\nabla f_1(z_{n,1}),\ldots, \nabla f_d(z_{n,d}))^t$. It follows
\begin{equation}
\label{e-4.3}
\sm M_n \sm \cdot |x_n-y_n|\leq |S(x_n)-S(y_n)|\leq \|M_n\|\cdot|x_n-y_n|.
\end{equation}
Since $S$ is $C^1$, $M_n$ tends to $S^\prime(x)$ as $n\to \infty$.
Thus we have  $\sm M_n \sm\to \sm S'(x)\sm$ and $\|M_n\|\to
\|S'(x)\|$. Meanwhile, $\sm S'(x_n) \sm\to \sm S'(x)\sm$ and
$\|S'(x_n)\|\to \|S'(x)\|$. These limits together (\ref{e-4.3}) lead
to a contradiction with (\ref{e-4.2}). \end{proof}

Let $\{S_1,\ldots, S_\ell\}$ be a  $C^1$ IFS  on a compact set
$X\subset \R^d$. Let $\pi:\Sigma\to \R^d$ be defined as in
(\ref{e-1.1}). By Lemma \ref{lem-4.1}, we have directly

\begin{lem}
\label{lem-4.2} Let $c>1$. Then there exists  $r_0>0$ such that
for any $1\leq i\leq \ell$,  $x\in \Sigma$ and $0<r<r_0$,
$$
B(S_i(\pi x), c^{-1}\sm S_i'(\pi x) \sm r)\subset
S_i\left(B(\pi x, r)\right)\subset B(S_i(\pi x), c\|S_i'(\pi x)\|r).
$$
\end{lem}
\bigskip

Let $\overline{\rho},\underline{\rho}:\;\Sigma\to \R$ be defined by
\begin{equation}
\label{e-4.4} \overline{\rho}(x)=\|S^{\prime}_{x_1}(\pi \sigma x)\|,
\quad \underline{\rho}(x)=\sm S^{\prime}_{x_1}(\pi \sigma x)\sm
\qquad (x=(x_i)_{i=1}^\infty\in \Sigma).
\end{equation}
Let $\P$ be the partition of $\Sigma$ defined as in (\ref{e-1P}).
For $x\in \Sigma$, let $\P(x)$ denote the element in $\P$ which
contains $x$. Then we have

\begin{lem}
\label{lem-4.3} Let $c>1$. Then there exists  $r_0>0$ such that
for any  $z\in \Sigma$ and $0<r<r_0$,
$$
B^\pi(z, c^{-1}\underline{\rho}(z)r)\cap \P(z)\subset
B^{\pi\sigma}(z, r)
\cap\P(z)\subset B^\pi(z,c\overline{\rho}(z)r)\cap \P(z),
$$
where $B^\pi(z,r)$ is defined as in (\ref{e-ball}).
\end{lem}
\begin{proof} Let $z=(z_j)^\infty_{j=1}\in \Sigma$. Taking $i=z_1$ and
$x=\sigma z$ in Lemma \ref{lem-4.2} we obtain
{\small
$$
B(S_{z_1}(\pi \sigma
z),c^{-1}\sm S_{z_1}'(\pi \sigma z)\sm r)\subset S_{z_1}(B(\pi \sigma z,r))\subset
B(S_{z_1}(\pi \sigma z),c\|S_{z_1}'(\pi \sigma z)\|r).
$$
}
That is,
$$
B(\pi z,c^{-1}\underline{\rho}(z)r)\subset S_{z_1}B(\pi \sigma
z,r)\subset B(\pi z,c\overline{\rho}(z)r),
$$
where weuse the fact  $S_{z_1}(\pi \sigma z)=\pi z$,
which can be checked directly from the definition of $\pi$. Thus
we have
{\small
$$
B^\pi(\pi  z,c^{-1}\underline{\rho}(z)r)\cap \P(z)\subset
\pi^{-1}\left(S_{z_1}\left(B(\pi \sigma z,r)\right)\right)\cap
\P(z)\subset B^\pi(z,c\overline{\rho}(z)r)\cap \P(z).
$$
}
At last we show that $\pi^{-1}\left(S_{z_1}\left(B(\pi \sigma z,r)\right)\right)\cap \P(z)= B^{\pi\sigma}(z, r) \cap\P(z)$. To see this, let $y=(y_j)_{j=1}^\infty\in \Sigma$. Then we have
the following equivalent implications.
\begin{eqnarray*}
&\mbox{}&y\in\pi^{-1}\left(S_{z_1}\left(B(\pi \sigma z,r)\right)\right)\cap \P(z)\\
&\mbox{}&\quad \Longleftrightarrow y_1=z_1,\quad \pi y\in S_{z_1}
\left(B(\pi \sigma z,r)\right)\\
&\mbox{}&\quad \Longleftrightarrow y_1=z_1, \quad S_{y_1}(\pi \sigma y)\in
S_{z_1}
\left(B(\pi \sigma z,r)\right)\\
& \mbox{}&\quad \Longleftrightarrow y_1=z_1, \quad \pi \sigma y\in
B(\pi \sigma z,r)\\
&\mbox{}&\quad \Longleftrightarrow y_1=z_1, \quad  y\in
B^{\pi \sigma}(z,r)\\
&\mbox{}&\quad \Longleftrightarrow y\in B^{\pi \sigma}(z,r)\cap \P(z).
\end{eqnarray*}
This finishes the proof of the lemma.
 \end{proof}

\begin{lem}
\label{lem-4.4} Assume that $\{S_i\}_{i=1}^\ell$ is a  weakly
conformal IFS with attractor $K$. Then for any $c>1$, there exists
$D>0$ such that for any $n\in \N$, $u\in \{1,\ldots,\ell\}^n$, and
$x,y\in K$ we have
\begin{equation*}
\label{e-4.5} D^{-1} c^{-n} \|S_u^\prime (x)\|\cdot |x-y| \leq
|S_u(x)-S_u(y)|\leq D c^n \|S_u^\prime (x)\|\cdot |x-y|.
\end{equation*}
and
\begin{equation}
\label{e-4.6}
  D^{-1}c^{-n}\|S_u^\prime(x)\|\leq \mbox{diam}(S_u(K))\leq
Dc^n\|S_u^\prime(x)\|.
\end{equation}
\end{lem}
\begin{proof}
The results were proved in the  conformal case in \cite[Lemma 3.5 and Corollary 3.6]{Fen07}.
A slight modification of that proof works for the  weakly conformal case.
\end{proof}

As a corollary, we have

\begin{cor}
\label{cor-4.5} Under the assumption of Lemma \ref{lem-4.4}, for $\alpha>0$,  there is $r_0>0$ such that for any $0<r<r_0$ and $z\in K$,
there exist $n\in \N$ and  $u\in
\{1,\ldots,\ell\}^n$ such that $S_u(K)\subset B(z,r)$ and
\begin{equation}
\label{e-4.7} |S_u(x)-S_u(y)|\geq r^{1+\alpha} |x-y|\qquad
(x,y\in K).
\end{equation}
\end{cor}
\begin{proof} Denote $a=\inf\{\sm S_i^\prime(x)\sm: \; x\in K, 1\leq
i\leq \ell\}$ and $b=\sup\{\|S_i^\prime(x)\|: \; x\in K, 1\leq i\leq
\ell\}$. Then $0<a\leq b<1$. Choose $c$ so that
\begin{equation}
\label{e-t34}
1<c<b^{\frac{-\alpha}{3(2+\alpha)}}.
\end{equation}
Let $D$ be the constant  in Lemma \ref{lem-4.4} corresponding to $c$.
Take $n_0\in \N$ and $r_0>0$ such that
\begin{equation}
\label{e-4.8'} \left(c^3
b^{\alpha/(2+\alpha)}\right)^{n_0}<D^{-3}ab^{\alpha/(2+\alpha)},\quad
(1+\alpha/2) \cdot \frac{\log r_0}{\log a} =n_0.
\end{equation}

Now fix $z\in K$ and $0<r<r_0$. We shall show that there exist
$n\in \N$ and $u\in \{1,\ldots,\ell\}^n$ such that $S_u(K)\subset
B(z,r)$ and (\ref{e-4.7}) holds. To see this, take
$\omega=(\omega_i)_{i=1}^\infty\in \Sigma$ such that $z=\pi \omega$,
where $\pi$ is defined as in (\ref{e-1.1}). Let $n$ be the unique
integer such that
\begin{equation}\label{e-t36}
\|S^\prime_{\omega_1\cdots \omega_n}(\pi
\sigma^n\omega)\|<r^{1+\alpha/2}\leq \|S^\prime_{\omega_1\cdots
\omega_{n-1}}(\pi \sigma^{n-1}\omega)\|.
\end{equation}
It follows $a^n<r^{1+\alpha/2}\leq b^{n-1}$, which together with
(\ref{e-4.8'}) forces that
\begin{equation}
\label{e-t35}
n>n_0 \quad \mbox{and}\quad c^{3n}<D^{-3} a
r^{-\alpha/2}.
\end{equation}
To see (\ref{e-t35}), we first assume on the contrary that $n\leq n_0$. Then
$$a^n\geq a^{n_0}=a^{(1+\alpha/2)\log r_0/\log \alpha}=r_0^{1+\alpha/2}>r^{1+\alpha/2},$$
which contradicts the fact $a^n<r^{1+\alpha/2}$. Hence $n>n_0$. To see $c^{3n}< D^{-3}a r^{-\alpha/2}$, note that
 \begin{eqnarray*}
c^{3n}r^{\alpha/2}&\leq &c^{3n}b^{(n-1)\alpha/(2+\alpha)}   \qquad (\mbox{ using $r^{1+\alpha/2}\leq b^{n-1}$ })\\
&\leq & \left(c^3b^{\alpha/(2+\alpha)}\right)^n b^{-\alpha/(2+\alpha)}\\
&\leq & \left(c^3b^{\alpha/(2+\alpha)}\right)^{n_0} b^{-\alpha/(2+\alpha)} \qquad (\mbox{ using $n>n_0$ and (\ref{e-t34}) })\\
&\leq & D^{-3} a \quad \qquad (\mbox{by (\ref{e-4.8'})}).
\end{eqnarray*}
This completes the proof of (\ref{e-t35}).
By (\ref{e-4.6}), we have
\begin{eqnarray*}
\mbox{diam} S_{\omega_1\cdots \omega_n}(K)\leq Dc^n
\|S^\prime_{\omega_1\cdots \omega_n}(\pi \sigma^n\omega)\|\leq Dc^n
r^{1+\alpha/2}<r.
\end{eqnarray*}
Since $z\in S_{\omega_1\cdots \omega_n}(K)$, the above inequality
implies $S_{\omega_1\cdots \omega_n}(K)\subset B(z,r)$.
By (\ref{e-4.6}) again, we have
\begin{equation}
\label{e-qq*}
\|S_u'(x)\|\geq D^{-2}c^{-2n}\|S_u'(y)\|,\quad \forall \; u\in \{1,\ldots,\ell\}^n, \;\forall\;x,y\in K.
\end{equation}
 By Lemma
\ref{lem-4.4}, we have for $x,y\in K$,
\begin{eqnarray*}
&\mbox{}& |S_{\omega_1\cdots \omega_n}(x)-S_{\omega_1\cdots
\omega_n}(y)|\\
&\mbox{}& \quad \geq D^{-1}c^{-n}
\|S^\prime_{\omega_1\cdots \omega_n}(x)\|\cdot |x-y|\\
&\mbox{}& \quad\geq D^{-3}c^{-3n}\|S^\prime_{\omega_1\cdots \omega_n}(\pi\sigma^n
\omega)\|\cdot |x-y|\qquad\mbox{ (by (\ref{e-qq*}))}\\
&\mbox{}& \quad\geq D^{-3}c^{-3n}\|S^\prime_{\omega_1\cdots
\omega_{n-1}}(\pi\sigma^{n-1} \omega)\|\sm
S^\prime_{\omega_n}(\pi\sigma^n\omega) \sm \cdot |x-y|\\
&\mbox{}& \quad\geq D^{-3}c^{-3n}a r^{1+\alpha/2} |x-y|\qquad\mbox{ (by (\ref{e-t36}))}\\
&\mbox{}& \quad\geq r^{1+\alpha}|x-y| \qquad\mbox{ (by (\ref{e-t35}))}.
\end{eqnarray*}
Hence the corollary follows by taking $u=\omega_1\cdots \omega_n$.
\end{proof}

\begin{pro}
\label{pro-4.6} Let $\{S_i\}_{i=1}^\ell$ be a $C^1$ IFS with attractor $K$. Assume that $K$ is not a singleton.
Then \begin{itemize}
\item[(i)]
for any $m\in \M_\sigma(\Sigma)$, we have for $m$-a.e.\! $x=(x_i)_{i=1}^\infty\in \Sigma$,
\begin{eqnarray*}
&\mbox{}& \liminf_{n\to \infty}\frac{\log
\text{\rm diam} S_{x_1\ldots x_n}(K)}{n}\geq -\overline{\lambda}(x),\\
&& \limsup_{n\to
\infty}\frac{\log \mbox{\rm diam} S_{x_1\ldots x_n}(K)}{n}\leq -\underline{\lambda}(x),
\end{eqnarray*}
where $\underline{\lambda},\overline{\lambda}$ are defined as in
Definition \ref{de-1.2}. In particular, if $\{S_i\}_{i=1}^\ell$ is
$m$-conformal, then for $m$-a.e.\! $x=(x_i)_{i=1}^\infty\in
\Sigma$,
$$
\lim_{n\to \infty}\frac{\log \mbox{\rm diam} S_{x_1\ldots x_n}(K)}{n}=
-\lambda(x).
$$
\item[(ii)] If $\{S_i\}_{i=1}^\ell$ is weakly conformal, then it is $m$-conformal for each $m\in \M_\sigma(\Sigma)$.
\end{itemize}

\end{pro}
\begin{proof} We first prove (i).
Take $c>1$ small enough so that $c\sup_{x\in \Sigma} \overline{\rho}(x)<1$.
 Let $r_0>0$ be given as in Lemma \ref{lem-4.2}. Let $x=(x_i)_{i=1}^\infty\in \Sigma$. Applying Lemma \ref{lem-4.2} repeatedly, we have
\begin{equation}
\label{e-t37}
 S_{x_1\cdots
x_n}(B(\pi \sigma^n x,r_0))\subset
B(\pi x, c^{n}\overline{\rho}(x)\cdots
\overline{\rho}(\sigma^{n-1}x)r_0).
\end{equation}
Since $\{S_i\}_{i=1}^\ell$ is contractive, there is a constant $k$
such that $$S_{x_{n+1}\cdots x_{n+k}}(K)\subset B(\pi \sigma^n
x,r_0).$$ This together with (\ref{e-t37}) yields
\begin{equation}
\label{e-t29}
\text{diam} S_{x_1\ldots
x_{n+k}}(K)\leq \text{diam} S_{x_1\cdots
x_n}(B(\pi \sigma^n x,r_0)) \leq c^{n}\overline{\rho}(x)\ldots
\overline{\rho}(\sigma^{n-1}x)r_0.
\end{equation}
Since $K$ is not a singleton, there exists $0<r_1<r_0$ such that for each $z\in K$,
 there exists $w\in K$ such that $r_1\leq |z-w|\leq r_0$. Indeed, to obtain
 $r_1$, one chooses an integer $n_0$ large enough such that $\sup_{u\in
 \Sigma_{n_0}} \text{diam} S_u(K)\leq r_0$, then set $$r_1=(1/2)\inf_{u\in
 \Sigma_{n_0}} \text{diam} S_u(K).$$
 For each such pair $(z,w)$, applying (\ref{e-a4.1}) repeatedly yields
$$
\text{diam} S_{x_1\ldots x_{n}}(K)\geq |S_{x_1\ldots
x_{n}}(z)-S_{x_1\ldots x_{n}}(w)|\geq r_1 c^{-n}\prod_{j=1}^n \sm
S_{x_j}^\prime(S_{x_{j+1}\ldots x_n}(z)\sm .
$$
Hence by taking $z=\pi{\sigma^n} x$, we have
\begin{equation}
\label{e-t4.15}
 \text{diam} S_{x_1\ldots x_{n}}(K)\geq r_1
c^{-n}\underline{\rho}(x)\ldots \underline{\rho}(\sigma^{n-1}x).
\end{equation}
Denote
\begin{eqnarray*}
&\mbox{}& g_*(x)=\liminf_{n\to \infty}\frac{\log \text{diam}
S_{x_1\ldots x_n}(K)}{n}\quad\mbox{ and }\\
&&g^*(x)=\limsup_{n\to
\infty}\frac{\log \text{diam} S_{x_1\ldots x_n}(K)}{n}.
\end{eqnarray*}
 It is clear
that $g_*(x)=g_*(\sigma x)$ and $g^*(x)=g^*(\sigma x)$. Let $\I$
denote the $\sigma$-algebra $\{B\in \B(\Sigma):\; \sigma^{-1}B=B\}$.
Then by (\ref{e-t4.15}), the Birkhoff ergodic theorem, and Theorem 34.2
in \cite{Bil95}, we have for $m$-a.e.\! $x\in \Sigma$,
\begin{equation}
\label{e-t291}
\begin{split}
g_*(x)=\E_m(g_*|\I)(x)&\geq \E_m\left(\lim_{n\to \infty}\frac{-n\log c+\sum_{i=0}^{n-1}\log \underline{\rho}\circ \sigma^{-i}}{n}\big|\I\right)(x)\\
&=-\log c+\lim_{n\to \infty}\frac{1}{n}\sum_{i=0}^{n-1}\E_m(\log \underline{\rho}\circ \sigma^{-i}|\I)(x)\\
&=-\log c+\E_m(\log \underline{\rho}|\I)(x)\\
\end{split}
\end{equation}
and similarly by (\ref{e-t29}),
\begin{equation}
\label{e-t292}g^*(x)\leq \log c+\E_m(\log \overline{\rho}|\I)(x).
\end{equation}
For $p\in \N$, write $A_p(x)=\log \sm S^\prime_{x_1\cdots x_p}(\pi \sigma^p x)\sm$
and $A^*_p(x)=\log \| S^\prime_{x_1\cdots x_p}(\pi \sigma^p
x)\|$.
Consider the IFS $\{S_{i_1\ldots i_p}:\;1\leq i_j\leq \ell, 1\leq
j\leq p\}$ rather than $\{S_i\}_{i=1}^\ell$. Then (\ref{e-t291}) and (\ref{e-t292}) can be replaced by
\begin{equation*}
g_*(x)\geq -\log c+\frac{1}{p}\E_m(A_p|\I_p)(x),\qquad g^*(x)\leq \log c+\frac{1}{p}\E_m(A_p^*|\I_p)(x),
\end{equation*}
where $\I_p:=\{B\in \B(\Sigma):\; \sigma^{-p}B=B\}$. Taking the conditional expectation with respect to $\I$ in the above inequalities and noting that
$g_*$, $g^*$ are $\sigma$-invariant, we obtain
\begin{equation}\label{e-t293}
g_*(x)\geq -\log c+\frac{1}{p}\E_m(A_p|\I)(x),\qquad g^*(x)\leq \log c+\frac{1}{p}\E_m(A_p^*|\I)(x).
\end{equation}
Since $A_p(x)$ is sup-additive (i.e., $A_{p+q}(x)\geq A_p(x)+ A_q(\sigma^p x)$) and $A^*_p(x)$ is sub-additive
(i.e., $A^*_{p+q}(x)\leq A^*_p(x)+ A^*_q(\sigma^p x)$), by Kingman's sub-additive ergodic theorem (cf. \cite{Wal-book}), we have
\begin{equation}
\label{e-t294}
\lim_{p\to \infty}  A_p(x)/p=-\overline{\lambda}(x),\qquad \lim_{p\to \infty} A^*_p(x)/p=-\underline{\lambda}(x)
\end{equation}
almost everywhere and in $L^1$. Hence letting $c\to 1$ and $p\to
\infty$ in (\ref{e-t293}) and using Theorem 34.2  in \cite{Bil95},
we obtain that $g_*(x)\geq -\overline{\lambda}(x)$ and $g^*(x)\leq
-\underline{\lambda}(x)$ almost everywhere. This finishes the proof
of (i).

To see (ii), assume that $\{S_i\}_{i=1}^\ell$ is weakly conformal
and $m\in \M_\sigma(\Sigma)$. Then $|A_p(x)-A^*_p(x)|/p$ converges
to $0$ uniformly as $p$ tends to infinity. This together with
(\ref{e-t294}) yields $\overline{\lambda}(x)=\underline{\lambda}(x)$
for $m$-a.e.\! $x\in \Sigma$. This proves (ii).
\end{proof}

\section{Estimates for local dimensions of invariant measures for $C^1$ IFS}\label{S5}

In this section, we prove a general version of Theorem \ref{thm-1.1},
which is also needed in the proof of Theorem \ref{thm-1.3}.  Let
$\{T_i\}_{i=1}^\ell$ be a $C^1$ IFS on $\R^d$, and
$\{S_i\}_{i=1}^\ell$ a $C^1$ IFS on $\R^k$. Let $\phi: \Sigma\to
\R^d$ and $\pi: \Sigma\to \R^k$ denote the canonical projections
associated with $\{T_i\}_{i=1}^\ell$ and $\{S_i\}_{i=1}^\ell$
respectively. Let $\eta$ and $\xi$ be two partitions of $\Sigma$
defined respectively by
$$\eta=\{\phi^{-1}(z):\; z\in \R^d\} ,\qquad \xi=\sigma^{-1}\eta.$$

Let $\P$ be the partition of $\Sigma$ given as in (\ref{e-1P}) and
let $\overline{\rho}(x),\underline{\rho}(x)$ be defined as in
(\ref{e-4.4}). Applying  Lemma \ref{lem-4.3} to the IFS
$\{S_i\}_{i=1}^\ell$, we have  for any $c>1$ there exist
$0<\delta<c-1$ and $r_0>0$ such that for any $r\in (0, r_0)$ and
$x\in \Sigma$,
\begin{equation}
\label{e-5.1}
B^\pi(x,(c-\delta)^{-1}\underline{\rho} (x) r)\cap \P(x)\subset B^{\pi\sigma}(x,r)\cap \P(x)\subset B^\pi(x,(c-\delta)\overline{\rho}(x) r)\cap \P(x).
\end{equation}
The following technical proposition is substantial in our proof.

\begin{pro}
\label{pro-5.1}
Let $m\in \M_\sigma(\Sigma)$ and $c>1$. Let $\delta,r_0$ be given as above. Then there exists  $\Lambda\subset\Sigma$ with $m(\Lambda)=1$ such that for all $x\in \Lambda$ and  $r\in (0, r_0)$,
\begin{equation}
\label{e-5.2}
\frac{m_x^{\eta}(B^\pi(x,c\overline{\rho} (x) r) \cap \P(x)) }{
m_{\sigma x}^{\eta}(B^\pi(\sigma x, r))}\geq f(x)\cdot
\frac{m_x^{\xi}(B^{\pi\sigma}(x, r) \cap \P(x)) }{
m_{x}^{\xi}(B^{\pi\sigma}(x, r))}
\end{equation}
and
\begin{equation}
\label{e-5.3} \frac{m_x^{\eta}(B^{\pi}(x,c^{-1}\underline{\rho} (x)r) \cap \P(x)) }{ m_{\sigma x}^{\eta}(B^\pi(\sigma x, r))}\leq f(x)\cdot
\frac{m_x^{\xi}(B^{\pi\sigma}(x,(1-c\delta/2)r)
\cap \P(x)) }{ m_{ x}^{\xi}(B^{\pi\sigma}(x,(1-c\delta/2)r))},
\end{equation}
where $f:=\sum_{A\in \P}
\chi_A\frac{\E_m(\chi_A|\phi^{-1}\gamma)}{\E_m(\chi_A|\sigma^{-1}\phi^{-1}\gamma)}$,
$\gamma=\B(\R^d)$.
\end{pro}

\begin{proof} Write $R_{t,x}(z)=T_{x_1}^{-1}B(T_{x_1} z, t)$ for $t>0$,
$x=(x_i)_{i=1}^\infty\in \Sigma$ and $z\in \R^d$. It is direct to
check that
\begin{equation}
\label{e-5.4'} \sigma^{-1} \phi^{-1}R_{t,x}(\phi \sigma x)\cap
\P(x)=B^\phi(x,t)\cap \P(x). \end{equation}
 Hence for
$m$-a.e.\! $x$,
\begin{eqnarray*}
\frac{m(\phi^{-1}R_{t,x}(\phi \sigma x))}{m(B^\phi(x,t))}&=& \frac{m(B^\phi(x,t)\cap
\P(x))}{m(B^\phi(x,t))}\cdot
\frac{m(\phi^{-1}R_{t,x}(\phi \sigma x))}{m(B^\phi(x,t)\cap \P(x))}\\
&=& \frac{m(B^\phi(x,t)\cap \P(x))}{m(B^\phi(x,t))}\cdot \frac{m(\sigma^{-1}\phi^{-1}R_{t,x}(\phi \sigma
x))}{m(\sigma^{-1} \phi^{-1}R_{t,x}(\phi \sigma x)\cap \P(x))}.
\end{eqnarray*}
Letting $t\to 0$ and applying Proposition \ref{pro-2.5} and Remark \ref{rem-2.6}, we have
\begin{equation}
\label{e-5.4} \lim_{t\to 0}\frac{m(\phi^{-1}R_{t,x}(\phi \sigma
x))}{m(B^\phi(x,t))}=\sum_{A\in \P}
\chi_A(x)\frac{\E_m(\chi_A|\phi^{-1}\gamma)(x)}{\E_m(\chi_A|\sigma^{-1}\phi^{-1}\gamma)(x)}=:f(x).
\end{equation}
for $m$-a.e.\! $x$.
Let $\widetilde{\Lambda}$ denote the set of $x\in\Sigma$ such that
the following properties (1)-(4) hold:

(1) $\displaystyle\lim_{t\to 0}\frac{m(B^\phi(x,t)\cap
\P(x))}{m(B^\phi(x,t))}=\sum_{A\in \P}\chi_A
\E_m(\chi_A|\phi^{-1}\gamma)(x)>0$.

(2) $\displaystyle\lim_{t\to
0}\frac{m(\sigma^{-1}\phi^{-1}R_{t,x}(\phi \sigma x)\cap
\P(x))}{m(\sigma^{-1}\phi^{-1} R_{t,x}(\phi\sigma x))}=\sum_{A\in
\P}\chi_A \E_m(\chi_A|\sigma^{-1}\phi^{-1}\gamma)(x)>0$.

(3) For all $q\in {\Bbb Q}^+$,
{\small
\begin{eqnarray*}
&& m^\eta_x(B^\pi(x,q)\cap \P(x))\geq \limsup_{t\to 0}
\frac{ m\left(B^\pi(x,q)\cap \P(x)\cap B^\phi(x,t)\right)}{m\left(B^\phi(x,t)\right)},\\
&& m^\eta_x(U^\pi(x,q)\cap \P(x))\leq \liminf_{t\to 0}\frac{ m\left(B^\pi(x,q)\cap \P(x) \cap
B^\phi(x,t)\right)}{m\left(B^\phi(x,t)\right)},\\
&& m^\xi_x(B^{\pi\sigma}(x,q)\cap \P(x))\geq \limsup_{t\to 0}\frac{ m\left(
B^{\pi\sigma}(x,q)\cap \P(x) \cap \sigma^{-1}\phi^{-1}R_{t,x}(\phi\sigma x)\right)}
{m\left(\sigma^{-1}\phi^{-1}R_{t,x}(\phi \sigma x)\right)},\\
&&m^\xi_x(U^{\pi\sigma}(x,q)\cap \P(x))\leq \liminf_{t\to 0}\frac{
 m\left(B^{\pi\sigma}(x,q)\cap \P(x) \cap \sigma^{-1}\phi^{-1}R_{t,x}(\phi\sigma x)\right)}
 {m\left(\sigma^{-1}\phi^{-1}R_{t,x}(\phi \sigma x)\right)},\\
\end{eqnarray*}
}
where $U^\pi(x,q):=\pi^{-1}U(\pi x,q)$, $U^{\pi\sigma}(x,q):=\sigma^{-1}\pi^{-1}U(\pi\sigma x,q)$ and $U(z,q)$ denotes the open ball in $\R^k$ of radius $q$ centered at $z$.

(4) $\displaystyle\lim_{t\to 0}\frac{m(\phi^{-1}R_{t,x}(\phi \sigma x))}{m(B^\phi(x,t))}=f(x)$.\\
Then we have  $m(\widetilde{\Lambda})=1$ by Proposition
\ref{pro-2.5}, Lemma \ref{lem-2.7}, Remarks \ref{rem-2.6},
\ref{rem-2.8} and (\ref{e-5.4}).

Now let $\Lambda=\widetilde{\Lambda}\cap
\sigma^{-1}\widetilde{\Lambda}$. Then $m(\Lambda)=1$.  Fix $x\in
\Lambda$ and $r\in (0,r_0)$. Let $q_1\in {\Bbb Q}^+\cap (r,
cr/(c-\delta))$. Choose $q_2, q_3\in {\Bbb Q}^+$ such that
$q_1<q_2<cr/(c-\delta)$ and
$q_2(c-\delta)\overline{\rho}(x)<q_3<c\overline{\rho}(x)r$. By
(\ref{e-5.1}), we have $B^\pi(x,q_3) \cap \P(x)\supset
B^{\pi\sigma}(x,q_2) \cap \P(x).$ It together
with (\ref{e-5.4'}) yields
\begin{equation}
\label{e-5.6'} B^\pi(x,q_3) \cap
\P(x)\cap B^\phi(x,t))\supset B^{\pi\sigma}(x,q_2) \cap \P(x)\cap \sigma^{-1}\phi^{-1}R_{t,x}(\phi
\sigma x)).
\end{equation}
Hence we  have
\begin{eqnarray*}
&\mbox{}& \frac{m_x^{\eta}(B^\pi(x,c\overline{\rho} (x) r) \cap \P(x)) }{
m_{\sigma x}^{\eta}(B^\pi(\sigma x, r))}\\
&\geq&
\frac{m_x^{\eta}(B^\pi(x,q_3) \cap \P(x)) }{
m_{\sigma x}^{\eta}(U^{\pi\sigma}(x,q_1))}
\\
&\geq &
\frac
{
\limsup_{t\to 0}m(B^\pi(x,q_3) \cap \P(x)\cap B^\phi(x,t))/m(B^\phi(x,t))
}
{\liminf_{t\to 0}m(B^\pi(\sigma x,q_1)\cap \phi^{-1}R_{t,x}(\phi \sigma x))/m(\phi^{-1}R_{t,x}(\phi \sigma x))
}\\
&& \qquad (\mbox{by Lemma \ref{lem-2.7} and Remark \ref{rem-2.8}})\\
 &\geq & \lim_{t\to 0} \frac{m(\phi^{-1}R_{t,x}(\phi \sigma
x))}{m(B^\phi(x,t))}\cdot \limsup_{t\to 0}\frac{
m(B^\pi(x,q_3) \cap \P(x)\cap B^\phi(x,t)) } {
m(\sigma^{-1}B^{\pi}( \sigma x,q_1) \cap
\sigma^{-1}\phi^{-1}R_{t,x}(\phi \sigma x))
}\\
&= & \lim_{t\to 0} \frac{m(\phi^{-1}R_{t,x}(\phi \sigma
x))}{m(B^\phi(x,t))}\cdot \limsup_{t\to 0}\frac{
m(B^\pi(x,q_3) \cap \P(x)\cap B^\phi(x,t)) } {
m(B^{\pi\sigma}(x,q_1) \cap
\sigma^{-1}\phi^{-1}R_{t,x}(\phi \sigma x))
}.\\
\end{eqnarray*}
Denote
\begin{eqnarray*}
&\mbox{}&X_t:=m(B^{\pi\sigma}(x,q_2) \cap \P(x)\cap \sigma^{-1}\phi^{-1}R_{t,x}(\phi\sigma x)),\\
&\mbox{}&Y_t:=m(B^{\pi\sigma}( x,q_1) \cap \sigma^{-1}\phi^{-1}R_{t,x}(\phi \sigma x)),\\
&\mbox{}&Z_t:=m(\sigma^{-1}\phi^{-1}R_{t,x}(\phi\sigma x)).
\end{eqnarray*}
Using the property (4), we have

\begin{eqnarray*}
&\mbox{}& \frac{m_x^{\eta}(B^\pi(x,c\overline{\rho} (x) r) \cap \P(x)) }{
m_{\sigma x}^{\eta}(B^\pi(\sigma x, r))}\\
&\geq & f(x)\cdot
\limsup_{t\to 0}\frac{
m(B^\pi(x,q_3) \cap \P(x)\cap B^\phi(x,t))
}
{
m(B^{\pi\sigma}( x,q_1) \cap \sigma^{-1}\phi^{-1}R_{t,x}(\phi \sigma x))
}\\
&\geq & f(x)\cdot
\limsup_{t\to 0}X_t/Y_t\qquad (\mbox{by (\ref{e-5.6'})})\\
&\geq & f(x)\cdot
\limsup_{t\to 0}\frac{X_t/Z_t}{Y_t/Z_t}
\geq  f(x)\cdot
\frac{\liminf_{t\to 0}X_t/Z_t}
{\limsup_{t\to 0} Y_t/Z_t
}\\
&\geq & f(x)\cdot
\frac{m^\xi_x(U^{\pi\sigma}( x,q_1) \cap \P(x))
}
{m^\xi_x(B^{\pi\sigma}( x,q_1))
}\quad (\mbox{by Lemma \ref{lem-2.7} and Remark \ref{rem-2.8}})\\
&\geq & f(x)\cdot \frac{m^\xi_x(B^{\pi\sigma}(x, r) \cap \P(x)) } {m^\xi_x(B^{\pi\sigma}( x,q_1))
}.
\end{eqnarray*}

Letting $q_1\downarrow r$, we obtain (\ref{e-5.2}). (\ref{e-5.3})
follows from an analogous argument. \end{proof}

Let $(\phi,\pi)$ denote the map $\Sigma\to \R^d\times\R^k$,
$x\mapsto (\phi x,\pi x)$. It is easy to see that $(\phi,\pi)$ is
the canonical projection w.r.t. the direct product of
$\{T_i\}_{i=1}^\ell$ and $\{S_i\}_{i=1}^\ell$. In the following we
give a general version of Theorem \ref{thm-1.1}.

\begin{thm}
\label{thm-5.2}
Let $m\in \M_\sigma(\Sigma)$. Then for $m$-a.e.\! $x\in \Sigma$, we have
\begin{eqnarray}
&& \limsup_{r\to 0} \frac{\log m^\eta_x(B^\pi(x,r))}{\log r}\leq \frac{\E_m(g|\I)(x)}{- \underline{\lambda}(x)} \quad {\mbox{and}} \label{e-5.5}\\
&&\liminf_{r\to 0} \frac{\log m^\eta_x(B^\pi(x,r))}{\log
r}\geq \frac{\E_m(g|\I)(x)}{- \overline{\lambda}(x)},
\label{e-5.6}
\end{eqnarray}
where \begin{eqnarray*}
g&:=&{\bf I}_m(\P|\sigma^{-1}\phi^{-1}\B(\R^d))-{\bf I}_m(\P|\phi^{-1}\B(\R^d))\\
&\mbox{}&\;\; +{\bf I}_m(\P|(\phi,\sigma)^{-1}\B(\R^d\times \R^k))-
{\bf I}_m(\P|\sigma^{-1}(\phi,\pi)^{-1}\B(\R^d\times\R^k)),
\end{eqnarray*}
and $\overline{\lambda}(x)$,\;$\underline{\lambda}(x)$ denote the
upper and lower Lyapunov exponents of $\{S_i\}_{i=1}^\ell$ at $x$
(see Definition \ref{de-1.2}). In particular, if
$\{S_i\}_{i=1}^\ell$ is  $m$-conformal, we have
$$
\lim_{r\to 0} \frac{\log m^\eta_x(B^\pi(x,r))}{\log
r}=\frac{h_{(\phi, \pi)}(\sigma,m, x)-h_\phi(\sigma,m,x)}{ \lambda(x)}.
$$
\end{thm}
\begin{proof}
 It suffices to prove (\ref{e-5.5}) and (\ref{e-5.6}). For short we only prove (\ref{e-5.5}). The proof of
(\ref{e-5.6}) is analogous.

We first prove the following inequality
 \begin{equation}
 \label{e-5.7}
  \limsup_{r\to 0} \frac{\log m^\eta_x(B^\pi(x,r))}{\log r}\leq \frac{\E_m(g|\I)(x)}
 {\E_m(\log \overline{\rho}|\I)(x)}\qquad \mbox{$m$-a.e.},
 \end{equation}
 where $\overline{\rho}(x)=\|S'_{x_1}(\sigma x)\|$ for $x=(x_i)_{i=1}^\infty$. To see it, let $c>1$ so that $$c\sup_{x\in \Sigma} \overline{\rho}(x)<1.$$
  Let $
r_0$ and $f$ be given as in Proposition \ref{pro-5.1}. For $n\in \N$
and $x\in \Sigma$, define
$$\overline{\rho}_n(x)=\overline{\rho}(x)\overline{\rho}(\sigma x)\cdots \overline{\rho}(\sigma^{n-1}x).$$
 Write
\begin{equation*}
\begin{split}
H_n(x)&:=\log \frac{m_x^\eta\left(B^\pi(x, c^n\overline{\rho}_n(x)r_0)\right)}
{m_{\sigma x}^\eta\left(B^\pi(\sigma x, c^{n-1}\overline{\rho}_{n-1}(\sigma x)r_0)\right)}, \\
G_n(x)&:=\log \frac{m_x^\eta\left(B^\pi(x, c^n\overline{\rho}_n(x)r_0) \cap \P(x)\right)}
{m_{x}^\eta\left(B^\pi(x, c^{n}\overline{\rho}_{n}(x)r_0)\right)},\\
W_n(x)&:=\log \frac{m_x^\xi\left(B^{\pi\sigma}(x, c^{n-1}\overline{\rho}_{n-1}(\sigma x)r_0) \cap \P(x)\right)}
{m_{x}^\xi\left(B^{\pi\sigma}(x, c^{n-1}\overline{\rho}_{n-1}(\sigma x)r_0)\right)}.\\
\end{split}
\end{equation*}
Then by Proposition \ref{pro-5.1} we have for $m$-a.e.\! $x$,
$H_n(x)+G_n(x)\geq \log f(x) +W_n(x)$, that is,
$$H_{n}(x)\geq \log f(x)-G_{n}(x)+W_{n}(x).$$
However
$$\log m_x^\eta\left(B^\pi(x, c^n\overline{\rho}_n(x)r_0) \right)=
\sum_{j=0}^{n-1}H_{n-j}(\sigma^jx)+
\log m^\eta_{\sigma^n x}\left(B^\pi(\sigma^n x, r_0)\right).$$
Hence  for $m$-a.e.\! $x$,
\begin{eqnarray*}\label{e-5.8}
\frac{\log m_x^\eta\left(B^\pi(x, c^n\overline{\rho}_n(x)r_0) \right)}{n}&\geq& \frac{1}{n}\sum_{j=0}^{n-1}\left[\log f(\sigma^jx)-G_{n-j}(\sigma^j x)+W_{n-j}(\sigma^{j}x)\right]
\nonumber \\
&& \mbox{} +\frac{1}{n}
\log m^\eta_{\sigma^n x}\left(B^\pi(\sigma^n x, r_0)\right).
\end{eqnarray*}
Note that by Proposition \ref{pro-2.5},
\begin{eqnarray*}
&\mbox{}& G_n\to G:=-{\bf I}_m(\P|\hat{\eta}\vee \pi^{-1}\B(\R^k)),\\
&&       W_n\to W:=-{\bf I}_m(\P|\sigma^{-1}\hat{\eta}\vee \sigma^{-1}\pi^{-1}\B(\R^k))
\end{eqnarray*}
pointwise and in $L^1$. By Lemma \ref{lem-3.15} and Proposition \ref{pro-2.9}, we have for $m$-a.e.\! $x$,
\begin{eqnarray*}
\liminf_{n\to \infty}\frac{\log
m_x^\eta\left(B^\pi(x, c^n\overline{\rho}_n(x)r_0)
\right)}{n}&\geq&
{\E_m((\log f-G+W)|\I)(x)}\\
&=&\E_m(g|\I)(x).
\end{eqnarray*}
In the meantime, by Birkhoff ergodic Theorem, we have $$ \lim_{n\to
\infty} \frac{1}{n}\log (c^n\overline{\rho}_n(x)r_0) =\log c
+\E_m(\log \overline{\rho}|\I)(x) \quad \mbox{$m$-a.e.}$$ Hence we
have
 \begin{eqnarray*}
  \limsup_{r\to 0}\frac{\log m_x^\eta\left(B^\pi(x,r)\right)}{\log r}
  &=&\limsup_{n\to \infty}\frac{\log m_x^\eta\left(B^\pi(x, c^n\overline{\rho}_n(x)r_0) \right)}
  {\log (c^n\overline{\rho}_n(x)r_0)}\\
  &\leq &
\frac{\E_m(g|\I)(x)}{\log c +\E_m(\log \overline{\rho}|\I)(x)}.
  \end{eqnarray*}
Taking $c\to 1$, we obtain (\ref{e-5.7}).

Let $q\in \N$. Considering the IFS $\{T_{i_1\ldots i_q}: \; 1\leq i_j\leq \ell,\; 1\leq j\leq q\}$ and
$\{S_{i_1\ldots i_q}: \; 1\leq i_j\leq \ell,\; 1\leq j\leq q\}$, analogous to (\ref{e-5.7}) we have
\begin{equation}
 \label{e-5.9}
  \limsup_{r\to 0} \frac{\log m^\eta_x(B^\pi(x,r))}{\log r}\leq \frac{\E_m(g_q|\I)(x)}
 {\E_m(\log h_q|\I)(x)},
 \end{equation}
where
\begin{eqnarray*}
g_q&:=&{\bf I}_m(\P_0^{q-1}|\sigma^{-q}\phi^{-1}\B(\R^d))-{\bf I}_m(\P_0^{q-1}|\phi^{-1}\B(\R^d))\\
&\mbox{}&\;\; +{\bf I}_m(\P_{0}^{q-1}|(\phi,\pi)^{-1}\B(\R^d\times
\R^k))- {\bf
I}_m(\P_{0}^{q-1}|\sigma^{-q}(\phi,\pi)^{-1}\B(\R^d\times\R^k))
\end{eqnarray*}
and $h_q(x):=\|S^\prime_{x_1\ldots x_q}(\sigma^{q}x)\|$ for
$x=(x_i)_{i=1}^{\infty}$.

Due to (\ref{e-tt}), we have $\E_m(g_q|\I)(x)=q\E_m(g|\I)(x)$. It is
easily seen that  $h_q(x)$ is sub-multiplicative in the sense that
$h_{p+q}(x)\leq h_{p}(x)h_q(\sigma^px)$. Thus by Kingman
sub-additive ergodic theorem (cf. \cite{Wal-book}), we have
  $$\lim_{q\to \infty} \frac{1}{q}\E_m(\log h_q|\I)(x)=-\underline{\lambda}(x)\quad\mbox{ for $m$-a.e.\! $x$}.$$
  Hence letting $q\to \infty$ in (\ref{e-5.9}) we obtain (\ref{e-5.5}). This finishes the proof of Theorem \ref{thm-5.2}.
\end{proof}

\begin{proof}[Proof of Theorem \ref{thm-1.1}]
In Theorem \ref{thm-5.2}, we take $T_i(x)=x/2$ for all $1\leq i\leq \ell$ to obtain Theorem \ref{thm-1.1}. To see it, we know that  the attractor of $\{T_i\}_{i=1}^\ell$ is just the singleton $\{0\}$.  Hence   $\eta$ is the trivial partition $\{\Sigma,\emptyset\}$ of $\Sigma$, and thus we have $m^\eta_x\equiv m$.
\end{proof}

\section{Proofs of Theorem \ref{thm-1.3} and Theorem \ref{thm-1.4}}\label{S6}

\subsection{Proof of Theorem \ref{thm-1.3}}
 Let $\Phi=\{S_i\}_{i=1}^\ell$ be the direct product of $k$ $C^1$ IFS $\Phi_1,\ldots, \Phi_k$, which
are defined respectively on compact $X_i\subset \R^{q_i}$ ($i=1,\ldots,k$).
 For each $i$, let $\Gamma_i$ denote the canonical projection w.r.t.  $\Phi_i$, and
 let $\lambda_i(x)$ denote the Lyapunov exponent of $\Phi_i$ at $x$ provided  it exists.

Let $m\in \M_\sigma(\Sigma)$.   Assume that $\Phi_1,\ldots,\Phi_k$ are $m$-conformal. Let
$\Omega$ denote the collection of all permutations of $\{1,\ldots,k\}$. For $\tau\in \Omega$, we denote
$$
\Lambda_\tau:=
\left\{x\in \Sigma: \; \lambda_i(x) \mbox{ exists for all $i$,\; }
\lambda_{\tau(1)}(x)\leq \lambda_{\tau(2)}(x)\leq \cdots \leq \lambda_{\tau(k)}(x)\right\}.
$$
Then $m\left(\bigcup_{\tau\in \Omega}\Lambda_\tau\right)=1$.
Let $\pi$ denote the canonical projection associated with the IFS $\Phi$.  In the following we show that the local dimension $d(m\circ \pi^{-1},\pi x)$ exists for $m$-a.e.\! $x\in \Sigma$.

Without loss of generality we only show that $d(m\circ \pi^{-1},\pi x)$ exists for $m$-a.e.\! $x\in \Lambda_e$,
 where $e$ denotes the identity in $\Omega$.
 Here we may assume $m(\Lambda_e)>0$.
 For  other $\Lambda_\tau$'s, the proof is essentially identical under a change of coordinates.

For $i=1,\ldots, k$, let $\pi_i$ denote the canonical projection w.r.t. $\Phi_1\times\cdots\times\Phi_i$.
It is clear that $\pi=\pi_k$.  Bear in mind  that
$$\lambda_1(x)\leq \lambda_2(x)\leq \cdots\leq \lambda_k(x)\qquad (x\in \Lambda_e).$$

For $i=1,\ldots, k$, we use $\{m^i_x\}$ to denote the family of conditional measures $\{m^{\eta_i}_x\}$ of $m$ associated with
the partition $$\eta_i=\left\{\pi_i^{-1}(z): \; z\in \prod_{t=1}^i\R^{q_t}\right\}.$$ For convenience,  we use
$\{m^0_x\}$ denote the family of conditional measures of $m$ with the trivial partition $\{\Sigma,\emptyset\}$.
It is clear that $m^0_x=m$ for all $x\in \Sigma$.

For $i=1,\ldots, k$, we give a metric $d_i$ on $\prod_{t=1}^i\R^{q_t}$ by
$$
d_i((z_1,\ldots, z_i), (w_1,\ldots, w_i))=\sup_{1\leq t\leq i}
|z_t-w_t|_{\R^{q_t}}.
$$
and define $d=d_k$.
We claim that for any $x\in \Lambda_e$ and $\epsilon>0$,
\begin{equation}
\label{e-t301} \eta_i(x)\cap \P_0^n(x)\subset  B^\pi(x, e^{-n(\lambda_{i+1}(x)-\epsilon)})
\end{equation}
when $n$ is large enough. Here $B^\pi(x,r)$ is defined as in (\ref{e-ball}). To see the claim, let $x\in \Lambda_e$ and $y\in
\eta_i(x)$. Then $\pi_i y=\pi_i x$. Thus $$d(\pi y,\pi x)=\sup_{1\leq
t\leq k}|\Gamma_ty,\Gamma_tx|_{\R^{q_t}}=\sup_{i+1\leq t\leq
k}|\Gamma_ty,\Gamma_tx|_{\R^{q_t}}.$$
 Since $y\in \P_0^n(x)$ and
$\lambda_{i+1}(x)\leq\ldots\leq \lambda_k(x)$, by Proposition
\ref{pro-4.6}, we have  $$d(\pi y,\pi x)\leq
e^{-n(\lambda_{i+1}(x)-\epsilon)}$$ when $n$ is large enough, and
(\ref{e-t301}) follows.

For $i=0,1,\ldots, k$ and $x\in \Sigma$, denote
$$
h_i(x)=\lim_{n\to \infty}\frac{-\log m^i_x(\P^n_0(x))}{n+1}
$$
provided that the limit exists.
By Proposition \ref{pro-3.16},
\begin{equation}
\label{e-h}
h_i(x)=h(\sigma,m,x)-h_{\pi_i}(\sigma,m,x)\quad \mbox{ for $m$-a.e.\! }x\in \Sigma.
\end{equation}
For $i=0,1,\ldots, k-1$ and $x\in \Sigma$, denote
$$
\vartheta_i(x)=\liminf_{r\to 0}\frac{\log
m^i_x(B^{\Gamma_{i+1}}(x,r))}{\log r}.
$$
By Theorem \ref{thm-5.2} and (\ref{e-h}), we have
\begin{equation}
\label{e-h'} \vartheta_i(x)=\frac{
h_{\pi_{i+1}}(\sigma,m,x)-h_{\pi_i}(\sigma,m,x)} {\lambda_{i+1}(x)}=
\frac{h_i(x)-h_{i+1}(x)}{\lambda_{i+1}(x)}
\end{equation}
 for $m$-a.e.\!
$x\in \Sigma$.

For $i=0,1,\ldots, k$ and $x\in \Sigma$, define
$$
\overline{\delta}_i(x)=\limsup_{r\to 0}\frac{\log m^i_x(B^\pi(x,r))}{\log r},\quad
\underline{\delta}_i(x)=\liminf_{r\to 0}\frac{\log m^i_x(B^\pi(x,r))}{\log r}.
$$
We claim that
\begin{itemize}
%\item[(i)]  $\overline{\delta}_i(x)=\overline{\delta}_i$ and  $\underline{\delta}_i(x)=\underline{\delta}_i$ for %$m$-a.e. $x$ and   $i=0,i,\ldots,  k$;

\item[(C1)] $\overline{\delta}_{k}(x)=\underline{\delta}_{k}(x)=0$ for all $x\in \Sigma$.
\item[(C2)]
 $h_{i}(x)-h_{i+1}(x)\geq \lambda_{i+1}(\overline{\delta}_i(x)-\overline{\delta}_{i+1}(x))$ for $m$-a.e.\! $x\in \Lambda_e$ and  $i=0,1\ldots, k-1$;
\item[(C3)]
 $\underline{\delta}_{i+1}(x)+\vartheta_i(x)\leq \underline{\delta}_i(x)$ for $m$-a.e.\! $x\in \Lambda_e$ and $i=0,1\ldots, k-1$;
\end{itemize}

It is easy to see that (C1)-(C3) together with (\ref{e-h})-(\ref{e-h'}) force that for $m$-a.e.\! $x\in \Lambda_e$, $\underline{\delta}_i(x)=\overline{\delta}_i(x)$ (we denoted the common value as $\delta_i(x)$) for $i=0,\ldots,k$ and, furthermore
\begin{equation}
d(m\circ\pi^{-1},\pi x)=\delta_0(x)=\sum_{i=0}^{k-1} \vartheta_i(x)=
\sum_{i=0}^{k-1} \frac{h_i(x)-h_{i+1}(x) }{\lambda_{i+1}(x)}.
\end{equation}
which is the desired result in Theorem \ref{thm-1.3}. In the
following we prove (C1)-(C3) respectively.

\bigskip

\begin{proof}[Proof of (C1)]  Since $\eta_k=\left\{\pi^{-1}(z): z\in \prod_{t=1}^k\R^{q_t}\right\}$, we have
$$m^k_x(B^\pi(x,r))=m^k_x(\eta_k(x))=1$$ for all $x\in \Sigma$. Thus $\overline{\delta}_k(x)=\underline{\delta}_k(x)=0$ for all $x\in \Sigma$.
\end{proof}

\bigskip

\begin{proof}[Proof of (C2)]  We give a proof by contradiction, which is  modified from \cite[\S10.2]{LeYo85}.
Assume that (C2) is not true. Then there exists $0\leq i\leq k$ such
that
$$
h_i(x)-h_{i+1}(x)<\lambda_{i+1}(x) (\overline{\delta}_i(x)-\overline{\delta}_{i+1}(x))
$$
on a subset of $\Lambda_e$ with positive measure. Hence there exist $\alpha>0$ and real numbers $h_i,h_{i+1},\lambda_{i+1}, \overline{\delta}_i,\overline{\delta}_{i+1}$ with $\lambda_{i+1}>0$ such that
\begin{equation}
\label{e-ass}
h_i-h_{i+1}<\lambda_{i+1} (\overline{\delta}_i-\overline{\delta}_{i+1})-\alpha
\end{equation}
and for any $\epsilon>0$, there exists $B_\epsilon\subset \Lambda_e$ with $m(B_\epsilon)>0$ so that for $x\in B_\epsilon$,
$$|h_i(x)-h_i|<\epsilon/2, \quad |h_{i+1}(x)-h_{i+1}|<\epsilon/2, \quad |\lambda_{i+1}(x)-\lambda_{i+1}|<\epsilon/2$$
and
$$|\overline{\delta}_i(x)-\overline{\delta}_i|<\epsilon/2, \quad |\overline{\delta}_{i+1}(x)-\overline{\delta}_{i+1}|<\epsilon/2.$$
Fix $\epsilon>0$. There exists $n_0\colon B_\epsilon\to \N$ such
that for $m$-a.e.\! $x\in B_\epsilon$ and $n>n_0(x)$, we have
\begin{itemize}
\item[(1)] $\displaystyle
\frac{\log m^{i+1}_x\left(B^\pi(x,e^{-n(\lambda_{i+1}-2\epsilon)})\right)}
 {-n(\lambda_{i+1}-2\epsilon)}\leq
 \overline{\delta}_{i+1}+\epsilon;$
\item[(2)] $\displaystyle
-\frac{1}{n} \log m^{i+1}_x(\P_{0}^n(x))\geq h_{i+1}-\epsilon$\qquad
(by (\ref{e-h}));

\item[(3)] $\displaystyle
\eta_i(x)\cap \P_0^n(x)\subset B^\pi(x,e^{-n(\lambda_{i+1}-2\epsilon)})$ \qquad (by (\ref{e-t301}));

\item[(4)] $\displaystyle
-\frac{1}{n}\log m^i_x(\P_0^n(x))\leq h_i+\epsilon$\qquad (by
(\ref{e-h})).
\end{itemize}

 Take $N_0$ such that $$
\Delta:=\{x\in B_\epsilon\colon n_0(x)\leq N_0\}$$
 has the positive measure.  By Lemma \ref{lem-2.4} and Lemma \ref{lem-2.10}, there exist $c>0$ and $\Delta'\subset \Delta$ with $m(\Delta')>0$ such that for $x\in \Delta'$, there exists $n=n(x)\geq N_0$ such that

\begin{itemize}
\item[(5)]  $\displaystyle \frac{m^{i+1}_x (L\cap \Delta)}{m^{i+1}_x(L)}\geq c$, where
$$
L:=B^\pi(x,e^{-n(\lambda_{i+1}-2\epsilon)});
$$
\item[(6)]
$\displaystyle
\frac{\log m^{i}_x\left(B^\pi(x,2e^{-n(\lambda_{i+1}-2\epsilon)})\right)} {-n(\lambda_{i+1}-2\epsilon)}> \overline{\delta}_{i}-\epsilon;$

\item[(7)] $\displaystyle\frac{\log (1/c)}{n}<\epsilon$.

\end{itemize}

Take $x\in \Delta'$ such that (1)--(7) are satisfied with $n=n(x)$.
Denote $C=\eta_{i+1}(x)$ and $C'=\eta_i(x)$. Then by (5) and (1),
$$m_x^{i+1}(L\cap \Delta )\geq cm_x^{i+1}(L)\geq
ce^{-n(\lambda_{i+1}-2\epsilon)(\overline{\delta}_{i+1}+\epsilon)}.$$
But for each $y\in L\cap \Delta$, we have by (2),
$m_y^{i+1}(\P_0^n(y))\leq e^{-n(h_{i+1}-\epsilon)}$. It follows that
the number of distinct $\P_0^n$-atoms intersecting $C\cap L\cap
\Delta$ is larger than $$m_x^{i+1}(L\cap \Delta)
e^{n(h_{i+1}-\epsilon)}.$$ However each such a $\P_0^n$-atom, say
$\P_0^n(y)$,  intersects $C'\cap L\cap \Delta$, and this together
with (3) guarantees that $C'\cap \P_0^n(y)$ is contained in $C'\cap
B^\pi(x,2e^{-n(\lambda_{i+1}-2\epsilon)})$. To see this,
let $z\in \P_0^n(y)\cap C^\prime\cap L\cap \Delta$. Since $z\in
\Delta$, we have $d(\pi z, \pi x)\leq
e^{-n(\lambda_{i+1}-2\epsilon)}$. Thus
$$\C'\cap\P_0^n(y)=\eta_i(z)\cap \P_0^n(z)\subset  B^\pi(x,e^{-n(\lambda_{i+1}-2\epsilon)})
\subset B^\pi(x,2e^{-n(\lambda_{i+1}-2\epsilon)}).$$
Meanwhile by (4), $m^i_x(\P_0^n(y))\geq e^{-n(h_i+\epsilon)}$ (for $w\in \P_0^n(y)\cap  C'\cap L$, we have $\eta_i(x)=\eta_i(w)$ and thus $m^i_x(\P_0^n(y))=m_w^i(\P_0^n(w))$). Hence we have
\begin{eqnarray*}
m^i_{x}(B^\pi(x,2e^{-n(\lambda_{i+1}-2\epsilon)})) &\geq &
 \#  \{ \mbox{$\P_0^n$-atoms intersecting } C'\cap L\cap\Delta\} \cdot e^{-n(h_i+\epsilon)}\\
 &\geq & m_x^{i+1}(L\cap \Delta)  e^{n(h_{i+1}-\epsilon)}e^{-n(h_i+\epsilon)}\\
&\geq & ce^{-n(\lambda_{i+1}-2\epsilon)
(\overline{\delta}_{i+1}+\epsilon)}e^{n(h_{i+1}-\epsilon)}e^{-n(h_i+\epsilon)}.
\end{eqnarray*}
Comparing this with (6), we have
\begin{eqnarray*}
&\mbox{}&(\lambda_{i+1}-2\epsilon)(\overline{\delta}_{i}-\epsilon)\\
&& \quad \leq (\lambda_{i+1}-2\epsilon)(\overline{\delta}_{i+1}+\epsilon)(\lambda_i-2\epsilon)+\frac{\log (1/c)}{n}+h_i-h_{i+1}+2\epsilon\\
&& \quad \leq (\lambda_{i+1}-2\epsilon)(\overline{\delta}_{i+1}+\epsilon)(\lambda_i-2\epsilon)+h_i-h_{i+1}+3\epsilon.
\end{eqnarray*}
Taking $\epsilon\to 0$ yields $h_i-h_{i+1}\geq \lambda_{i+1} (\overline{\delta}_i-\overline{\delta}_{i+1})$, which leads to a contradiction with (\ref{e-ass}).
\end{proof}

\bigskip

\begin{proof}[Proof of (C3)] Here  we give a proof by contradiction,
adopting an idea from the proof of \cite[Lemma 11.3.1]{LeYo85}.
Assume that (C3) is not true. Then there exists $0\leq i\leq k-1$
such that $
\underline{\delta}_{i+1}(x)+\vartheta_i(x)>\underline{\delta}_i(x) $
on a subset of $\Lambda_e$ with positive measure. Hence there exists
$\beta>0$ and real numbers
$\underline{\delta}_i,\underline{\delta}_{i+1},\lambda_{i}$   such
that
\begin{equation}
\label{e-as1}
\underline{\delta}_{i+1}+\vartheta_i>\underline{\delta}_i+\beta,
\end{equation}
and for any $\epsilon>0$, there exists $A_\epsilon\subset \Lambda_e$ with $m(A_\epsilon)>0$ so that for $x\in A_\epsilon$,
\begin{equation}\label{e-as2}
|\underline{\delta}_i(x)-\underline{\delta}_i|<\epsilon/2, \quad
|\underline{\delta}_{i+1}(x)-\underline{\delta}_{i+1}|<\epsilon/2,
\quad |\vartheta_i(x)-\vartheta_{i}|<\epsilon/2.
\end{equation}

Let  $0<\epsilon<\beta/4$. Find $N_1$ and a set $A_\epsilon'\subset A_\epsilon$ with $m(A_\epsilon')>0$ such that for $x\in A_\epsilon'$ and $n>N_1$,
\begin{equation}
\label{e-6.2} m_x^{i+1}\left(B^\pi(x,2e^{-n})\right)\leq
e^{-n(\underline{\delta}_{i+1}-\epsilon)}.
\end{equation}
By Lemma  \ref{lem-2.4} and Lemma \ref{lem-2.10}, we can find $c>0$ and   $A_\epsilon''\subset A_\epsilon'$ with $m(A_\epsilon'')>0$ and
$N_2$ such that for all $x\in A_\epsilon''$ and $n\geq N_2$,
$$
\frac{m_x^i(A_\epsilon'\cap B^\pi(x,e^{-n}))}{m_x^i(B^\pi(x,e^{-n}))}>c.
$$
For $x\in A_\epsilon''$ and $n\geq N_2$, we have
\begin{equation}
\label{e-6.3}
\begin{split}
m_x^i(B^\pi(x,e^{-n}))&\leq c^{-1}m_x^i(A_\epsilon'\cap B^\pi(x,e^{-n}))\\
&= c^{-1}\int m_y^{i+1}(A_\epsilon'\cap B^\pi(x,e^{-n}))\; dm^i_x(y)\\
&= c^{-1}\int_{B^{\Gamma_{i+1}}(x,e^{-n})} m_y^{i+1}(A_\epsilon'\cap B^\pi(x,e^{-n}))\; dm^i_x(y).\\
\end{split}
\end{equation}

Let $y\in \eta_i(x)$ such that $\eta_{i+1}(y)\cap A_\epsilon'\cap B^\pi(x,e^{-n})\neq \emptyset$.
Then there exists $w\in A_\epsilon'\cap B^\pi(x,e^{-n})$ such that $\pi_{i+1}y=\pi_{i+1} w$. Hence
$A_\epsilon'\cap B^\pi(x,e^{-n}) \subset  B^\pi(w,2e^{-n})$ and by (\ref{e-6.2})
\begin{eqnarray*}
m^{i+1}_y(A_\epsilon'\cap B^\pi(w,e^{-n}))&=&m^{i+1}_w(A_\epsilon'\cap B^\pi(w,e^{-n}))\\
&\leq& m^{i+1}_w(B^\pi(w,2e^{-n}))\\
&\leq& e^{-n(\underline{\delta}_{i+1}-\epsilon)}.
\end{eqnarray*}
Combining  it with (\ref{e-6.3}), we have
$$
m_x^i(B^\pi(x,e^{-n})) \leq c^{-1}e^{-n(\underline{\delta}_{i+1}-\sigma)}
m^i_x(B^{\Gamma_{i+1}}(x,e^{-n}))\qquad (x\in A_\epsilon'',\; n\geq N_2).
$$
Letting $n\to \infty$, we obtain $\underline{\delta}_i(x)\geq
\underline{\delta}_{i+1}-\epsilon+\vartheta_i(x)$ for $x\in
A_\epsilon''$. Combining it with (\ref{e-as2}) yields
$$
\underline{\delta}_i\geq
\underline{\delta}_{i+1}+\vartheta_i-4\epsilon\geq
\underline{\delta}_{i+1}+\vartheta_i-\beta,
$$
which contradicts (\ref{e-as1}).
\end{proof}

\subsection{Proof of Theorem \ref{thm-1.4}}
\begin{de}
{\rm
A  real square matrix $A$ is called {\it asymptotically similar} if all the (complex) eigenvalues of $A$ are equal in modulus.  Correspondingly, a linear transformation  $T$ on a finite-dimensional vector space $V$ is called {\it asymptotically similar} if its representation  matrix (associated with  some basis of $V$) is  asymptotically similar.
}
\end{de}

\begin{lem}
\label{lem-6.2}
Let $(A_1,\ldots, A_\ell)$ be an $\ell$-tuple of commuting linear transformations on $\R^d$. Then there are  subspaces $V_1, \ldots, V_k$ of $\R^d$ such that
\begin{itemize}
\item [(i)]
$\R^d=V_1\oplus \cdots \oplus V_k$;
  \item[(ii)]
$V_i$ is $A_{j}$-invariant for $1\leq i\leq k$ and $1\leq j\leq \ell$;
\item[(iii)]
The restriction of $A_j$ on $V_i$ is  asymptotically similar for $1\leq i\leq k$ and $1\leq j\leq \ell$.
\end{itemize}
\end{lem}
\begin{proof}
For brevity, we only prove the lemma in the case  $\ell=2$. The reader will see that the idea  works for all  cases.

Let $S, T$ be two commuting  linear transformations on $\R^d$. Let $f$ denote the real minimal polynomial of $S$. Suppose $f=f_1^{t_1}\cdots f_p^{t_p}$ is the decomposition of $f$ into powers of distinct, real irreducible monic factors $f_i$. Let $W_i$ denote the null space of $[f_i(S)]^{t_i}$, $i=1,\ldots, p$. Then $W_i$'s are $S$-invariant and $\R^d=W_1\oplus\cdots\oplus W_p$ (cf. \cite[Theorem 7.3]{Sto52}). Moreover $S_{W_i}$, the restriction of $S$ on $W_i$, is asymptotically similar.

Since $ST=TS$, $W_i$ is also $T$-invariant for each $i$.  But $T_{W_i}$ may be not asymptotically similar. However,  as above, for each $i$, we can decomposed $W_i$ into $W_i=W_{i,1}\oplus\cdots \oplus W_{i,u_i}$ such that $W_{i,j}$ are the null spaces corresponding to some factors of the minimal polynomial of $T_{W_i}$. Again,  $W_{i,j}$ is $T_{W_i}$-invariant and $S_{W_i}$-invariant. Furthermore $T_{W_{i,j}}$ and $S_{W_{i,j}}$ are asymptotically similar.
Hence  $\R^d=\bigoplus_{i,j} W_{i,j}$ is the desired  decomposition for $S$ and $T$.
\end{proof}

\noindent \begin{proof}[Proof of Theorem \ref{thm-1.4}] Let $\{S_i\}_{i=1}^\ell$ be the IFS given in the theorem. By Lemma
\ref{lem-6.2}, there is a non-singular linear transformation $Q$ on $\R^d$ such that  $\{QS_iQ^{-1}\}_{i=1}^\ell$ is
the direct product of $k$ asymptotically conformal IFS. Hence the desired result follows from Theorem \ref{thm-1.3}.
\end{proof}

\section{A variational principle about  dimensions of self-conformal sets}\label{S7}

In this section, we assume that $K$ is the attractor of a $C^1$ weakly
conformal IFS $\Phi=\{S_i\}_{i=1}^\ell$ on a compact set $X\subset \R^d$. The main result of this section
is the following variational principle.

\begin{thm}
\label{thm-7.1} Under the above setting, we have
\begin{eqnarray}\mbox{}\qquad  \dim_H K&=&\dim_BK\label{e-7.1'}\\
&=&
\sup
\left\{\dim_H \mu:\; \mu=m\circ \pi^{-1}, \;m\in
\M_\sigma(\Sigma), \; m \mbox{ is ergodic}\right\}\label{e-7.2'}\\
&=&
\max
\left\{\dim_H \mu:\; \mu=m\circ \pi^{-1}, \;m\in
\M_\sigma(\Sigma)\right\}\label{e-7.2''}\\
&=&\sup\left\{\frac{h_\pi(\sigma,m)}{\int{\lambda} \;dm}: \;m\in
\M_\sigma(\Sigma)\right\}.\label{e-7.3'} \end{eqnarray}
\end{thm}

\begin{proof}
Without  loss of generality we assume that $\overline{\dim}_B(K)>0$, where $\overline{\dim}_B$ denotes the upper box-counting dimension (cf. \cite{Fal-book}).
Let $$0<t_3<t_2<t_1< \overline{\dim}_B(K).$$
We first prove that there is an ergodic measure $m\in \M_\sigma(\Sigma)$ such that
$\dim_H m\circ \pi^{-1}\geq t_3$.  To achieve this,  let $\alpha=\frac{t_2}{t_3}-1$ and  let $r_0$ be given as in Corollary \ref{cor-4.5}.  Since $\overline{\dim}_B(K)>t_1$, for any $0<\epsilon<r_0$, there exist $r\in (0,\epsilon)$ and integer $N\geq r^{-t_1}$ such  that there are disjoint closed balls $B(z_i,r)$ ($i=1,\ldots, N$) with centers $z_i\in K$. By Corollary \ref{cor-4.5}, we can find words $w_i\in \Sigma^*$ ($i=1,\ldots, N$) such that $S_{w_i}(K)\subset B(z_i,r)$ and
\begin{equation}
\label{e-7.5}
|S_{w_i}(x)-S_{w_i}(y)|\geq r^{1+\alpha}|x-y| \qquad (x,y\in K).
\end{equation}
 This implies $r^{1+\alpha}\mbox{diam}(K)\leq  \mbox{diam}(S_{w_i}(K))\leq 2r$. According to this fact and (\ref{e-4.6}), there exist two positive constants $A,B$ (independent of $r$) such that
$$
B\log (1/r)\leq |w_i|\leq A\log (1/r)\mbox{ for all $1\leq i\leq N$}.
$$
Hence by the pigeon hole principle,  there is a subset $\J$ of $\{1,\ldots, N\}$ with cardinality $$\# \J\geq \frac{N}{(A-B)\log (1/r)+1}\geq \frac{r^{-t_1}}{(A-B)\log (1/r)+1}\geq r^{-t_2}$$
 such that the words $w_i$ ($i\in \J$)  have the same length, say $n$.

Now we adopt an argument  from the proof of \cite[Theorem 4]{Fal89}.
Let $$\delta=\min\{d(B(z_i,r), B(z_j,r)): \; i,j\in \J, i\neq j\}.$$ For any positive integer $q$ and distinct sequences $i_1,\ldots, i_q$ and $j_1,\ldots, j_q$ taking values in $\J$,
let $k$ be the least integer such that $i_k\neq j_k$. Applying (\ref{e-7.5}) $(k-1)$ times, we have
\begin{eqnarray*}
&\mbox{}& d(S_{w_{i_1}}\circ\cdots  \circ S_{w_{i_q}}(K), S_{w_{j_1}}\circ\cdots \circ S_{w_{j_q}}(K))\\
&& \quad \geq r^{(1+\alpha)(k-1)}d(B(z_{i_k},r), B(z_{j_k},r))\geq r^{q(1+\alpha)}\delta.
\end{eqnarray*}
Define a measure $\eta$ on the class of finite unions of sets
$S_{w_{i_1}}\circ\cdots  \circ S_{w_{i_q}}(K)$ by letting
$\eta(S_{w_{i_1}}\circ\cdots  \circ S_{w_{i_q}}(K))=(\#\J)^{-q}$. This extends to a measure $\eta$ on the $\sigma$-algebra generating by these sets. Let $U$ be any subset of $K$ with $\mbox{diam}(U)<\delta$ and let $q$ be the least integer such that
$$
 r^{(q+1)(1+\alpha)}\delta\leq \mbox{diam}(U)< r^{q(1+\alpha)}\delta.$$
Then $U$ intersects at most one set $S_{w_{i_1}}\circ\cdots  \circ
S_{w_{i_q}}(K)$, hence \begin{eqnarray*} \eta(U)&\leq &(\#\J)^{-q}
\leq r^{t_2q} \leq r^{-t_2}\delta^{-t_2/(1+\alpha)}
\mbox{diam}(U)^{t_2/(1+\alpha)}\\
&=&r^{-t_2}\delta^{-t_3} \mbox{diam}(U)^{t_3}. \end{eqnarray*}
 This
implies $\dim_H\eta\geq t_3$.

We point out that the measure $\eta$ constructed as above is, indeed, the projection of a $\sigma^n$-invariant and ergodic measure $\nu$ under $\pi$. Actually $\nu$ is the unique measure on $\Sigma$ satisfying
$$
\nu([w_{i_1}\ldots w_{i_q}])=(\#\J)^{-q} \qquad (q \in \N, \;  i_1,\ldots , i_q\in \J).
$$
Applying Theorem \ref{thm-1.2} to the IFS $\{S_{w_i}: \; i\in \J\}$,
We have $$\dim_H \eta=\dim_H \nu\circ \pi^{-1}=\frac{h_\pi(\sigma^n, \nu)}{-\int \log \|S_{x_1\ldots x_n}^\prime(\pi \sigma^nx)\| d\nu}.$$
Take $m=\frac{1}{n}\sum_{i=0}^{n-1} \nu\circ \sigma^{-i}$.   Then $m$ is $\sigma$-invariant and ergodic.
Applying Theorem \ref{thm-1.2} and Proposition \ref{pro-3.2}, we have
\begin{eqnarray*}
\dim_Hm\circ \pi^{-1} &=&\frac{h_\pi(\sigma, m)}{-\int \log \|S_{x_1}^\prime(\pi\sigma x)\| dm}=
\frac{h_\pi(\sigma^n, \nu)}{-\int \log \|S_{x_1\ldots x_n}^\prime(\pi\sigma^n x)\| d\nu}\\
&=&\dim_H\eta\geq t_3.
\end{eqnarray*}
Since $t_3<\overline{\dim}_BK$ is arbitrarily given, we  obtain
(\ref{e-7.1'}) and (\ref{e-7.2'}). To show  (\ref{e-7.2''}), let
$(m_i)$ be a sequence of measures in $\M_\sigma(\Sigma)$ with
$$\lim_{i\to \infty} \dim_H m_i\circ \pi^{-1}=\dim_HK.$$ Take a
sequence of positive numbers $(a_i)$ such that $\sum_{i=1}^\infty
a_i=1$. Then $m=\sum_{i=1}^\infty a_i m_i$ is an element in
$\M_\sigma(\Sigma)$ with $$\dim_H m\circ \pi^{-1}= \sup_{i}\dim_H
m_i\circ \pi^{-1}=\dim_HK.$$

To show (\ref{e-7.3'}), according to (\ref{e-7.2'}), it suffices to show that
\begin{equation}
\label{e-7.6}
\dim_Hm\circ \pi^{-1}\geq \frac{h_\pi(\sigma,m)}{-\int\log \|S^\prime_{x_1}(\pi\sigma x)\| \;dm(x)} \qquad
(m\in \M_\sigma(\Sigma)).
\end{equation}
Fix $m$ and let $\mu=m\circ \pi^{-1}$. Denote by $\Lambda$ the righthand side of (\ref{e-7.6}). By Theorem \ref{thm-1.2}, $d(\mu,z)$ exists for $\mu$-a.e.\! $z\in \R^d$. Hence to show (\ref{e-7.6}), we only need to show that for any $\epsilon>0$, there is a Borel set $E\subset \R^d$ such that  $\mu(E)>0$ and $d(\mu, z)\geq \Lambda-\epsilon$ for $z\in E$. Assume this is false. Then $d(\mu, z)< \Lambda-\epsilon$ for $\mu$-a.e.\! $z\in \R^d$. Thus by
Theorem \ref{thm-1.2} again, we have
$$
h_\pi(\sigma,m,x)<
 \lambda(x)(\Lambda-\epsilon)\qquad \mbox{for $m$-a.e.\! $x\in \Sigma$}.
$$
Taking integration w.r.t. $m$ on both sides yields
$$h_\pi(\sigma,m)<(\Lambda-\epsilon) \int \lambda\;dm,$$ which leads to a contradiction.
\end{proof}

\begin{rem}
\label{rem-4.4} {\rm Assume that $\{S_i\}_{i=1}^\ell$ is a
weakly conformal IFS which  satisfies the AWSC (see Definition
\ref{de-1.5}). Then the supremum in (\ref{e-7.2'}) and
(\ref{e-7.3'}) can be  attained  by ergodic measures.
 To see this, by Proposition \ref{pro-3.9}, the map $m\mapsto h_\pi(\sigma,m)$ is upper semi-continuous on $\M_\sigma(\Sigma)$, hence the supremum in (\ref{e-7.3'}) is  attained at some member, say $m_0$, in $\M_\sigma(\Sigma)$.
Let $m_0=\int \nu \; d{\Bbb P}(\nu)$ be the ergodic decomposition of
$m_0$. By Theorem \ref{thm-1.0}(ii), we have
$$
\dim_HK=\frac{h_\pi(\sigma,m_0)}{\int \lambda \; dm_0}=\frac{\int h_\pi(\sigma,\nu)\; d{\Bbb P}(\nu)}{\int\!\int \lambda\;d\nu\; d{\Bbb P}(\nu)}.
$$
Since $\frac{h_\pi(\sigma,\nu)}{\int \lambda\;d\nu}\leq \dim_HK$ for each $\nu$,
 the above equality implies that $\frac{h_\pi(\sigma,\nu)}{\int \lambda\;d\nu}= \dim_HK$ for ${\Bbb P}$-a.e.$\;\nu$.
Hence the supremum in (\ref{e-7.3'}) can be attained at  some ergodic measure, so do  the supremum in
(\ref{e-7.2'}).
}
\end{rem}

\section{Proof of Theorem \ref{thm-1.6}}\label{S8}

We first present some lemmas.

\begin{lem}
\label{lem-8.1}
Let $\{S_i\}_{i=1}^\ell$ be an IFS with attractor $K$.   For $n\in \N$, write $\Sigma_n=\{1,\ldots,\ell\}^n$ and denote
$$
N_n=\#\{S_u:\; u\in \Sigma_n\}.
$$
Then
\begin{itemize}
\item[(i)] $\sup\{h_\pi(\sigma,m):\; m\in \M_\sigma(\Sigma)\}\leq \frac{\log N_n}{n}$.
\item[(ii)] Let $t_n=\sup_{x\in \R^d}\#\{S_u:\; u\in \Sigma_n, x\in S_u(K)\}$. Then
$$\sup\{h_\pi(\sigma,m):\; m\in \M_\sigma(\Sigma), \mbox { $m$ is ergodic}\}\geq  \frac{\log N_n-\log t_n}{n}.$$
\end{itemize}
\end{lem}
\begin{proof}
We first show (i). Let   $n\in \N$ and $m\in \M_{\sigma}(\Sigma)$.
By the definition of $N_n$, we can construct a subset $\Omega$ of $\Sigma_n$ with $\#\Omega=N_n$ such that for any $u\in \Sigma_n$, there exists $w\in \Omega$ so that $S_u=S_w$.   Hence there is a map $g:\;\Sigma_n\to \Omega$ such that
$S_u=S_{g(u)}$ for each $u\in \Sigma_n$. Let $(\Omega^\N, T)$ denote the one-sided full shift over $\Omega$. Define $G:\; \Sigma\to \Omega^\N$ by
$$G((x_i)_{i=0}^\infty)=(w_j)_{j=1}^\infty\qquad ((x_i)_{i=1}^\infty\in \Sigma),$$
where $w_j=g(x_{(j-1)n+1}x_{(j-1)n+2}\cdots x_{jn})$. Let $\widetilde{\pi}:\ \Omega^\N\to \R^d$ denote the canonical projection w.r.t. the IFS $\{S_u:\; u\in \Omega\}$. Then by Lemma \ref{lem-3.new}(ii), we have
$$h_\pi(\sigma^n,m)=h_{\widetilde{\pi}}(T,m\circ G^{-1})\leq \log (\#\Omega)=\log N_n.$$ It follows
that $h_\pi(\sigma,m)\leq \log N_n/n$. This proves (i).

To show (ii), let $\nu$ be the Bernoulli measure on $\Omega^\N$ with
probability weight $(1/N_n, \ldots, 1/N_n)$. Then $\nu$ can be
viewed as a $\sigma^n$-invariant measure on $\Sigma$. By Lemma
\ref{lem-3.new}(ii), we have
$h_\pi(\sigma^n,\nu)=h_{\widetilde{\pi}}(T,\nu)$. Note that for
$x\in \R^d$, there are  at most $t_n$ words $u$ in $\Omega$ such
that $x\in S_u(K)$.  By Corollary \ref{cor-3.11}, we have
$$h_{\widetilde{\pi}}(T,\nu)\geq h(T,\nu)-\log t_n=\log N_n-\log
t_n.$$ Let $\mu=\frac{1}{n}\sum_{i=0}^{n-1}\nu\circ \sigma^{-i}$.
Then $\mu$ is $\sigma$-invariant and ergodic, furthermore
$$h_\pi(\sigma,\mu)=\frac{1}{n}h_\pi(\sigma^n,\nu)=\frac{1}{n}h_{\widetilde{\pi}}(T,\nu)\geq
(\log N_n-\log t_n)/n,$$
as desired.
\end{proof}

\begin{lem}
\label{lem-8.2}
Let $\Phi=\{S_i\}_{i=1}^\ell$ be an affine  IFS on $\R^d$ given by
$$S_i(x_1,\ldots,x_d)=(\rho_1 x_1,\cdots, \rho_dx_d)+(a_{i,1},\ldots, a_{i,d}),$$
where $1>\rho_1>\rho_2>\cdots >\rho_d>0$ and $a_{i,j}\in \R$. Let $K$ denote the attractor of $\Phi$, and write $\lambda_j=\log(1/\rho_j)$ for $j=1,\ldots, d$ and $\lambda_{d+1}=\infty$. View $\Phi$ as the direct product of $\Phi_1,\ldots, \Phi_d$, where $\Phi_j=\{S_{i,j}(x_j)=\rho_jx_j+ a_{i,j}\}_{i=1}^\ell$.
Let $\pi_j$ denote the canonical projection w.r.t. the IFS $\Phi_1\times\cdots \times \Phi_j$. Then
we have
\begin{equation}
\label{e-8.1}
\sum_{j=1}^d\left(\frac{1}{\lambda_j}-\frac{1}{\lambda_{j+1}}\right){H}_j
\leq
\underline{\dim}_B(K)\leq \overline{\dim}_B(K)\leq \sum_{j=1}^d\left(\frac{1}{\lambda_j}-\frac{1}{\lambda_{j+1}}\right)\widetilde{H}_j,
\end{equation}
with $H_j=\sup\left\{h_{\pi_j}(\sigma,m):\; m\in \M_\sigma(\Sigma)\right\}$ and
$$
\widetilde{H}_j=\lim_{n\to \infty} \frac{\log \#\left\{S^{(j)}_u:\; u\in \Sigma_n\right\}}{n},
$$
where $\left\{S^{(j)}_i\right\}_{i=1}^\ell$ is the IFS $\Phi_1\times\cdots \times \Phi_j$ on $\R^j$.
\end{lem}
\begin{proof}
Without loss of generality we assume that
$$S_{i}([0,1]^d)\subset [0,1]^d\qquad (i=1,\ldots,\ell).$$
For $n\in \N$, we write
$$
N_n^{(j)}=\#\{S_u^{(j)}:\ u\in \Sigma_n\} \qquad (j=1,\ldots,d),
$$
and
$$
q_d(n)=n,\quad  q_j(n)=\left[\left(\frac{\log \rho_d}{\log \rho_j}-\frac{\log \rho_d}{\log \rho_{j+1}}\right)n\right]\mbox{ for }1\leq j\leq d-1,
$$
where $[x]$ denotes the integral part of $x$.

 Construct $\Omega_{n,j}\subset \Sigma_{q_j(n)}$ ($j=1,\ldots,d$) such that
$\#\Omega_{n,j}=N_{q_j(n)}^{(j)}$ and for each $u\in \Sigma_{q_j(n)}$, there is $w\in \Omega_{n,j}$ so that
$S^{(j)}_u=S^{(j)}_w$. Then the family of following rectangles
\begin{equation}
\label{e-8.2}
\prod_{j=1}^d S_{w_dw_{d-1}\cdots w_j, j}([0,1])\qquad (w_1\in \Omega_{n,1}, \ldots, w_d\in \Omega_{n,d})
\end{equation}
is a cover of $K$. To see it, let $u_j\in \Sigma_{q_j(n)}$ ($j=1,\ldots,d$). Then we can find
 $w_j\in \Omega_{n,j}$ ($j=1,\ldots,d$) such that $S^{(j)}_{u_j}=S^{(j)}_{w_j}$. Hence
\begin{eqnarray*}
S_{u_du_{d-1}\ldots u_1}(K)&\subset& S_{u_du_{d-1}\ldots u_1}([0,1]^d)
\subset \prod_{j=1}^d S_{u_du_{d-1}\cdots u_1, j}([0,1])\\
&\subset &\prod_{j=1}^d S_{u_du_{d-1}\cdots u_j, j}([0,1])
= \prod_{j=1}^d S_{w_dw_{d-1}\cdots w_j, j}([0,1]).
\end{eqnarray*}
It follows that the family of rectangles in (\ref{e-8.2}) covers $K$. One can check that each rectangle in  (\ref{e-8.2}) is an almost $(\rho_d)^n$-cube. Hence by the definition of box-counting dimension,  we have
\begin{eqnarray*}
\overline{\dim}_BK&\leq& \limsup_{n\to \infty}\frac{\prod_{j=1}^{d}\# \Omega_{n,j}}{-\log (\rho_d)^n}=
\limsup_{n\to \infty}\frac{\prod_{j=1}^{d} N_{q_j(n)}^{(j)}}{-\log (\rho_d)^n}\\
&=& \sum_{j=1}^d\left(\frac{1}{\lambda_j}-\frac{1}{\lambda_{j+1}}\right)\widetilde{H}_j.
\end{eqnarray*}
This proves one part of (\ref{e-8.1}).

To see the other part of (\ref{e-8.1}), for $j=1,\ldots, d$, let $\Q_j$ denote the collection $\{[0,1)^j+\alpha: \;\alpha\in \Z^j\}$, and define
$$M_n^{(j)}=\#\{Q\in \Q_j:\; \mbox{diag}(\rho_1^n,\ldots, \rho_j^n)Q \cap K_j\neq \emptyset\},$$
where $K_j$ denotes the attractor of $\Phi_1\times\cdots\times
\Phi_j$. Then by Proposition \ref{pro-3.19}(ii), we have
$H_j=\lim_{n\to \infty} \frac{\log M_n^{(j)}}{n}$. We claim that for
$n\in \N$, there exists a subset $\overline{\Omega}_{n, j}\subset
\Sigma_n$ with cardinality $\geq 7^{-j}M_n^{(j)}$ such that
\begin{equation}
\label{e-8.3} S_w^{(j)}([0,1]^j)\cap S_{w'}^{(j)}([0,1]^j)=
\emptyset \mbox{ for all } w,w'\in \overline{\Omega}_{n,j} \mbox{
with } w\neq w'.
\end{equation}
To show the claim, we construct a finite subset of $\Q_j$, denoted
by $W_n^{(j)}$, such that (i) $\#W_n^{(j)}\geq 7^{-j}M_n^{(j)}$;
(ii) $\mbox{diag}(\rho_1^n,\ldots, \rho_j^n)Q \cap K_j\neq
\emptyset$ for each $Q\in W_n^{(j)}$; (iii) $2Q\cap
2\widetilde{Q}=\emptyset$ for $Q,\widetilde{Q}\in W_n^{(j)}$ with
$Q\neq \widetilde{Q}$, where $2Q:=\bigcup_{Q'\in \Q_j:\; Q'\cap
Q\neq \emptyset} Q'$.  For each $Q\in W_n^{(j)}$, since
$\mbox{diag}(\rho_1^n,\ldots, \rho_j^n)Q \cap K_j\neq \emptyset$, we
can pick a word $w(Q)\in \Sigma_n$ such that
$\mbox{diag}(\rho_1^n,\ldots, \rho_j^n)Q \cap S_{w(Q)}^{(j)}K_j\neq
\emptyset$ and hence
$$\mbox{diag}(\rho_1^n,\ldots, \rho_j^n)Q \cap S_{w(Q)}^{(j)}([0,1]^j)\neq \emptyset.$$
Denote $\overline{\Omega}_{n,j}=\{w(Q):\; Q\in W_n^{(j)}\}$. The
separation condition (iii) for the elements in $W_n^{(j)}$
guarantees  (\ref{e-8.3}).  This finishes the proof of the claim.

As above, we can construct   $\overline{\Omega}_{n,j}$ well for each $j=1,\ldots, d$ and $n\in \N$. Now fix $n$
and consider the following collection of rectangles:
\begin{equation*}
\label{e-8.4}
\prod_{j=1}^d S_{w_dw_{d-1}\cdots w_j, j}([0,1])\qquad (w_j\in \overline{\Omega}_{q_j(n),j},\; 1\leq j\leq d).
\end{equation*}
It is clear that  the above rectangles are almost $(\rho_d)^n$-cubes and each of  them intersects with $K$. Furthermore they are disjoint due to (\ref{e-8.3}). Hence by the definition of box-counting dimension, we have
\begin{eqnarray*}
\underline{\dim}_B(K)&\geq& \liminf_{n\to \infty} \frac{\prod_{j=1}^d \#\overline{\Omega}_{q_j(n),j}}{-\log (\rho_d)^n}\geq \liminf_{n\to \infty} \frac{\prod_{j=1}^d 7^{-j}M_{q_j(n)}^{(j)}}{-\log (\rho_d)^n}\\
&=&\sum_{j=1}^d\left(\frac{1}{\lambda_j}-\frac{1}{\lambda_{j+1}}\right){H}_j.
\end{eqnarray*}
This finishes the proof of (\ref{e-8.1}).
\end{proof}

\begin{proof}[Proof of Theorem \ref{thm-1.6}] We divide the proof into two steps:

\noindent {\em Step 1. Show the variational principle for $\dim_HK$}.

We first give an upper bound for $\dim_HK$.  Fix $n\in \N$. Define
$$N_j=\#\{S^{(j)}_u:\; u\in \Sigma_n\}\qquad (j=1,\ldots,d),$$
where $\{S^{(j)}_i\}_{i=1}^\ell$ denotes the IFS $\Phi_1\times\cdots\times \Phi_j$. Then we can construct $$\Omega_j\subset \Sigma_n\quad   (j=d,d-1,\ldots,1)$$  such that  $\#\Omega_j=N_j$,
$\Sigma_n\supset \Omega_d\supset \Omega_{d-1}\supset\cdots \supset \Omega_1$ and furthermore, for each $u\in \Sigma_n$ and $1\leq j\leq d$, there is $w_j\in \Omega_j$ such that $S_u^{(j)}=S^{(j)}_{w_j}$. Hence there are natural maps
$\theta_d,\theta_{d-1},\ldots,\theta_1$ with
$$
\Sigma_n\stackrel{\theta_d}{\longrightarrow}\Omega_d \stackrel{\theta_{d-1}}{\longrightarrow} \Omega_{d-1}\stackrel{\theta_{d-2}}{\longrightarrow}\cdots\stackrel{\theta_2}{\longrightarrow}\Omega_2
\stackrel{\theta_1}{\longrightarrow}\Omega_1
$$
such that $S_u^{(j)}=S_{\theta_j(u)}^{(j)}$ for any $1\leq j\leq d$ and $u\in \Omega_{j+1}$, with convention $\Omega_{d+1}=\Sigma_n$.

Let $Z_d:\Omega_d\to \R$ be the indicator of $\Omega_d$, i.e., $Z_d(u)=1$ for all $u\in \Omega_d$. Define
$$
Z_{d-1}(w)=\sum_{u\in \theta_{d-1}^{-1}(w)} Z_{d}(u)\qquad (w\in
\Omega_{d-1}).
$$
Define inductively
$$
Z_{j}(w)=\sum_{u\in \theta_{j}^{-1}(w)} Z_{j+1}(u)^{\frac{\log \rho_{j+1}}{\log \rho_{j+2}}}\qquad (w\in \Omega_{j}, \; j=d-2,\ldots, 1).
$$
In particular, define
$$
Z_0=\sum_{u\in \Omega_1} Z_{1}(u)^{\frac{ \log \rho_1}{\log \rho_{2}}}.
$$
Using the technique  by Kenyon \& Peres \cite{KePe96} (which is an extension of McMullen \cite{McM84}), we have
\begin{equation}
\label{e-8.5}
\dim_HK\leq \frac{\log Z_0}{-n\log \rho_1}.
\end{equation}
More precisely,  define a probability vector $\left(p(u)\right)_{u\in \Omega_d}$ by
$$
p(u)=\frac{Z_d(u)}{Z_{d-1}(\theta_{d-1}(u))}\cdot
\prod_{j=1}^{d-1}
\frac{
Z_{j}(\theta_j\theta_{j+1}\cdots \theta_{d-1}(u))^
{
\frac{\log \rho_{j}}{\log \rho_{j+1}}
}
}
{
Z_{j-1}(\theta_{j-1}\theta_j\cdots \theta_{d-1}(u))
}
$$
with convention $Z_0(\theta_0\ldots \theta_{d-1}(u))=Z_0$ for any $u\in \Omega_{d}$.  Let $\nu$ be the product measure on $(\Omega_d)^\N$ by assigning probability $p(u)$ to each digit $u\in \Omega_d$. The measure $\nu$ can be viewed as a measure on $\Sigma$, which is $\sigma^n$-invariant and ergodic.  Let $\mu=\nu\circ \pi^{-1}$. Then
\begin{equation}
\label{e-8.6}
\liminf_{r\to 0}\frac{\log\mu(B(\pi x, r))}{\log r}\leq \frac{\log Z_0}{-n\log \rho_1}\qquad (x\in \Sigma).
\end{equation}
A detailed proof of (\ref{e-8.6}) was given by Shmerkin (see the
proof of (4.3) in \cite{Shm06}) for the case $d=2$, whilst a slight
modification of the proof of  \cite[Theorem 1.2]{KePe96} provides a
proof of (\ref{e-8.6}) for  $d\geq 2$. Then (\ref{e-8.5}) follows
from (\ref{e-8.6}) and Billingsley's lemma.

Now we want to indicate certain connection between the upper bound  $\frac{\log Z_0}{-n\log \rho_1}$ and the projection entropies.  First we define the projections  $\theta_{j}^*: \; \Omega_{j+1}^\N\to \Omega_j^\N$ ($j=d-1,\ldots,1$) by
$$\theta_{j}^*\left((u_k)_{k=1}^\infty\right)=\left(\theta_{j}(u_k)\right)_{k=1}^\infty\qquad (\left(u_k\right)_{k=1}^\infty\in \Omega_{j+1}^\N).$$
Then it is easy to see that for each $1\leq j\leq d-1$,
the measure
$$\nu_j:=\nu\circ \left(\theta_j^*\circ \theta_{j+1}^*\circ \cdots\circ \theta_{d-1}^*\right)^{-1}$$
is a product measure on $\Omega_j^\N$. Let $T_j$ denote the left shift operator on $\Omega_j^\N$.  By a direct calculation, we have
\begin{equation*}
\label{e-8.7}
\frac{\log Z_0}{-n\log \rho_1}=\sum_{j=1}^d\left(\frac{1}{\lambda_j}-\frac{1}{\lambda_{j+1}}\right)
\frac{h(T_j,\nu_j)}{n}.
\end{equation*}
Thus we have
\begin{equation}
\label{e-8.8}
\dim_HK\leq \sum_{j=1}^d\left(\frac{1}{\lambda_j}-\frac{1}{\lambda_{j+1}}\right)
\frac{h(T_j,\nu_j)}{n}.
\end{equation}

Let $\widetilde{\pi}_j$ ($j=1,\ldots, d$) denote the canonical
projection from $\Omega_j^\N$ to $\R^j$ w.r.t. the IFS
$\{S^{(j)}_u\}_{u\in \Omega_j}$ ( remember that  $\pi_j$ denotes the
canonical projection from $\Sigma$ to $\R^d$ w.r.t. $\{S^{(j)}_u:\;
u\in \Sigma_n\}$). According to Lemma \ref{lem-3.new}(ii), we have
\begin{equation}
\label{e-8.9}
h_{\widetilde{\pi}_j}(T_j,\nu_j)=h_{\pi_j}(\sigma^n,\nu) \qquad (j=1,\ldots,d).
\end{equation}
Since  $\Phi_1\times \cdots \times \Phi_j$ ($j=1,\ldots,d$) satisfy
the AWSC,  there is a sequence $(t_n)$ of positive integers  with
$\lim_n\log t_n/n=0$, such that
\begin{equation}
\label{e-AA} \sup_{x\in \R^j}\#\{S_{u}^{(j)}:\; u\in \Omega_j, \;
x\in S_{u}^{(j)}(K_j)\}\leq t_n \qquad (j=1,\ldots, d),
\end{equation}
 where $K_j$ denotes the attractor of $\Phi_1\times\cdots \times \Phi_j$. By Corollary \ref{cor-3.11}, we have
$$
h_{\widetilde{\pi}_j}(T_j,\nu_j)\geq h(T_j,\nu_j)-\log t_n\geq
h(T_j,\nu_j)-\log t_n.
$$
It together with (\ref{e-8.9}) yields $h_{\pi_j}(\sigma^n,\nu)\geq
h(T_j,\nu_j)-\log t_n$. Now applying Theorem \ref{thm-1.3} to the
IFS $\{S_u:u\in \Sigma_n\}$, we have
\begin{eqnarray*}
\dim_H\nu\circ \pi^{-1}&=&\frac{1}{n}\sum_{j=1}^d \left(\frac{1}{\lambda_j}-\frac{1}{\lambda_{j+1}}\right)
h_{\pi_j}(\sigma^n,\nu)\\
&\geq& \frac{1}{n}\sum_{j=1}^d \left(\frac{1}{\lambda_j}-\frac{1}{\lambda_{j+1}}\right)
(h(T_j,\nu_j)-\log t_n)\\
&\geq& \dim_HK-\frac{\log t_n}{n}\cdot\sum_{j=1}^d
\left(\frac{1}{\lambda_j}-\frac{1}{\lambda_{j+1}}\right)\qquad
(\mbox{by (\ref{e-8.8})}).
\end{eqnarray*}
Let $m=\frac{1}{n}\sum_{i=1}^n\nu\circ \sigma^{-i}$.
Then $m$ is ergodic and $\dim_Hm\circ \pi^{-1}=\dim_H\nu\circ\pi^{-1}$.
Letting $n$ tend to $\infty$, we obtain
\begin{equation}
\label{e-8.10}
\sup
\{
\dim_H m\circ \pi^{-1}:\; m\in \M_\sigma(\Sigma), \mbox { $m$ is ergodic }\}\geq\dim_HK.
\end{equation}
It is clear the ``$\geq$'' in above inequality can be replaced by
``$=$'' since $m\circ \pi^{-1}$ is supported on $K$. Note that
$h_{\pi_j}(\sigma, \cdot)$ ($j=1,\ldots, d$) are upper
semi-continuous on $\M_\sigma(\Sigma)$ (see Proposition
\ref{pro-3.9} and (\ref{e-AA})). By Theorem \ref{thm-1.0}(ii) and
Theorem \ref{thm-1.3}, we see that the supremum in (\ref{e-8.10}) is
attained at some ergodic element in $\M_\sigma(\Sigma)$. This
finishes the proof of the variational principle for $\dim_HK$.

\bigskip
\noindent {\em Step 2. Show the variational principle for $\dim_BK$}.

By Lemma \ref{lem-8.2}, we only need to show that under the assumption of Theorem \ref{thm-1.6},
\begin{equation}
\label{e-8.12}
H_j\geq \widetilde{H}_j \qquad (j=1,\ldots, d),
\end{equation}
where
$$H_j=\sup\{h_{\pi_j}(\sigma,m):\; m\in \M_\sigma(\Sigma)\},\quad
\widetilde{H}_j=\lim_{n\to \infty} \frac{\log \#\{S^{(j)}_u:\; u\in \Sigma_n\}}{n}.
$$
To see (\ref{e-8.12}), by (\ref{e-AA}) and Lemma \ref{lem-8.1}, we have
$$
H_j\geq \frac{\log \#\left\{S^{(j)}_u:\; u\in \Sigma_n\right\}-\log
t_n}{n} \qquad (n\in \N).
$$
Letting $n\to \infty$, we obtain (\ref{e-8.12}) by the assumption $\log t_n/n\to 0$. This finishes the proof of the theorem.
\end{proof}

\begin{rem}
\label{rem-9.3}
   {\rm With an essentially identical proof, Theorem \ref{thm-1.6} can be  extended to the following class of IFS $\Phi=\Phi_1\times \cdots \times \Phi_k$ on $\R^{q_1}\times\cdots\times\R^{q_k}$, where $\Phi_j$ has the form $\{A_j z_j+ c_{i,j}\}_{i=1}^\ell$ such that $A_j$ is the inverse of an integral matrix and all the eigenvalues of $A_j$ equals $\rho_j$ in modulus, $\rho_1>\cdots>\rho_k$,   $c_{i,j}\in \Q^{q_j}$.
   }
       \end{rem}

    This together with Lemma \ref{lem-6.2} and the proof of Theorem \ref{thm-1.4} yields
\begin{thm}
\label{thm-9.4}
Let $\Phi=\{S_i\}_{i=1}^\ell$ be an IFS on $\R^d$ of the form
$$
S_i(x)=Ax+c_i\qquad (i=1,\ldots,\ell),
$$
where $A$ is the inverse of an integral expanding $d\times d$ matrix, $c_i\in \Z^d$. Let $K$ be the attractor of the IFS. Then there is an ergodic measure on $K$ of full Hausdorff dimension.
\end{thm}

\section{A final remark about infinite non-contractive IFS}
\label{S-10}
In the previous sections, we have made the restriction that an IFS consists of finitely many contractive maps.
We remark that part of our results  can be extended to certain infinite non-contractive IFS.

Let $\Phi=\{S_i\}_{i=1}^\infty$ be a family of maps on $\R^d$ of the form
$$
S_i(x)=\rho_iR_i(x)+a_i\qquad(i=1,2,\ldots),
$$
where $\rho_i>0$, $R_i$ are orthogonal $d\times d$ matrices, $a_i\in \R^d$.

Let $(X,\sigma)$ be the left shift over the alphabet $\{i:\; i\in \N\}$, and let $m$ be an ergodic measure on $X$ satisfying  $H_m(\P_\infty)<\infty$, where $\P_\infty$ denotes the partition of $X$ given by
$$
\P_\infty=\{[i]:\; i\in \N\},
$$
where $[i]=\{(x_i)_{i=1}^\infty\in X:\; x_1=i\}$. Assume that $\Phi$ is {\it $m$-contractive} in the sense that
\begin{equation}
\label{e-10.1}
\sum_{i=1}^\infty (\log \rho_i)  m([i])<0,\qquad \sum_{i=1}^\infty (\log |a_i|) m([i])<\infty.
\end{equation}
Denote $$
\lambda=-\sum_{i=1}^\infty (\log \rho_i)  m([i]).
$$

Let $X'$ denote the set of points $x=(x_i)_{i=1}^\infty\in X$ such that
$$
\lim_{n\to \infty} (1/n) \log (\rho_{x_1}\rho_{x_2}\ldots \rho_{x_n})=-\lambda, \quad \lim_{n\to \infty} (1/n) \log |a_{x_n}|=0.
$$
Then $X'$ satisfies   $\sigma^{-1}(X')=X'$. Furthermore by Birkhoff's ergodic theorem,   $$m(X')=1.$$
Define the projection map $\pi:\; X'\to \R^d$ by
$$
\pi(x)=\lim_{n\to \infty} S_{x_1}\circ S_{x_2}\circ \cdots\circ S_{x_n}(0) \qquad (x\in X').
$$
It is easily checked that $\pi$ is well defined.  Let $\mu=m\circ \pi^{-1}$ be the projection of $m$ under $\pi$.
We have the following theorem

\begin{thm}
\label{thm-10.1}
Under the above setting,  $\mu=m\circ \pi^{-1}$ is exactly dimensional and
$$\dim_H\mu=\frac{h_\pi(\sigma,m)}{\lambda},$$
where $H_\pi(\sigma,m)=H_m(\P_\infty|\sigma^{-1}\pi^{-1}\gamma)-H_m(\P_\infty|\pi^{-1}\gamma)$, $\gamma=\B(\R^d).$
\end{thm}
We remark that when $m$ is a Bernoulli product measure, $\mu=m\circ \pi^{-1}$ is the stationary measure of certain affine random walk determined by $\Phi$ and $m$, and  the decay property of $\mu$ at infinity has been extensively studied in the literature (cf. \cite{GuLe08} and references therein).

The proof of  Theorem \ref{thm-10.1} is essentially identical to that given in Section~\ref{S5}.
Indeed we only need to replace $\Sigma$ in Section~\ref{S5} by $X'$, and replace `let $c>1$ so that $c\sup_{x\in \Sigma}\overline{\rho}(x)<1$' in the proof of Theorem \ref{thm-5.2} by `let $1<c<e^{\lambda}$'.

%%      ---------------------------------------------------------------------
%%      ------------------------- APPENDIX (OPTIONAL) -----------------------
%%      ---------------------------------------------------------------------

%%      If you have one appendix, uncomment the line \appendix and add
%%      a \section{ *** APPENDIX TITLE ***}. If you have more than
%%      one, uncomment the line \appendices and add a \section{ ***
%%      APPENDIX TITLE ***} command for each appendix title.

%\appendix
%\appendices
%\section{}

%%      Type body of appendix/-ices here.

%%      ---------------------------------------------------------------------
%%      ---------------------------ACKNOWLEDGMENTS (OPTIONAL) ---------------
%%      ---------------------------------------------------------------------

%% ***** UNCOMMENT THE FOLLOWING LINE TO ADD ACKNOWLEDGMENTS.

 {\bf Acknowledgement}.
 The authors are grateful to Fran\c{c}ois Ledrappier for his encouragement and helpful comments. They  are indebted to  Eric Olivier for stimulating discussions of the variational principle about the Hausdorff dimension of self-affine sets, and to Wen Huang for  the  discussions of the entropy theory. They also thank Quansheng Liu, Emile Le Page and Yong-Luo Cao for valuable comments, and  Guo-Hua Zhang for critical reading of the manuscript. The first author was partially supported by the RGC grant in CUHK, Fok Ying Tong Education Foundation and NSFC (Grant 10571100). The second author was partially supported by NSF under grants DMS-0240097 and DMS-0503870.

%%      Type acknowledgments here.

%%      ---------------------------------------------------------------------
%%      --------------------------- BIBLIOGRAPHY ----------------------------
%%      ---------------------------------------------------------------------

\frenchspacing
\bibliographystyle{plain}

\begin{thebibliography}{99}




\bibitem{Bar07} Bara\'{n}ski, K.
Hausdorff dimension of the limit sets of some planar geometric constructions.
{\it Adv. Math.} {\bf 210} (2007),  215--245.

\bibitem{Bar-book}  Barnsley, M. {\it Fractals everywhere}. Academic Press, Inc., Boston, MA, 1988.

\bibitem{BaMe07}
Barral, J.  and  Mensi, M.  Gibbs measures on self-affine Sierpi\'{n}ski carpets and their singularity spectrum.
{\it Ergodic Theory Dynam. Systems} {\bf 27} (2007),  1419--1443.

\bibitem{BPS99}
Barreira, L.,   Pesin, Ya.,  and  Schmeling, J.
Dimension and product structure of hyperbolic measures.
{\it Ann. of Math.} {\bf 149} (1999),  755--783.

\bibitem{Bed84}Bedford, T.
{\it Crinkly curves, Markov partitions and box dimension in self-similar sets}, Ph.D. Thesis, University of Warwick, 1984.

\bibitem{Bed91}Bedford, T.
Applications of dynamical systems theory to fractals---a study of cookie-cutter Cantor sets. {\it Fractal geometry and analysis} (Montreal, PQ, 1989), 1--44,  Kluwer Acad. Publ., Dordrecht, 1991.

\bibitem{Bil95}
Billingsley, P.  {\it Probability and measure}. Third edition.  John Wiley \& Sons, Inc., New York, 1995.

\bibitem{Bog92} Bogensch\"{u}tz, T. Entropy, pressure, and a variational principle for random dynamical
systems, {\it Random Comput. Dynam.} {\bf 1} (1992/93), 99--116.


\bibitem{Bow75} Bowen, R.  Equilibrium states and the ergodic theory of Anosov diffeomorphisms, {\it Lecture notes in Math.,} No. {\bf  470}, Springer-Verlag, 1975.

\bibitem{EcRu85}
Eckmann, J. P.  and Ruelle, D.  Ergodic theory of chaos and strange attractors. {\it Rev. Modern Phys.} {\bf 57} (1985),
617--656.

\bibitem{Fal88}
Falconer, K. J. The Hausdorff dimension of self-affine fractals.
{\it Math. Proc. Cambridge Philos. Soc.} {\bf 103} (1988),
339--350

\bibitem{Fal89} Falconer, K. J. Dimensions and measures of quasi self-similar sets.
{\it Proc. Amer. Math. Soc.} {\bf 106} (1989), 543--554.


\bibitem{Fal-book}
Falconer, K. J. {\it Fractal geometry,
mathematical foundations and applications}.
Wiley, 1990.

%\bibitem{Fal97} Falconer, K. J. {\it Techniques in fractal geometry}. John Wiley \& Sons, Ltd., Chichester, 1997.


\bibitem{Fan94}Fan, A. H.  Sur les dimensions de mesures. {\it Studia Math.} {\bf 111} (1994),  1--17.

%\bibitem{FaLa99}
%A. H. Fan and K. S. Lau, Iterated function system and Ruelle operator. {\it J. Math. Anal. Appl.} {\bf 231} (1999),  319--344.



\bibitem{FLR02}
Fan, A. H.,  Lau, K. S.,  and Rao, H. Relationships between different dimensions of a measure. {\it Monatsh. Math.}
{\bf  135} (2002),  191--201.

%\bibitem{FOY83}
%J. Farmer, E. Ott and J. Yorke,  The dimension of chaotic attractors. {\it Phys. D} {\bf 7} (1983),  153--180.

\bibitem{Fen03} Feng, D. J.
The smoothness of $L^q$-spectrum of self-similar measures with
overlaps.
{\it J. Lond. Math. Soc. } {\bf 68} (2003), 102--118.

\bibitem{Fen05}
Feng, D. J.  The limited Rademacher functions and Bernoulli convolutions associated with Pisot numbers. {\it Adv. Math.} {\bf 195} (2005),  24--101.

\bibitem{Fen07}Feng, D. J. Gibbs properties of self-conformal measures and the multifractal formalism.
{\it Ergodic Theory Dynam. Systems} {\bf 27} (2007), 787--812.

\bibitem{FeWa05} Feng, D. J.  and   Wang, Y. A class of self-affine sets and self-affine measures.
{\it J. Fourier Anal. Appl.} {\bf 11} (2005), 107--124.


\bibitem{GaLa92}   Gatzouras, D. and Lalley, S. P.
Hausdorff and box dimensions of certain self-affine fractals.
{\it Indiana Univ. Math. J.}  {\bf 41}  (1992),   533--568.

\bibitem{GaPe97}  Gatzouras, D. and  Peres, Y.
Invariant measures of full dimension for some expanding maps.
{\it Ergodic Theory Dynam. Systems} {\bf 17} (1997), 147--167.


\bibitem{GaPe96}  Gatzouras, D. and  Peres, Y. The variational principle for Hausdorff dimension: a survey. {\it Ergodic theory of $\Z\sp d$ actions} (Warwick, 1993--1994),  London Math. Soc. Lecture Note Ser. {\bf 228}, 113--125, Cambridge Univ. Press, Cambridge, 1996.

\bibitem{GeHa89} Geronimo, J. S. and   Hardin, D. P.  An exact formula for the measure dimensions associated with a class of piecewise linear maps. Fractal approximation. {\it Constr. Approx.} {\bf 5} (1989),  89--98.

\bibitem{GuLe08} Guivarc'h, Y. and  Le Page, E.  On spectral properties of a family of transfer operators and convergence to stable laws for affine random walks. {\it Ergodic Theory Dynam. Systems} {\bf 28} (2008),  423--446.

 \bibitem{Hu96} Hu, H.  Dimensions of invariant sets of expanding maps. {\it Comm. Math. Phys.} {\bf 176} (1996),  307--320.

\bibitem{HYZ06} Huang, W.,   Ye, X. D., and Zhang, G. H.  A local variational principle for conditional entropy.
{\it Ergodic Theory Dynam. Systems} {\bf 26} (2006),  219--245.

\bibitem{HuLa95}
Hueter I. and Lalley, S. P.
 Falconer's formula for the Hausdorff dimension of a self-affine set
 in $\R\sp 2$.
{\it Ergodic Theory Dynam. Systems} {\bf 15} (1995),  77--97.


\bibitem{Hut81}   Hutchinson, J. E. Fractals and self-similarity. {\it %
Indiana Univ. Math. J.} {\bf 30} (1981), 713--747.

\bibitem{JPS07} Jordan, T.,  Pollicott M.,  and  Simon, K.  Hausdorff dimension for randomly perturbed self affine attractors. {\it Comm. Math. Phys.} {\bf 270} (2007),  519--544.

\bibitem{Kae04} K\"{a}enm\"{a}ki, A.
On natural invariant measures on generalised iterated function systems.
{\it Ann. Acad. Sci. Fenn. Math.} {\bf 29} (2004), 419--458.

\bibitem{KaSh08}K\"{a}enm\"{a}ki, A.  and  Shmerkin, P. Overlapping self-affine sets of Kakeya type.
{\it Ergodic Theory Dynam. Systems}. At press.

\bibitem{Kel-book} Keller, G.  {\it Equilibrium states in ergodic theory}.  Cambridge University Press, Cambridge, 1998.

\bibitem{KePe96} Kenyon R. and Peres, Y.  Measures of full dimension on affine-invariant sets.
{\it Ergodic Theory Dynam. Systems} {\bf 16} (1996), 307--323.

\bibitem{KePe96a}  Kenyon R. and Peres, Y.
Hausdorff dimensions of sofic affine-invariant sets.
{\it Israel J. Math.}  {\bf 94}  (1996), 157--178.


%\bibitem{LRY01}K.S. Lau, H. Rao and Y.L. Ye, Corrigendum: ``Iterated function system and Ruelle operator'' [J. Math. %Anal. Appl. 231 (1999), 319--344] by Lau and A. H. Fan.,  {\it J. Math. Anal. Appl.} {\bf 262} (2001), 446--451.

\bibitem{Lal97} Lalley, S. P.
$\beta$-expansions with deleted digits for Pisot numbers $\beta$.
{\it Trans. Amer. Math. Soc.} {\bf 349} (1997),  4355--4365.

\bibitem{Lal98} Lalley, S. P. Random series in powers of algebraic integers: Hausdorff dimension of the limit distribution. {\it J. London Math. Soc.} {\bf 57} (1998), 629--654.


\bibitem{LaNg99}
 Lau, K. S. and Ngai, S. M.
Multifractal measure and a
weak separation condition.
{\it Adv. Math.} {\bf 141} (1999), 45--96.

\bibitem{Led92}  Ledrappier, F.  ``On the dimension of some graphs'' in {\it  Symbolic dynamics and its applications} (New Haven, CT, 1991),  Contemp. Math. {\bf 135}, Amer. Math. Soc., Providence, RI, 1992, 285--293.

\bibitem{LePo94}
Ledrappier, F. and  Porzio, A.
A dimension formula for
Bernoulli convolutions.
{\it J. Stat. Phys.} {\bf 76} (1994), 1307--1327.

\bibitem{LeYo85}  Ledrappier, F. and  Young, L.-S. The metric entropy of diffeomorphisms. I. Characterization of measures satisfying Pesin's entropy formula. II. Relations between entropy, exponents and dimension. {\it Ann. of Math.}  {\bf 122} (1985),  509--539;  540--574.


\bibitem{Luz06}   Luzia, N. A variational principle for the dimension for a class of non-conformal repellers.
{\it Ergodic Theory Dynam. Systems} {\bf 26} (2006),  821--845.

%\bibitem{Man81} Manning, A.   A relation between Lyapunov exponents, Hausdorff dimension and entropy.
%{\it Ergodic Theory Dynam. Systems}  {\bf 1} (1981),  451--459.

\bibitem{Man87}  Ma\~{n}\'{e}, R. {\it Ergodic theory and differentiable dynamics}.  Springer-Verlag, Berlin, 1987.



\bibitem{Mat95} Mattila, P.  {\it Geometry of Sets and Measures in Euclidean
Spaces}, Cambridge University Press, 1995.

 \bibitem{McM84} McMullen, C. The Hausdorff dimension of general Sierpi\'nski carpets. {\it Nagoya Math. J.} {\bf 96} (1984), 1--9.

\bibitem{NgWa01}  Ngai, S. M. and    Wang, Y.
Hausdorff dimension of self-similar sets with overlaps.
{\it J. London Math. Soc. }   {\bf 63}  (2001),  655--672.

\bibitem{Oli08} Olivier, E. Variational principle for dimension and the uniqueness of the measure with full dimension on (mod $1$) Sierpi\'{n}ski carpets. Preprint.
\bibitem{Pat97} Patzschke, N.  Self-conformal multifractal measures. {\it Adv. in Appl. Math.} {\bf 19} (1997),  486--513.

\bibitem{Par-book} Parry, W.  {\it Topics in ergodic theory}, Cambridge University Press, 1981.

\bibitem{Par-book1}  Parry, W. {\it Entropy and generators in ergodic theory}.
W. A. Benjamin, Inc., New York-Amsterdam, 1969.

%\bibitem{Pat97} N. Patzschke,
%Self-conformal multifractal measures. {\it Adv. in Appl. Math.} {\bf 19} (1997), 486--513.

\bibitem{PrUr89}
 Przytycki, F. and  Urba\'{n}ski, M.
On the Hausdorff
dimension of some fractal sets.
{\it Studia Math.} {\bf 93} (1989), 155--186.

%\bibitem{PUZ89}   Przytycki, F.,  Urba\'nski, M.  and  Zdunik, A.
%Harmonic, Gibbs and Hausdorff measures on repellers for holomorphic maps. I.
%{\it Ann. of Math.} {\bf 130} (1989), 1--40.

%\bibitem{PRSS}Y. Peres, M. Rams, K. Simon and B. Solomyak,
% Equivalence of positive Hausdorff measure and the open set condition for self-conformal sets. {\it Proc. Amer. Math. Soc.} {\bf 129} (2001), 2689--2699.

%\bibitem{PeSo00}  Y. Peres, B. Solomyak, Existence of $L\sp q$ dimensions
%and entropy dimension for self-conformal measures. {\it Indiana Univ. Math.
%J.} {\bf 49} (2000), 1603--1621.

\bibitem{Pes-book}
 Pesin, Ya.
{\it Dimension theory in dynamical systems:
contemporary views and applications}.
The University of Chicago Press, 1997.

 \bibitem{PeSo00}  Peres, Y. and   Solomyak, B. Existence of $L\sp q$ dimensions and entropy dimension for self-conformal measures. {\it Indiana Univ. Math. J.} {\bf 49} (2000),  1603--1621.

\bibitem{QiXi08} Qian, M. and  Xie, J.-S. Entropy formula for endomorphisms: relations between entropy, exponents and dimension. {\it Discrete Contin. Dyn. Syst.} {\bf 21} (2008),  367--392.

\bibitem{RaWe98}  Rao, H. and   Wen, Z. Y. A class of self-similar fractals with overlap structure.
{\it Adv. in Appl. Math.} {\bf 20} (1998), 50--72.

\bibitem{Roh49}  Rohlin, V. A. On the fundamental ideas of measure theory. {\it Mat. Sbornik N.S.} 25 (67) (1949), 107--150. (see also {\it Amer. Math. Soc. Translation 1952}, (1952). no. 71.)


\bibitem{Rud-book}
 Rudin, W.  {\it Real and complex analysis}. Third edition. McGraw-Hill Book Co., New York, 1987.


\bibitem{Sal63}  Salem, R. {\it Algebraic numbers and Fourier analysis}. D. C.
Heath and Co., Boston,  1963.

\bibitem{Shm06}
 Shmerkin, P.  Overlapping self-affine sets. {\it Indiana Univ. Math. J.} {\bf 55} (2006),  1291--1331.

\bibitem{Sch98a}  Schmeling, J. A dimension formula for endomorphisms---the Belykh family.
{\it Ergodic Theory Dynam. Systems} {\bf  18} (1998),  1283--1309.

\bibitem{Sch98b}  Schmeling, J. and  Troubetzkoy, S. Dimension and invertibility of hyperbolic endomorphisms with singularities. {\it Ergodic Theory Dynam. Systems} {\bf 18} (1998),  1257--1282.

\bibitem{Sol98}
  Solomyak, B.
Measure and dimension for some fractal families. {\it Math. Proc.
Cambridge Philos. Soc.} {\bf  124}
 (1998),  531--546.

\bibitem{Sto52}Stoll, R. R.  {\it Linear algebra and matrix theory}. McGraw-Hill Company, Inc., New York--Toronto--London, 1952.

\bibitem{Wal-book} Walters, P. {\it An introduction to ergodic theory}.
    Springer-Verlag, 1982.

\bibitem{Yam67}  Yamamoto, T.  On the extreme values of the roots of matrices.
{\it J. Math. Soc. Japan} {\bf  19} (1967), 173--178.

\bibitem{You82}
 Young, L.-S. Dimension, entropy and Lyapunov exponents.
{\it Ergodic Theory Dynam. Systems} {\bf 2} (1982), 109--124.

%\bibitem{You95}L.-S. Young, Ergodic theory of attractors, {\it Proc. ICM} (Z\"{u}rich, 1994),
%Birkh\"{a}ser, Basel, 1995, 1230--1237.





\end{thebibliography}

%%      For each reference, provide the following information:

% \bibitem{ *** LABEL *** }             %% Give a reference label.
% * Name(s) of Author(s) *              %% Enter author(s) names.
% EXAMPLE:  Gray, M., Black, F., and White, A.

%%      Use the following template for a journal article:
% * Title of article *.                 %% Example: Existence and uniqueness.
% \textit{* Abbreviated journal name *} %% Example: \textit{Comm. Pure Appl. Math.}
% \textbf{* Volume number *}            %% Example: \textbf{72}
% (* Year of publication *),            %% Example: (1993),
% * Issue number [optional],            %% Example: no. 6,
% * Page range *.                       %% Example: 675--690.

%%      Use the following template for a book:
% \textit{* Title of book *}.           %% Example: \textit{Ancient Topology}.
% * Publisher *,                        %% Example: Wiley-Interscience,
% * City of publisher *,                %% Example: New York,
% * Year of publication *.              %% Example: 1993.

%%      ---------------------------------------------------------------------
%%      ------------------------ CONTACT INFORMATION ------------------------
%%      ---------------------------------------------------------------------

%      Place contact information for each author between
%      the \begin{comment} and \end{comment} commands. Include
%      preferred mailing address and e-mail addresses. Please note
%      that these will not print. We will format them for printing
%      during the editing stage.

\end{document}